\documentclass{article}
\usepackage{amsthm}
\usepackage{amsmath}
\usepackage{setspace}
\usepackage{amsfonts}
\usepackage{mathtools}
\usepackage{wasysym}
\DeclareMathAlphabet{\mathbbb}{U}{bbold}{m}{n}

\usepackage{enumitem}
\usepackage{cite}

\usepackage[depth=subsection]{bookmark}

\theoremstyle{plain}
\newtheorem{prop}{Proposition}[section]
\newtheorem{theorem}[prop]{Theorem}
\newtheorem{maintheorem}[prop]{Main Theorem}
\newtheorem{corollary}[prop]{Corollary}

\theoremstyle{definition}
\newtheorem{df}[prop]{Definition}
\newtheorem{exa}[prop]{Example}
\newtheorem{lemma}[prop]{Lemma}
\newtheorem{remark}[prop]{Remark}

\newtheorem{construction}[prop]{Construction}
\newtheorem{recoll}[prop]{Recollection}
\newtheorem{notation}[prop]{Notation}
\newtheorem{alert}[prop]{Caveat}

\usepackage{url}
\usepackage{scalerel}
\DeclareMathOperator{\map}{Map}
\DeclareMathOperator{\alg}{Alg}
\DeclareMathOperator{\m}{Mod}

\newcommand{\cotr}{\cot_{\mathit{R//B}}}
\newcommand{\sqzr}{\sqz_{\mathit{R//B}}}
\newcommand{\tbar}{\overline{T}}
\newcommand{\rmfib}{\mathrm{fib}}
\newcommand{\rmfor}{\mathrm{forget}}

\newcommand{\cwafp}{\mathcal{C}^{\text{wafp}}}
\newcommand{\cwafpf}{\mathcal{C}^{\text{wafpf}}}
\newcommand{\sig}{\Sigma}
\newcommand{\infcat}{$\infty$-category}
\DeclareMathOperator{\pr}{\mathcal{P}\mathit{r}}
\DeclareMathOperator{\st}{St}

\DeclareMathOperator{\com}{Com}
\DeclareMathOperator{\gen}{gen}
\DeclareMathOperator{\qf}{qf}

\DeclareMathOperator{\tc}{\widetilde{\mathit{C}}^*}

\DeclareMathOperator{\afp}{afp}

\DeclareMathOperator{\cpl}{cpl}
\DeclareMathOperator{\fil}{Fil}
\DeclareMathOperator{\fr}{free}
\DeclareMathOperator{\adic}{\mathrm{adic}}
\DeclareMathOperator{\calg}{CAlg}

\DeclareMathOperator{\ce}{CE}
\DeclareMathOperator{\cet}{\widetilde{\ce}}

\DeclareMathOperator{\sq}{St}

\DeclareMathOperator{\klie}{Lie^{\pi}_{\mathit{k},\mathbb{E}^{nu}_{\infty}}}
\DeclareMathOperator{\kdlie}{Lie^{\pi}_{\mathit{k},\Delta}}
\DeclareMathOperator{\bdlie}{Lie^{\pi}_{\mathit{B},\Delta}}

\DeclareMathOperator{\lie}{Lie^{\pi}_{\mathit{R},\mathbb{E}^{nu}_{\infty}}}
\DeclareMathOperator{\aplie}{Lie^{\pi,dap}_{\mathit{R},\mathbb{E}^{nu}_{\infty}}}
\DeclareMathOperator{\dlie}{Lie^{\pi}_{\mathit{R},\Delta}}

\DeclareMathOperator{\apdlie}{Lie^{\pi,dap}_{\mathit{R},\Delta}}

\DeclareMathOperator{\apbdlie}{Lie^{\pi,dap}_{\mathit{B},\Delta}}
\newcommand{\fol}{\mathcal{F}ol}

\DeclareMathOperator{\ch}{\mathbf{Ch}}
\DeclareMathOperator{\dap}{dap}
\DeclareMathOperator{\ap}{ap}

\DeclareMathOperator{\sqz}{\mathrm{sqz}}
\DeclareMathOperator{\perf}{\mathrm{Perf}}
\DeclareMathOperator{\tot}{\mathrm{Tot}}
\DeclareMathOperator{\gr}{\mathrm{Gr}}
\DeclareMathOperator{\neoo}{\mathbb{E}^{nu}_{\infty}}
\DeclareMathOperator{\sym}{\mathrm{Sym}}
\DeclareMathOperator{\lsym}{\mathrm{LSym}}
\DeclareMathOperator{\br}{\mathrm{Bar}}
\DeclareMathOperator{\cobr}{\mathrm{CoBar}}
\DeclareMathOperator{\colim}{\mathrm{colim}}

\DeclareMathOperator{\dD}{\mathbb{D}}
\DeclareMathOperator{\triv}{triv}
\DeclareMathOperator{\kd}{\mathrm{KD}^{pd}}
\DeclareMathOperator{\sseq}{\mathrm{sSeq}}
\DeclareMathOperator{\const}{\mathrm{const}}
\DeclareMathOperator{\vect}{Vect}
\DeclareMathOperator{\aperf}{APerf}
\DeclareMathOperator{\bcirc}{\bar{\circ}}
\DeclareMathOperator{\coh}{Coh}
\DeclareMathOperator{\scoh}{coh}
\DeclareMathOperator{\scr}{SCR}
\DeclareMathOperator{\laft}{laft}
\DeclareMathOperator{\QC}{QC}
\DeclareMathOperator{\fin}{Fin}
\DeclareMathOperator{\levtimes}{\otimes_{lev}}
\DeclareMathOperator*{\bigamalg}{\scalerel*{\amalg}{\sum}}
\DeclareMathOperator{\red}{red}
\DeclareMathOperator{\op}{Op}
\DeclareMathOperator{\dalg}{DAlg}
\DeclareMathOperator{\pd}{pd}

\DeclareMathOperator{\eoo}{\mathbb{E}_{\infty}}
\DeclareMathOperator{\lagd}{LieAlgd^{\pi}_{\textit{B}/\textit{R},\Delta}}
\DeclareMathOperator{\lagdf}{LieAlgd^{\pi,\fil}_{\textit{B}/\textit{R},\Delta}}
\DeclareMathOperator{\aplagd}{LieAlgd^{\pi,dap}_{\textit{B}/\textit{R},\Delta}}
\DeclareMathOperator{\aplagdX}{LieAlgd^{\pi,dap}_{\textit{X}/\textit{R},\Delta}}
\DeclareMathOperator{\aplagdf}{LieAlgd^{\pi,\fil,dap}_{\textit{B}/\textit{R},\Delta}}

\DeclareMathOperator{\Sp}{Sp}

\DeclareMathOperator{\sgr}{gr}
\DeclareMathOperator{\sfil}{fil}
\DeclareMathOperator{\infsus}{\Sigma^{\infty}_{+}}

\DeclareMathOperator{\infcoho}{F^H_*\mathbbb{\Pi}^{\mathit{or}}_{\textit{-/R}}}

\DeclareMathOperator{\infcohgr}{\widehat{F^H_*\mathbbb{\Pi}}_{\textit{-/R}}^{\mathit{or},\gen}}

\DeclareMathOperator{\infcohR}{F^H_*\mathbbb{\Pi}_{\textit{B/R}}}
\DeclareMathOperator{\infcohRo}{F^H_*\mathbbb{\Pi}^{\mathit{or}}_{\textit{B/R}}}
\DeclareMathOperator{\infcohnewfunctor}{F^H_*\mathbbb{\Pi}}
\DeclareMathOperator{\infcohaa}{F^H_*\mathbbb{\Pi}^{or}_{\mathit{A}_1/\mathit{A}_0}}

\newcommand{\ul}{\underline}

\newcommand{\dT}{\mathbb{T}}
\newcommand{\dL}{\mathbb{L}}
\newcommand{\ooop}{$(\infty,1)$-operad}
\newcommand{\clm}{\mathcal{LM}}

\newcommand{\infcats}{$\infty$-categories}
\newcommand{\ed}{\text{\textexclamdown}}
\newcommand{\fD}{\mathfrak{D}}

\newcommand{\dst}{\mathbf{dSt}_{\mathit{R}}}

\newcommand{\cH}{\mathcal{H}_{\pi}}

\DeclareMathOperator{\cotrfil}{\cot_{\mathit{R//B}}^{\fil}}

\DeclareMathOperator{\sqzrfil}{\sqz_{\mathit{R//B}}^{\fil}}

\DeclareMathOperator{\aq}{AQ}
\DeclareMathOperator{\aqf}{\widetilde{AQ}}
\DeclareMathOperator{\dafilr}{\mathcal{D}^{\fil}_\textit{/R}}
\DeclareMathOperator{\dafilbr}{\mathcal{D}^{\fil}_\textit{B/R}}

\newcommand{\spnc}{\mathrm{Spec}^{\textit{nc}}}
\newcommand{\spec}{\mathrm{Spec}}
\DeclareMathOperator{\dg}{DG}
\newcommand{\fh}{\mathfrak{h}}
\newcommand{\fg}{\mathfrak{g}}

\usepackage{latexsym, amssymb,amscd,mathrsfs,stackrel}
\usepackage{tikz-cd,tikz}
\usetikzlibrary{backgrounds}
\usepackage{graphicx} 
\usepackage{float} 
\usepackage{subfigure} 
\usetikzlibrary{arrows}
\usepackage{geometry}

\title{A duality between Lie algebroids and infinitesimal foliations}
\author{Jiaqi Fu}
\date{}
\begin{document}
	\maketitle
	\begin{abstract}
		There are two natural analogues of algebraic foliations in derived algebraic geometry, called partition Lie algebroids and infinitesimal derived foliations, and both make sense in general characteristics. We construct an equivalence between these two notions under some finiteness conditions. Our method is refining the PD Koszul duality in \cite{BM}\cite{BCN} using the (completed) Hodge filtration.
	\end{abstract}
	\section{Introduction}
	Let $X$ be a smooth variety of dimension $n$ over $\mathbb{C}$. A \textit{holomorphic foliation on $X$} consists of a holomorphic atlas $\{\phi_i:U_i^{\subset X}\to V_i^{\subset \mathbb{C}^k}\times W_i^{\subset\mathbb{C}^{n-k}}\}$ such that the transition maps are of the form $\phi_j\circ\phi^{-1}_i(x,y)=(f(x),g(x,y))$ with $x\in V_i$ and $y\in W_i$. The level set of each $x_0\in V_i$ is a complex submanifold of $U_i$, and gluing up the overlapping ones gives rise to immersed complex submanifolds of $X$, which are called \textit{leaves}. Holomorphic foliations have two natural algebraic analogues: one is given by smooth Lie algebroids, i.e.\ $k$-codimensional subbundles $\mathfrak{g}\hookrightarrow T_{X/k}$ of the tangent bundle which are closed under the Lie bracket $[-,-]$ \cite{pradines1967theorie}; the other is given by smooth algebraic foliations, namely quotient bundles $\Omega_{X/k}\twoheadrightarrow\Omega_{\mathscr{F}}$ with a rank $k$ kernel $I$ being a differential ideal \cite{Ayoub}, i.e.\ the composition $\bar{d}_{DR}:\mathcal{O}_{X}\to\Omega_{X/k}\twoheadrightarrow \Omega_{\mathscr{F}}$ induces a quotient map of \textit{commutative differential graded algebras} $({\wedge}^\bullet_R \Omega_{X/k},d_{DR})\twoheadrightarrow ({\wedge}^\bullet_R \Omega_{\mathscr{F}},\bar{d}_{DR})$. These two notions are shown to be equivalent using Chevalley-Eilenberg algebras \cite[Proposition 12.4]{liediff}. However, we often need to consider their generalization with singularities, i.e.\ drop the locally free condition on $\mathfrak{g}$ or $\Omega_\mathscr{F}$, where a disharmony appears between the tangent formalism (e.g.\ \cite[\S3]{Miyaoka}) and the cotangent formalism (e.g.\ \cite[\S1]{Baum}, \cite[\S3]{malgrange1977frobenius}).
	
	A well-known philosophy suggests that many algebro-geometric objects have well-behaved derived analogues. This paper is devoted to exploring the connection between the above two formalisms of algebraic foliations in derived algebraic geometry.
	\subsection{Formal geometry of foliations}
	An algebraic leaf of some algebraic foliation $\mathscr{F}$ on $X/R$ is a subvariety $i:Z\hookrightarrow X$ such that $\Omega_{Z/R}$ and $i^*\Omega_{\mathscr{F}}$ are isomorphic as quotient sheaves of $i^*\Omega_{X/R}$.\ Unfortunately, even a smooth algebraic foliation over $\mathbb{C}$ might not admit any algebraic leaves. For instance, the smooth closed $1$-form $dx/x+dy$ on $\spec(\mathbb{C}[x,x^{-1},y])$, generating a subbundle differential ideal of rank $1$, has only transcendental leaves in the form of $\{xe^{y}=c\}$. When one works on a singular scheme or over a base ring without embedding into $\mathbb{C}$, even the definition of leaves can be tricky, not to mention their existence or uniqueness. One approach to fixing this is finding a formal version of Frobenius theorem, i.e.\ determining the infinitesimal deformations along the foliations. Some smooth cases of this problem were solved before the rise of derived geometry, e.g. \cite[Proposition 2.4]{Eke} and \cite[\S6]{Miyaoka}.\ However, proper complex varieties  rarely foliate without singularity (e.g. $\mathbb{P}^n_\mathbb{C}$), while the singular foliations exist in great abundance \cite[p.287]{baum1972singularities}.
	
	\textbf{Complex derived foliations.}
	The notion of \textit{derived foliations}, recently proposed by To\"en and Vezzosi in \cite[\S1.2]{TV}, offers a new framework for studying singular algebraic foliations over $\mathbb{C}$. Their definition could be motivated by the hidden smoothness philosophy: a reasonable singular foliation should be the truncation of some ``regular'' higher notion of foliations.
	
	Recall that an algebraic foliation $\mathscr{F}$ \cite[\S6.1]{Ayoub} on a scheme $\spec(R)/\mathbb{C}$ is determined by a quotient of quasi-coherent sheaves $f:\Omega_{R/\mathbb{C}}\twoheadrightarrow \Omega_{\mathscr{F}}$, with the ``integrability'' condition that $f$ gives rise to a \textit{commutative differential graded algebra} (cdga) over $\mathbb{C}$ 
	\[R\xrightarrow{\bar{d}_{DR}}\Omega_{\mathscr
		F}\xrightarrow{\bar{d}_{DR}}\wedge^2_R\Omega_{\mathscr
		F}\xrightarrow{\bar{d}_{DR}}\wedge^3_R\Omega_{\mathscr
		F}\xrightarrow{\bar{d}_{DR}}\ldots,\]
	whose cdga structure is inherited from the de Rham algebra $(\wedge^\bullet_R \Omega_{R/\mathbb{C}},d_{DR})$.
	
	In derived geometry in characteristic 0, the ambient categories of rings and modules are replaced by the \infcat\ of cdgas and their derived \infcats \cite{HAG2}.\ To express a foliation structure here, the strict cdga in modules as above should be replaced by a notion of \textit{homotopy-coherent cdgas} in $\m_\mathbb{C}$, the derived \infcat\ of chain complexes over $\mathbb{C}$. Here a \textit{homotopy-coherent cochain complex} (Definition \ref{df dg}) refers to $C^\bullet$ a $\mathbb{Z}$-indexed collection of chain complexes and morphisms $d:C^\bullet\to C^{\bullet+1}$ in $\m_\mathbb{C}$, where $d^2$ vanishes up to a coherent homotopy. The homotopy-coherent cdgas are then $\eoo$-algebras in homotopy-coherent chain complexes by definition. A primary example is the \textit{derived de Rham algebra} $(\wedge^\bullet_{R}\dL_{R/k},d_{DR})$ in \cite{bhatt2012completions}\cite[\S1.2]{TV}.
	\begin{df}\label{df1.1}\cite[\S1.2]{TV}
		Let $\spec(R)/\mathbb{C}$ be an affine derived scheme of finite presentation. A \textit{derived foliation} $\mathscr{F}$ on $\spec(R)$ is a homotopy-coherent cdga \footnote{It is modeled as a \textit{graded mixed cdga} in \cite{TV}.} $(A^\bullet, d)$ over $\mathbb{C}$ such that $A^0\simeq R$, $\dL_{\mathscr{F}}:=A^1$ is a connective perfect $R$-module, and there is a natural equivalence of graded cdgas $\sym_R(\dL_{\mathscr{F}}[-1])\simeq A^\bullet[-\text{\small\textbullet}]$ (via \ref{thm dg_-}). The $R$-module $\dL_{\mathscr{F}}$ is called the \textit{cotangent complex} of $\mathscr{F}$.
	\end{df}
	The derived de Rham algebra $(\wedge^\bullet_{R}\dL_{R/k},d_{DR})$ is initial among the derived foliations in this sense. 
	If the truncation map $\Omega_{R/\mathbb{C}}\to \pi_0(\dL_{\mathscr{F}})=\Omega_{\mathscr{F}}$ is surjective, then $\Omega_{\mathscr{F}}$ determines an algebraic foliation, and $\mathscr{F}$ could be thought as a derived enhancement.
	
	\textbf{Formal deformation theory.} A principle introduced by Deligne and Drinfeld postulates that every formal moduli problem in characteristic $0$ is controlled by a dg-Lie algebra.\ It was then further developed in \cite{GM} \cite{Kont} \cite{Man}, and eventually was upgraded to an equivalence by the work of Hinich \cite{Hin}, Pridham \cite{pridham2010unifying}, and Lurie \cite{lurie2011derived}. This equivalence is generalized into a ``many-object'' version using \textit{dg-Lie algebroids} \cite{gaitsgory2019study}\cite{Nuiten}.
	
	Meanwhile, the existence of a (2-term) derived enhancement is related to the integrability of algebraic foliations \cite[Corollary 1.5.4, 2.3.3]{TV} by formal schemes. This also suggests that the higher structure of a derived foliation should determine the formal deformation along the ``imaginary leaves''. Here, an advantage of $\mathscr{F}$ compared to its truncation $\Omega_{\mathscr{F}}=\pi_0(\dL_{\mathscr{F}})$ is that $\mathscr{F}$ could be related to dg-Lie algebroids.\ Given a dg-Lie algebroid $\rho:\mathfrak{g}\to \dT_{R/\mathbb{C}}$  \cite[Definition 2.1]{Nuiten} such that the underlying $R$-module $\mathfrak{g}$ is perfect and of Tor-amplitude $(-\infty,0]$, the graded cdga $\sym_R(\mathfrak{g}^\vee[-1])$ admits a natural derived foliation structure, whose extra differential is defined in terms of the Lie bracket and $\rho$. To\"en and Vezzosi conjecture that this assignment is fully faithful \cite[\S1.3.1]{TV}.\ Our goal is to prove this conjecture, as well as a more general statement in arbitrary characteristics.
	
	\textbf{Interlude: Some Conventions.}
	Until now, everything happens in characteristic 0. However, once we are equipped with some more refined machinery, most of the discussion also holds in general characteristics. We will return to this in \S\ref{sec1.3} with more details.\ Here, we only make necessary clarification and variation on the notions to state our main theorem.
	
	The \textit{derived commutative algebras}  are chosen as our preferred generalization of cdgas, which are the algebras of the monad $\lsym$, cf. \cite[I\S4]{illusie2006complexe} \cite{kaledin2015trace} for a definition using simplicial-cosimplicial algebras, or \cite[\S4.2]{Raksit} applying the method of left-right extension in \cite[\S3]{BM} to $\sym$.\ For a map of simplicial commutative rings $R\to B$, a \textit{foliation-like algebras} $A$ over $B$ relative to $R$ is a completely filtered derived $R$-algebra such that $\gr_0A\simeq B$ and the natural map $\lsym_B(\gr_1A)\to\gr A$ is an equivalence of graded algebras, see Notation \ref{n4.14} (2).
	
	Alternatively, Proposition \ref{prop5.1n} shows that each foliation-like algebra $A$ can be regarded as
	\begin{equation}
		B\xrightarrow{d}\dL_{A}\xrightarrow{d}\lsym_B^2(\dL_{A}[-1])[2]\xrightarrow{d}\ldots
	\end{equation}
	a free graded derived algebra with a compatible homotopy-coherent differential $d$, where $\dL_{A}:=\gr_1A[1]$ is the \textit{cotangent complex} of $A$. The foliation-like algebras are then a generalization of To\"en and Vezzosi's derived foliations (Remark \ref{rk5.4}(2)) to general characteristics.\ Let $\dalg^{\fol}_{B/R,\ap}$ be the \infcat\ of foliation-like algebras whose cotangent is an almost perfect $B$-module.
	
	Write $\lagd$ for the \infcat\ of \textit{(derived) partition Lie algebroids} over $B$ relative to $R$ (Definition \ref{df4.4}), which generalizes dg-Lie algebroids into general characteristics. One application of $\lagd$ is to generalize Lurie and Pridham's theorem into general characteristics \cite{BM} \cite{BMN}.
	\begin{maintheorem}[Affine case, \ref{thm4.25}]\label{thm1.1}
		Let $R\to B$ be a morphism of coherent simplicial commutative rings such that $\dL_{B/R}$ is an almost perfect  $B$-module. Then, there is an equivalence
		\[\aplagd\xrightarrow{\simeq}\dalg^{\fol,op}_{B/R,\ap},\]
		from the \infcat\ of \textnormal{dually almost perfect} (\ref{df2.25}) partition Lie algebroids to the \infcat\ of \textnormal{almost perfect} foliation-like algebras.
	\end{maintheorem}
	\begin{remark}
		Here $B/R$ is not necessarily almost of finite presentation. For instance, the inclusion $k\hookrightarrow k^{\mathrm{sep}}$ from a field into its separable closure satisfies the conditions.
	\end{remark}
	
	\textbf{Infinitesimal derived foliations.} We will now globalize Theorem \ref{thm1.1} to a geometric statement. Let $\mathbb{V}(\dL_{\mathscr{F}}[-1])$ be the \textit{non-connective affine space} corepresented by $\lsym_B(\dL_{\mathscr{F}}[-1])$. As a ``stacky vector bundle'', it admits a natural $\mathbb{G}_m$-action by scaling. Meanwhile, a $R$-linear complete filtration could be regarded as a comultiplication by $R[\epsilon]$ (with $|\epsilon|=1$ and the Hopf structure $\Delta \epsilon=\epsilon\otimes 1+1\otimes\epsilon$), i.e. an action of $\Omega_0\mathbb{G}_a:=\spec(R[\epsilon])$, see \cite[Remark 1.1]{pantev2013shifted}.
	
	Motivated by these observations, an \textit{infinitesimal derived foliations} \cite[\S2]{Toen2023} (or Definition \ref{df5.5}) $\mathscr{F}$ is by definition a $\mathbb{G}_m\ltimes\Omega_0\mathbb{G}_a$-equivariant $R$-stack such that the underlying graded stack being the $\mathbb{V}(\dL_{\mathscr{F}}[-1])$ for some $B$-module $\dL_{\mathscr{F}}$. Let $\fol^{\pi}_{\ap}(\spec(B)/R)$ denote the \infcat\  of infinitesimal derived foliations for which $\dL_{\mathscr{F}}$ is \textit{almost perfect}.\ Taking the non-connective affine space induces a categorical equivalence $\dalg^{\fol,op}_{B/R,\ap}\simeq\fol^{\pi}_{\ap}(\spec(B)/R)$ (Theorem \ref{thmn5.11}), which shows that our switch of convention causes no confusion.
	
	The notion $\fol^{\pi}_{\ap}(-/R)$ can be extended to \textit{quasi-compact and quasi-separated} (qcqs) derived schemes $X$ with $\dL_{X/R}$ almost perfect \cite[Definition 2.8.4.4]{SAG}, by considering the ``stacky vector bundle'' $\mathbb{V}(E)$ for quasi-coherent sheaves $E$ and global $\mathbb{G}_m\ltimes\Omega_0\mathbb{G}_a$-actions on $\mathbb{V}(E)$ compatible with the scaling.

	\begin{maintheorem}[Global case, \ref{thmn5.13}]\label{thm1.4}
		Let $X$ be a locally coherent qcqs derived scheme over an eventually coconnective coherent simplicial commutative ring $R$, whose cotangent complex $\dL_{X/R}$ is an almost perfect quasi-coherent sheaf. Then, there is an equivalence
		\[\aplagdX\xrightarrow{\simeq}\fol^{\pi}_{\ap}(X/R),\]
		from the \infcat\ of \textnormal{dually almost perfect} partition Lie algebroids to the \infcat\ of infinitesimal derived foliations with \textnormal{almost perfect cotangent complexes}.
	\end{maintheorem}
		\subsection{Strategy of proof}\label{sec:1.2}
	Our main tool is the right-left extension and divided power Koszul duality presented in \cite{BM}\cite{BCN}.

	We first review filtrations in stable \infcats\ and the theory of pro-coherent modules in \S\ref{sec2}. Let $R$ be a coherent simplicial commutative ring \cite[Definition 7.2.4.16]{HA}. The pro-coherent module category $\QC^\vee_R$ enjoys the pleasant property that the $R$-linear dual $(-)^\vee$ induces a symmetric monoidal equivalence of full subcategories between $\aperf_R$ consisting of almost perfect modules and $\aperf^\vee_R$ its essential image under $(-)^\vee$.	We then recall the theory of (derived) $\infty$-operads (Definition \ref{df3.22}) in the pro-coherent context, and construct the \infcats\ of filtered and graded algebras.
	
		Section \ref{sec3} is devoted to a filtered PD Koszul duality. Let $\mathcal{P}$ be a reduced derived $\infty$-operad (Proposition \ref{prop3.9n}) with PD Koszul dual $\kd(\mathcal{P})$. We construct an adjunction of filtered algebras 
	\[\kd:\alg_{\mathcal{P}}(\fil_{\ge 1}\QC^\vee_R)\rightleftarrows \alg_{\kd(\mathcal{P})}(\fil_{\le -1}\QC^\vee_R)^{op}:\aq,\]where $\aq$ takes an increasingly filtered $\kd(\mathcal{P})$-algebra $L$ to a decreasingly filtered $\mathcal{P}$-algebra $\aq(L)$, which is automatically complete (\ref{dualfil}).\ Applying $\aq$ to the constantly filtered algebras gives rise to a sifted-colimit-preserving functor $\aqf:\alg_{\kd(\mathcal{P})}(\QC^\vee_R)\to  \alg_{\mathcal{P}}(\fil_{\ge1}\QC^\vee_R)^{op}$.
	
	When $\mathcal{P}$ is the derived $\infty$-operad of \textit{non-unital derived commutative algebras} $\com^{nu}$, and its Koszul dual is by definition the derived PD $\infty$-operad of \textit{(derived) partition Lie algebras} $\dlie$, the functor $\aqf$ is preferably written as $\cet$ for (completed) Hodge filtered Chevalley-Eilenberg complex. If $R$ is a regular $\mathbb{Q}$-algebra, $\dlie$ is the operad of (shifted) dg-Lie algebras over $R$, and $\cet$ is the ordinary Chevalley-Eilenberg functor filtered by the degree of forms. For each $R$-augmented algebra $S$, $\cet(\kd(S))$ is exactly the Hodge-completed de Rham cohomology $\widehat{\mathrm{DR}}(R/S)$ (Proposition \ref{prop4.20old}).
	\begin{theorem}[Theorem \ref{thm3.24n}]\label{thmintro2}
		If $\mathcal{P}$ is \textnormal{almost finitely presented} \cite[Definition 3.52]{BCN} so that $\kd(\mathcal{P})$ is dually almost perfect, then the functor $\aqf$ restricts to a fully faithful embedding \[\alg_{\kd(\mathcal{P})}(\aperf^\vee_R)\hookrightarrow \alg_{\mathcal{P}}(\fil_{\ge1}\QC^\vee_R)^{op},\]whose essential image consists of complete $\mathcal{P}$-algebras $A$ such that the graded pieces form a free graded $\mathcal{P}$-algebra generated by an almost perfect $R$-module $\gr_1A$ concentrating at weight $1$.	
	\end{theorem}
	
	Theorem 1.15 in \cite{BM} establishes a non-filtered Koszul duality for coconnective partition Lie algebras (of finite type). The novelty here is that we do not assume coconnectivity, for which the filtered structure plays a crucial role (even in characteristic $0$).
	
	\begin{exa}
		Over $\mathbb{R}$, the Lie algebra cohomology of $\mathfrak{su}(2)$ (or $\mathfrak{su}(2)[1]$ as a shifted dg-Lie algebra) is the free cdga $\ce(\mathfrak{su}(2))\simeq\wedge_{\mathbb{R}}(\mathbb{R}.e_3)$ with $|e_3|=-3$, cf.\cite[Vol.I Theorem 6.5.(2)]{mimura1991topology} \cite[Theorem 15.2]{chevalley1948cohomology}. Therefore, the corresponding shifted dg-Lie algebra of $\ce(\mathfrak{su}(2))$ is abelian and has the underlying module $\mathbb{R}[3]$, which differs from $\mathfrak{su}(2)[1]$. It means that $\ce$ without filtration does not embed the ordinary Lie algebras into the cdgas. At the same time, $\cet(\mathfrak{su}(2)[1])$ as a graded mixed complex, is the strict cochain complex of the Lie algebra cohomology of $\mathfrak{su}(2)$, which loses no information.
		
		From a geometric perspective, the formal moduli problem determined by $\mathfrak{su}(2)[1]$ is $B\widehat{\mathrm{SU}}(2)$, the classifying stack of the formal group associated to $\mathrm{SU}(2)$. As a highly stacky object, $B\widehat{\mathrm{SU}}(2)$ is not (pro-)corepresented by $\Gamma(B\widehat{\mathrm{SU}}(2),\mathcal{O})\simeq \ce(\mathfrak{su}(2)[1])$. However, as an infinitesimal derived foliation, the $\mathbb{G}_m\ltimes\Omega_0\mathbb{G}_a$-action on $B\widehat{\mathrm{SU}}(2)$ induces a complete filtered structure \cite{Mou}, where the graded stack is an affine stack corepresented by $\sym_{\mathbb{R}}(\mathbb{R}^{\oplus 3}[-1])$.
		
	\end{exa}
	
	In \S\ref{sec4}, we construct $\lagd$ the \infcat\ of \textit{partition Lie algebroids over $B$ (relative to $R$)} for $R\to B$ a morphism of simplicial commutative rings with some finiteness conditions, in the spirit of \cite{BW}. When $R$ and $B$ are regular $\mathbb{Q}$-algebras, the \infcat\ $\lagd$ is equivalent to the \infcat\ of dg-Lie algebroids in \cite[\S3]{youngNuiten}.
	\begin{theorem}[Theorem \ref{thm4.25}]\label{thmintro3}
		Let $R\to B$ be a map of coherent simplicial commutative rings such that $\dL_{B/R}$ is almost perfect as an $B$-module. There is a fully faithful embedding, by a ``many-object'' variant of $\cet$,
		\[\tc:\aplagd\hookrightarrow (\dafilbr)^{op}\]
		from dually almost perfect partition Lie algebroids,\ to filtered $R$-algebras $A$ with a natural augmentation $\gr_0A\simeq B$. Moreover, its essential image consists of almost perfect foliation-like algebras.
	\end{theorem}
	\noindent 
	
	Then we construct in \S\ref{sec5} a functor $\spnc$ realizing foliation-like algebras as infinitesimal derived foliations, which relies on the fact that taking non-connective spectra $\dalg(\m_R)\to (\dst)^{op}$ is symmetric monoidal with respect to the cocartesian monoidal structures. Finally, we use a reasoning inspired by \cite{Monier} to prove that:
	\begin{theorem}[Theorem \ref{thmn5.11}]\label{thm1.5}
		Further assuming that $R$ is eventually coconnective, there is a categorical equivalence
		\[\spnc:\dalg^{\fol,op}_{B/R,\ap}\simeq \fol^{\pi}_{\ap}(\spec(B)/R)\]between almost perfect foliation-like algebras and almost perfect infinitesimal derived foliations.
	\end{theorem}
	\noindent The required equivalence in the main theorem is gluing up $\spnc\circ\tc$ on each affine chart.
	
		\subsection{Foliations in general characteristics}\label{sec1.3}
	Ordinary algebraic geometry has two inequivalent generalizations in general characteristics, namely \textit{derived algebraic geometry} based on simplicial commutative rings and \textit{spectral algebraic geometry} based on $\eoo$-ring spectra.\ The aforementioned derived algebras serve as a non-connective generalization of simplicial commutative rings. The infinitesimal derived foliations also live in the world of derived algebraic geometry.
	
	Now we exemplify the interest and subtlety of foliation theory away from characteristic $0$:
	\begin{exa}[Kronecker's criterion of root of unity]
		Consider some $x=\exp(2\pi i\lambda)\in\mathbb{C}^{\times}$ with $\lambda\in \bar{\mathbb{Q}}$. Here $x$ is a root of unity (or equivalently $\lambda\in\mathbb{Q}$) precisely when, for almost every non-zero prime $\mathfrak{p}$ of $\mathcal{O}_{\mathbb{Q}(\lambda)}$, the residue class of $\lambda$ in $\kappa(\mathfrak{p})$ belongs to the prime field $\mathbb{F}_p\subset \kappa(\mathfrak{p})$ \cite[\S2]{esnault2023lectures}.
		
		Bost gave a foliation-theoretic proof \cite[Corollary 2.4]{Bost} by observing that, setting $G:=\mathbb{G}_{m,\mathbb{Q}(\lambda)}\times \mathbb{G}_{m,\mathbb{Q}(\lambda)}$, the graph of the multiplication by $\lambda$ defines a sub-Lie algebra $\fh\hookrightarrow\mathrm{Lie}(G)$ and then a smooth algebraic foliation by $\mathrm{Lie}(G)\cong \Gamma(G,T_{G/\mathbb{Q}(\lambda)})$, and $x$ is a root of unity if and only if $\fh$ has an algebraic leaf, since the only codimension $1$ subgroups of $G$ are defined by one equation of the form $X^aY^b=1$. On the other hand, the residue condition on $\lambda$ is translated into, for almost all non-zero prime $\mathfrak{p}\subset \mathcal{O}_{\mathbb{Q}(\lambda)}$, the reduction of $\fh$ to $\kappa(\mathfrak{p})$ is closed under \textit{restriction}. The question is then reduced to a special case of the generalized Grothendieck-Katz p-curvature conjecture. This conjecture is formulated in \cite[conjecture F]{Aconjecture} and proved, under the condition that the analytic leaves contain open subsets satisfying the Liouville property, in \cite[Theorem 2.1]{Bost}.
		
	\end{exa}
	This example shows that, even in smooth case, a good notion of foliations away from characteristic $0$ should include operations beyond the Lie bracket (e.g. the restriction) to fully capture the formal deformation along the ``imaginary leaves''.
	
	Here is an example containing higher restrictions.
	\begin{exa}
		Let $L$ be a derived partition Lie algebra over $\mathbb{F}_2$ such that $\pi_n(L)=0$ for $n\ne 0,1$. The homotopy operations on $L$ consist of $([-,-],(-)^{\{2\}})$ a restricted Lie structure on $\fg_1:=\pi_1(L)$, a $\fg_{1}$-representation structure on $\fg_0:=\pi_0(L)$ and an additive higher restriction $R^1:\fg_1\to \fg_0$, cf. Proposition \ref{EXAprop1}.
		\[\begin{tikzcd}
			& {\fg_1} \\
			{} & {\fg_0}
			\arrow["{[-,-]}", from=1-2, to=1-2, loop, in=150, out=210, distance=5mm]
			\arrow["{(-)^{\{2\}}}", from=1-2, to=1-2, loop, in=330, out=30, distance=5mm]
			\arrow["{R^1}", from=1-2, to=2-2]
			\arrow["{[\fg_1,-]}", from=2-2, to=2-2, loop, in=150, out=210, distance=5mm]
		\end{tikzcd}\]	In particular, the partition Lie algebra $\mathrm{Lie}(\mu_2)^{\{2\}}$ of the Frobenius kernel $\mu_2:=\ker((-)^2:\mathbb{G}_{m,\mathbb{F}_2}\to \mathbb{G}_{m,\mathbb{F}_2})$\footnote{encoding the formal moduli problem $B\mu_2$} has $\fg_i\cong \mathbb{F}_2.D_i$ ($i=0,1$), where the operations are given by $[D_1,D_i]=0$ ($i=0,1$), $(D_1)^{\{2\}}=D_1$ and $R^1(D_1)=D_0$, see Proposition \ref{EXAprop4}.
	\end{exa}
	
	A foliation-like algebra $A\in \dalg^{\fol}_{B/R,\ap}$ captures the higher restrictions in a more subtle way. Following Notation \ref{n4.14} (2), the graded algebra $\gr A$ is $ \lsym_B(\dL_{A}[-1])$ rather than the wedge power $(\wedge_B\dL_{A})[-*]$. The point is that, unlike in characteristic $0$, the derived symmetric power $\lsym$ is no longer compatible with the degree shifting:
	\begin{exa}
		Let $v$ and $\tau$ be the generators of $E:=\dL_{\mu_2/\mathbb{F}_2}$ at degree $1$ and $0$ respectively. One has easily $\wedge_{\mathbb{F}_2}(\mathbb{F}_2.\tau)[-*]\simeq \mathbb{F}_2\oplus \mathbb{F}_2.\tau[-1]$. However, $\pi_*(\lsym_{\mathbb{F}_2}(\mathbb{F}_2[-1].\tau))\cong \mathbb{F}_2[\tau,\st^0\tau,(\st^0)^{\circ2}\tau,\ldots]$ is a polynomial algebra generated by $(\st^0)^{\circ n}\tau$ lies in degree $-1$ and weight $2^n$ \cite[Theorem 4.0.1]{priddy1973mod}.
		
		Considering $d:E\to \lsym^2(E[-1])[2]$ the extra differential from $\tc(\mathrm{Lie}(\mu_2)^{\{2\}})$, the formulas $(D_1)^{2}=D_1$ and $R^1(D_1)=D_0$ are dual to the fact that $d:\tau\mapsto \tau^2$ and $d:v\mapsto \st^0\tau$, Lemma \ref{EXAlemma2}.
		
	\end{exa}
	
		\subsection{Related works}
	\textit{Partition Lie algebroids.}\ The filtered PD Koszul duality in \S\ref{sec3.4} is a continuation of \cite{BM}\cite{BCN}. Here we add a filtered structure on the Koszul duality of (PD) $\infty$-operads and discuss the convergence of arity $0$. Our definition of partition Lie algebroids is in the spirit of \cite{BW}, and agrees with that in \cite{BMN}. In characteristic $0$, a detailed discussion of Lie algebroids in pro-coherent modules could be found in \cite{Nuiten}.
	
	\textit{Derived foliations.}\ There is more than one natural notion of derived foliations away from characteristic $0$. For instance, there is a notion of \textit{derived foliations} in \S2\cite{Toen} other than infinitesimal derived foliations.
	
	The scope of this paper is limited to the small Zariski site of some derived scheme, but most of the content should still hold on the small \'etale site of reasonable Deligne-Mumford stacks. For example, Alfieri proves a theorem similar to Theorem \ref{thm1.5} for relative derived algebraic spaces in characteristic $0$ in an ongoing project.
	
	\textbf{Conventions.} Unless stated conversely, we implicitly fix a Grothendieck universe $\mathbb{U}$ and work with $\mathbb{U}$-small objects.
	
	Every category mentioned will be, by default, an $\infty$-category, for which we mainly use the model of quasicategories detailed in \cite{HTT}.\ Similarly, every notion in algebra or topology should be understood in the \textit{derived} or \textit{homotopical} sense unless stated conversely. For instance, the module category $\m_R$ means the derived \infcat, while $\sym_R$ and $\lsym_R$ refer to the homotopical and derived symmetric power.
	
	Two different notions of higher operads are involved in this paper. To avoid ambiguity, the ones through (generalized) symmetric sequences are referred as (generalized) $\infty$-operads, while the ones through fibrations of \infcats\ are called \ooop s.

	The homotopy-coherent chain complexes appear as the comodules of a derived Hopf algebra, which will take different forms as $R[\epsilon]$, $\dD^\vee_-$ or $1_{\mathcal{C}}\otimes_{1_{\mathcal{C}}[t]}1_{\mathcal{C}}$ depending on the context.
	
	\textbf{Acknowledgments.} It is a pleasure to express our infinite gratitude to Bertrand To\"en and Lukas Brantner for their guidance throughout this project. We are also grateful to Joan Mill\`es and Joost Nuiten for introducing us to the world of homotopical algebras. We want to thank Benjamin Antieau, Tasos Moulinos and Marco Robalo for helpful conversations.
	
	We wish to thank Institut de Mathématiques de Toulouse and LabEx CIMI for their funding and support.
	We also thank Yifei Zhu and Southern University of Science and Technology for their hospitality during my visit in the summer of 2023.

	\section{Preliminaries}\label{sec2}
	
	This section is dedicated to reviewing filtered algebras in stable \infcats\ and fixing notations.
	\subsection{Filtrations in module categories}\label{sec2.1}
	\begin{notation}\label{n2.1}
		For an \infcat\ $\mathcal{C}$, let $\fil\mathcal{C}$ (resp. $\gr\mathcal{C}$) be the \infcat\ of filtered (resp. graded) objects in $\mathcal{C}$. A filtered object can be written as $F_\star X=(\ldots\to F_{1}X\to F_{0}X\to F_{-1}X\to\ldots)$. Evaluating at degree $n$ admits fully faithful left adjoint in both cases, 
		\[(-)_n:\mathcal{C}\hookrightarrow\fil\mathcal{C},\ \ \ \ [-]_n:\mathcal{C}\hookrightarrow\gr\mathcal{C},\]where $(-)_n$ is given by $F_i(X)_n\simeq X$ for $i\le n$ and $F_i(X)_n\simeq 0$ otherwise, $[-]_n$ is the embedding into weight $n$. A detailed discussion of filtrations in stable \infcat\ could be find in \cite[\S2]{gwilliam2018enhancing}.
	\end{notation}
	\begin{df}\label{df2.2}
		Suppose that $\mathcal{C}$ is a stable \infcat, the \textit{functor of associated graded objects}
		\[\gr:\fil\mathcal{C}\to \gr\mathcal{C}\]takes $X=\{F_iX\}_{i\in\mathbb{Z}}$ to $\gr X$ such that $\gr_iX\simeq \mathrm{cofib}(F_{i+1}X\to F_iX)$.
	\end{df}
	Recall that the $\infty$-category $\mathcal{P}r^{L}$ of presentable $\infty$-categories is endowed with a symmetric monoidal structure \cite[Proposition 4.8.1.15]{HA} such that $\calg(\mathcal{P}r^{L})$ consists of presentable symmetric monoidal \infcats. The \infcat\ $\pr^{\st}$ of presentable stable \infcats\ is a symmetric monoidal localization of $\pr^L$ \cite[Proposition 4.8.2.18]{HA}.
	\begin{remark}\label{rk2.3}
		The addition of $\mathbb{Z}$ can be understood as a symmetric monoidal structure on $\mathbb{Z}^{\le}$ \textit{the $1$-category defined by the linear order $\le$} or $\mathbb{Z}^{dis}$ \textit{the discrete $1$-category}. By Day convolution \cite[\S2.2.6]{HA}, we can enhance $\fil(-)$ and $\gr(-)$ into endofunctors of $\calg(\mathcal{P}r^{\st})$, while the embeddings $(-)_0$, $[-]_0$ and associated graded objects $\gr$ can be regarded as natural transformations between $\fil(-)$, $\gr(-)$ and $id_{\calg(\mathcal{P}r^{\st})}$.
	\end{remark}
	\begin{df}\label{df2.4}
		Assume that $\mathcal{C}$ is a stable \infcat\ admitting sequential limits.\\
		(1) Let $\gr_{\ge a}\mathcal{C}\subset \gr\mathcal{C}$ ($\gr_{\le -a}\mathcal{C}\subset \gr\mathcal{C}$) be the full subcategory spanned by $X_\star$ such that $X_i\simeq 0$ for all $i< a$ (resp. $i> -a$).\\(2) A filtered object $\{F_iX\}$ in $\mathcal{C}$  is said to be \textit{complete} if $\lim F_i X\simeq 0$. By $\fil^{\cpl}\mathcal{C}$ denote the full subcategory of $\fil\mathcal{C}$ spanned by the complete filtered objects in $\mathcal{C}$.\\(3) Denote the full subcategory of $\fil\mathcal{C}$ spanned by (complete) $X$ such that $\gr(X)\in \gr_{\ge a}\mathcal{C}$ (resp. $\gr(X)\in \gr_{\le-a}\mathcal{C}$) by $\fil_{\ge a}\mathcal{C}$ (resp. $\fil_{\le -a}\mathcal{C}$). 
	\end{df}
	\begin{remark}\label{rk2.5}
		For each $(\mathcal{C},\otimes)\in\calg(\pr^{\st})$, Proposition 2.2.1.9 in \cite{HA} shows that there is a symmetric monoidal structure $\widehat{\otimes}$ on $\fil^{\cpl}\mathcal{C}$ such that $\gr$ admits a factorization of symmetric monoidal left adjoints
		$\fil\mathcal{C}\xrightarrow{(-)^\wedge}\fil^{\cpl}\mathcal{C}\xrightarrow{\gr|_{\fil^{\cpl}\mathcal{C}}}\gr\mathcal{C}.$
	\end{remark}
	
	Now we interpret $\fil^{\cpl}\mathcal{C}$ as a generalization of cochain complexes. Recall that a cochain complex in some abelian category $\mathscr{A}$ consists of a sequence of objects $C^\bullet$ in $\mathscr{A}$ and morphisms
	\[\ldots\xrightarrow{d} C^{n-1}\xrightarrow{d} C^n\xrightarrow{d}C^{n+1}\xrightarrow{d} \ldots\]such that $d^2=0$. More categorically, one can consider the pointed 1-category $\ch$ given by
	\begin{align*}
		\mathrm{obj}\ch=&\ \mathbb{Z}\cup\{*\},\\
		\ch(n,m)=&\begin{cases}
			\{id,\mathbf{0}\}, &\text{if }m=n;\\
			\{\partial,\mathbf{0}\}, &\text{if }m=n-1;\\
			\{0\}, &\text{otherwise}.
		\end{cases}
	\end{align*}A cochain complex in $\mathscr{A}$ is then equivalent to a contravariant functor $\ch^{op}\to \mathscr{A}$ preserving the zero object. This observation motivates the following definition:
	\begin{df}\label{df dg}
		For a stable \infcat\ $\mathcal{C}$, we define the \infcat\ of \textit{homotopy-coherent cochain complexes} in $\mathcal{C}$ as $\dg_{-}\mathcal{C}:=\mathrm{Fun}_*(N(\ch^{op}),\mathcal{C})$ the \infcat\ of pointed functors.
	\end{df}
	The natural inclusion $\mathbb{Z}^{dis}\hookrightarrow\ch$ induces a conservative functor $u:\dg_-\mathcal{C}\to \gr\mathcal{C}$. If $\mathcal{C}$ is equipped with a t-structure, $\gr\mathcal{C}$ has a \textit{negative t-structure} such that $X_\star\in(\gr\mathcal{C})^{-}_{\ge0}$ precisely when $X_{n}\in \mathcal{C}_{\ge-n}$ for all $n\in\mathbb{Z}$. The preferred t-structure on $\dg_-\mathcal{C}$ is obtained by transferring the negative t-structure on $\gr\mathcal{C}$ along $[-*]\circ u$, where the \textit{shearing functor} $[-*]:\gr\mathcal{C}\to \gr\mathcal{C}$ sends $X_{\star}$ to $X_{\star}[-*]$. In particular, $(\dg_-\mathcal{C})^{\heartsuit}$ is exactly the 1-category of cochain complexes in $\mathcal{C}^{\heartsuit}$.
	\begin{theorem}[Theorem 4.7, Proposition 6.9 and 6.13\cite{ariotta2021coherent}]\label{thm dg_-}
		
		Let $\mathcal{C}$ be a stable \infcat\ with sequential limits. There is an equivalence of \infcats\ $\fil^{\cpl}\mathcal{C}\xrightarrow{\simeq}\dg_{-}\mathcal{C}$ that maps $F_\star X$ to a homotopy-coherent cochain complex
		\[\ldots\xrightarrow{d} \gr_{n-1}X[n-1]\xrightarrow{d} \gr_{n}X[n]\xrightarrow{d}\gr_{n+1}X[n+1]\xrightarrow{d} \ldots,\]where the underlying map of $d$ are given by the cofibre sequences $F_{n-1}X/F_{n+1}X\to\gr_{n-1}X\to \gr_{n}X[1]$.
		Moreover, if $\mathcal{C}$ is presentably symmetric monoidal, then this equivalence induces a symmetric monoidal structure on $\dg_-\mathcal{C}$ and $[-*]\circ u:\dg_-\mathcal{C}\to \gr\mathcal{C}$ is symmetric monoidal.

		Additionally, if $\mathcal{C}$ is endowed with a compatible t-structure, this equivalence identifies the Beilinson t-structure on $\fil^{\cpl}\mathcal{C}$  and the above t-structure on $\dg_{-}\mathcal{C}$.\ The heart $(\fil^{\cpl}\mathcal{C})^{Bei,\heartsuit}$ of Beilinson t-structure is symmetric monoidally equivalent to the 1-category $(\dg_-\mathcal{C})^{\heartsuit}$ of cochain complexes in $\mathcal{C}^{\heartsuit}$ with the ordinary tensor product.
	\end{theorem}
	\begin{remark}\label{rk2.8: k[t]}
		For $\mathcal{C}\in\calg(\pr^{\st})$, there is another approach to $\dg_{-}\mathcal{C}$ using comodules in \cite{Raksit}. Let $1_{\mathcal{C}}[t]:=1_{\mathcal{C}}[\mathbb{N}]$ be the algebra whose multiplication is given by the commutative monoid $\mathbb{N}$. There is a symmetric monoidal equivalence $\fil\mathcal{C}\simeq \mathrm{LMod}_{1_{\mathcal{C}}[t]}(\gr\mathcal{C})$ \cite[3.4]{Mou} identifying filtered objects with modules over $1_{\mathcal{C}}[t]$. Denote $\dD^\vee_-:=1_{\mathcal{C}}\otimes_{1_{\mathcal{C}}[t]}1_{\mathcal{C}}$, the graded bicommutative algebra given by bar construction. Further, Theorem 3.2.14 in \cite{Raksit} shows that there is a factorization of $\gr$
		\[\begin{tikzcd}[column sep=2cm]
			{\fil^{\cpl}\mathcal{C}} & {\mathrm{LComod}_{\dD^\vee_-}(\gr\mathcal{C})} \\
			& {\gr\mathcal{C}}
			\arrow[from=1-2, to=2-2]
			\arrow["{1_{\mathcal{C}}\otimes_{1_{\mathcal{C}}[t]}-}", from=1-1, to=1-2]
			\arrow["\gr"', from=1-1, to=2-2]
		\end{tikzcd}\]that induces an equivalence $\fil^{\cpl}\mathcal{C}\simeq \mathrm{LComod}_{\dD^\vee_-}(\gr\mathcal{C})$ in $\calg(\pr^{\st})$. Besides, the method of \cite[Theorem 7.8]{ariotta2021coherent} provides an equivalence of these two approaches. In particular, if $\mathcal{C}=\m_k$ for some cdga over $\mathbb{Q}$, $\dg_-\m_k$ has an explicit model given by \textit{graded mixed complexes}, refer to \cite[\S1.1]{PTVV} \cite[\S1]{TV} for details and applications in derived geometry.
	\end{remark}

	Many presentable stable \infcats\ appear as the module categories of certain additive \infcats. Recall that an additive $\infty$-category is an $\infty$-category which admits finite products and finite coproducts and whose homotopy category is an additive category in the classical sense.
	\begin{df}\label{df2.6}
		Let $\mathscr{A}$ be an additive $\infty$-category. The $\infty$-category of \textit{(left) $\mathscr{A}$-modules} is defined as $\m_{\mathscr{A}}:=\text{Fun}_{\oplus}(\mathscr{A}^{op},\Sp)$, the \infcat\ of functors into $\Sp$ that preserve finite direct sums.
	\end{df}
	\begin{remark}\label{rk2.7}
		By Remark C.1.5.9 in \cite{SAG}, the module category $\m_{\mathscr{A}}$ can alternatively be defined as the stablization of $\mathcal{P}_{\Sigma}(\mathscr{A})$, where $\mathcal{P}_{\Sigma}(-)$ is taking the nonabelian derived category \cite[\S5.5.8]{HTT}. Thus the stablized Yoneda functor induces a fully faithful embedding $j:\mathscr{A}\to \m_{\mathscr{A}}$. Additionally, the functor $j$ exhibits $\m_{\mathscr{A}}$ as the initial presentable stable $\infty$-category receiving a functor from $\mathscr{A}$ that preserves finite direct sums  \cite[Proposition 4.8.2.18]{HA}.
	\end{remark}
	The objects in $\mathscr{A}$ with splitting filtration are building blocks to build filtered modules:
	\begin{notation}\label{n2.8}
		Let $\mathscr{A}$ be an additive \infcat\, and $I$ be one of the following integer intervals: $(-\infty,\infty)$, $(-\infty,-a]$ or $[a,\infty)$. Let $\sgr_I\mathscr{A}\subset \gr\mathscr{A}$ be the full subcategory consisting of all $X$ such that $X_i\simeq 0$ for all but finitely many $i\in I$. Define $\sfil_I\mathscr{A}$ as the full subcategory of $\fil_I\mathscr{A}\subset\fil_I\m_{\mathscr{A}}$ spanned by all complete $X$ such that $\gr(X)\in \sgr_I\mathscr{A}$.
	\end{notation}
	\begin{exa}\label{exa2.9}
		Let $R$ be a connective $\mathbb{E}_1$-ring spectrum, and $I$ be as in Notation \ref{n2.8}. The \infcat\ $\vect^{\omega}_R\subset\m_R$ of finitely generated free left $R$-modules is additive. In fact, the ordinary inclusion exhibits $\m_R$ as the module category of $\vect^{\omega}_R$. Moreover, Proposition \ref{prop2.14} will show that the following inclusions\[\sfil_I\vect^\omega_{R}\hookrightarrow\fil_I\m_R,\ \ \ \ \sgr_I\vect^\omega_{R}\hookrightarrow\gr_I\m_R,\]
		can be identified with the stablized Yoneda functors $\sfil_I\vect^\omega_{R}\hookrightarrow\m_{\sfil_I\vect^\omega_{R}}$ and $\sgr_I\vect^\omega_{R}\hookrightarrow\m_{\sgr_I\vect^\omega_{R}}$.
	\end{exa}

	The module category $\m_\mathscr{A}$ is equiped with a t-structure such that the \textit{(co)connective} modules are the functors $\mathscr{A}^{op}\to \Sp$ taking value in $\Sp_{\ge 0}$ ($\Sp_{\le 0}$).
	\begin{df}\label{df2.10}
		Let $\mathscr{A}$ be an additive $\infty$-category. An $\mathscr{A}$-module $M$ is said to be

		(1) \textit{perfect} if it is a compact object in $\m_{\mathscr{A}}$;

		(2) \textit{almost perfect} if for each $n$, there exists an $n$-connected map $P_n\to M$ with $P_n$ being perfect;

		(3) \textit{coherent} if it is almost perfect and eventually coconnective.\\
		The corresponding full subcategories will be denoted as $\perf_{\mathscr{A}}$, $ \aperf_{\mathscr{A}}$ and $\coh_{\mathscr{A}}$ respectively.
	\end{df}
	The basic properties of almost perfect modules are listed in \cite[Proposition 7.2.4.11]{HA}, which holds for $\aperf_{\mathscr{A}}$ without changing the proof.
	\begin{df}\label{df2.11}
		An additive $\infty$-category $\mathscr{A}$ is said to be \textit{left coherent} if $\aperf_{\mathscr{A}}$ inherits a t-structure from $\m_{\mathscr{A}}$, and it is said to be \textit{right coherent} if $\mathscr{A}^{op}$ is left coherent. When $\mathscr{A}$ is both left and right coherent, we say it is \textit{coherent}.
	\end{df}
	\begin{exa}\label{exa2.12}
		Let $R$ be a connective $\mathbb{E}_1$-ring spectrum.\ The additive \infcat\ $\vect^\omega_R$ is coherent precisely when $R$ is coherent as a ring spectrum in the sense of \cite[Definition 7.2.4.16]{HA}.
	\end{exa}
	
	The perfect filtered modules can be detected from their graded pieces. For each $\mathcal{C}\in\pr^{\st}$, write the full subcategory of compact objects as $\mathcal{C}^\omega$. Let $\mathcal{C}_{\sgr}^\omega\subset \gr\mathcal{C}$ denote the full subcategory consisting of all $X$ such that $X_i\simeq 0$ for all but finitely many $i\in\mathbb{Z}$ and $X_i\in \mathcal{C}^\omega$ for each $i$, and $\mathcal{C}_{\sfil}^\omega$ be the full subcategory of $\fil\mathcal{C}$ spanned by complete $X$ such that $\gr(X)\in \mathcal{C}_{\sgr}^\omega$.
	\begin{lemma}\label{lemma2.13}
		For any $\mathcal{C}\in\pr^{\st}$, there are natural equivalences of \infcats
		\[\text{Ind}(\mathcal{C}_{\sfil}^\omega)\xrightarrow{\simeq}\fil\mathcal{C},\ \ \ \ \text{Ind}(\mathcal{C}_{\sgr}^\omega)\xrightarrow{\simeq}\gr\mathcal{C},\]which identitfy $\mathcal{C}_{\sfil}^\omega$ (or $\mathcal{C}_{\sgr}^\omega$) with the \infcat\ of compact filtered (resp. graded) objects in $\mathcal{C}$.
	\end{lemma}
	\begin{proof}
		Set $\fil_{[m,n]}\mathcal{C}:=\fil_{\ge m}\mathcal{C}\cap \fil_{\le n}\mathcal{C}$ and $\mathcal{C}_{\sfil}^{\omega,[m,n]}:=\mathcal{C}_{\sfil}^\omega\cap \fil_{[m,n]}\mathcal{C}$. For any finite $m,n$, \cite[Proposition 5.3.5.15]{HTT} shows that there is an equivalence
		$\text{Ind}(\mathcal{C}_{\sfil}^{\omega,[m,n]})\simeq \fil_{\ge m}\mathcal{C}\cap\fil_{\le n}\mathcal{C}$. Meanwhile, the construction $\text{Ind}(-)$ preserves colimits \cite[Proposition 5.5.7.10]{HTT}, which means that $\text{Ind}(\mathcal{C}_{\sfil}^\omega)\simeq \text{Ind}(\colim_{m,n}\mathcal{C}_{\sfil}^{\omega,[m,n]})\simeq \colim_{m,n}\fil_{[m,n]}\mathcal{C}\simeq\fil\mathcal{C}$. The proof is done by mutatis mutandis for the graded case.
	\end{proof}
	
	\begin{prop}\label{prop2.14}
		Take an additive \infcat\ $\mathscr{A}$. There are natural equivalences in $\pr^{\st}$
		\[\m_{\sfil_I\mathscr{A}}\xrightarrow{\simeq}\fil_I\m_{\mathscr{A}},\ \ \ \ \m_{\sgr_I\mathscr{A}}\xrightarrow{\simeq}\gr_I\m_{\mathscr{A}},\]where $I$ is one of the following integer intervals: $(-\infty,\infty)$, $(-\infty,-a]$ or $[a,\infty)$.
	\end{prop}
	\begin{proof}
		The inclusion $\sfil\mathscr{A}\subset \fil\m_{\mathscr{A}}$ factors through $(\m_{\mathscr{A}})_{\sfil}^{\omega}$, and the latter is the smallest thick full subcategory containing $\sfil\mathscr{A}$. Thus, the functor $\m_{\sfil_I\mathscr{A}}\to\fil_I\m_{\mathscr{A}}$ given by the universal property is an equivalence by Lemma \ref{lemma2.13} and \cite[Proposition 5.4.2.4]{HTT}. The other cases are proved in the same way.
	\end{proof}
	Now we characterize the almost perfect filtered or graded modules.
	\begin{prop}\label{prop2.15}
		Suppose that $\mathscr{A}$ is a coherent additive \infcat. 
		Fix $I=(-\infty,\infty)$, $(-\infty,-a]$ or $[a,\infty)$. The \infcat\ $\sgr_I\aperf_{\mathscr{A}}:=\aperf_{\gr_I\mathscr{A}}$ is spanned by all $X\in\gr_I\m_{\mathscr{A}}$ such that $\oplus_i X_i$ lies in $\aperf_{\mathscr{A}}$, while the \infcat\ $\sfil_I\aperf_{\mathscr{A}}:=\aperf_{\sfil_I\mathscr{A}}$ is spanned by all complete $Y\in \fil_I\m_{\mathscr{A}}$ such that $\gr(Y)\in\sgr_I\aperf_{\mathscr{A}}$. .
		
		In particular, the additive \infcats\ $\sfil_I\mathscr{A}$ and $\sgr_I\mathscr{A}$ are coherent as well.
	\end{prop}
	\begin{proof}
		The graded case is straightforward from Lemma \ref{lemma2.13} and (2) of \cite[Proposition 7.2.4.11]{HA}.
		
		For each $Y\in\sfil_I\aperf_{\mathscr{A}}$, (up to a shifting) there is a simplicial object $Y^{\bullet}$ in $\sfil_I\mathscr{A}$ such that $|Y^\bullet|\simeq Y$, as shown in (5) of \cite[Proposition 7.2.4.11]{HA}.\ Then, since $\gr Y^\bullet\in \sgr_*\mathcal{A}$, $\gr Y\simeq |\gr Y^{\bullet}|$ is almost perfect. For completeness, it suffices to note that each term of $Y^\bullet$ is complete, and $\pi_n(Y)\cong \pi_n(Y^{\bullet}_{\le n+1})$ with $Y^{\bullet}_{\le n+1}$ being the $(n+1)$-skeletal object.
		
		Conversely, take any completely filtered $Y$ with $\gr(Y)\in\sgr_I\aperf_\mathscr{A}$. The truncation $\tau_{\le n}Y$ for each $n$ is a finite sequence of coherent modules, which is then almost perfect. This implies that $Y$ itself is almost perfect as well. 
	\end{proof}
	
	\begin{exa}\label{exa2.16}
		If $R$ is a coherent connective $\mathbb{E}_1$-ring spectrum, then $\sfil_I\vect^\omega_R$ and $\sgr_I\vect^\omega_R$ are coherent.
	\end{exa}
	We close this subsection with a recollection of spectral Mackey functors, which are useful for constructing (filtered) derived PD $\infty$-operads. Let $G$ be a finite group. Its \textit{effective Burnside \infcat}\ $A^{\mathrm{eff}}(G)$ has finite $G$-sets as objects, and the morphism spaces are the $\eoo$-spaces of spans $X\leftarrow Z\rightarrow Y$ of finite $G$-sets \cite[3.6]{bar1}. The effective Burnside \infcat\ $A^{\mathrm{eff}}(G)$ is pointed by the empty set, and it admits direct products calculated by the sums of $G$-sets. In the light of \cite[2.15]{bar2}, the product of $G$-sets gives rise to a symmetric monoidal structure on $A^{\mathrm{eff}}(G)$. Additionally, each span $X\leftarrow Z\rightarrow Y$ can be regarded as a point in $\map_{A^{\mathrm{eff}}(G)}(X,Y)$ as well as $\map_{A^{\mathrm{eff}}(G)}(Y,X)$. This induces a natural equivalence $A^{\mathrm{eff}}(G)^{op}\simeq A^{\mathrm{eff}}(G)$.
	\begin{df}\label{df2.18}
		The \infcat\ $\Sp^G$ of \textit{spectral Mackey functors of $G$} or \textit{genuine $G$-spectra} 
		\[\Sp^G:=\text{Fun}^\oplus(A^{\mathrm{eff}}(G)^{op},\Sp)\subset\text{Fun}(A^{\mathrm{eff}}(G)^{op},\Sp)\]is defined as the full subcategory of direct-product-preseriving functors from $A^{\mathrm{eff}}(G)$ to $\Sp$.
	\end{df}
	The \textit{Burnside \infcat\ $A(G)$} of $G$ is the universal additive \infcat\ receiving a functor from $A^{\mathrm{eff}}(G)$ that preserves finite direct sums. The \infcat\  A(G) has finite $G$-sets as objects, and the mapping space $\map_{A(G)}(X,Y)$ is the group completion $\Omega B\map_{A^{\mathrm{eff}}(G)}(X,Y)$. There are natural equivalences $\Sp^G\simeq \m_{A(G)}$ and $\Sp^G_{\ge0}\simeq \mathcal{P}_{\Sigma}(A(G))$ by the universal properties. Additionally, the adjunction of inclusion and truncation
	\[i:\mathcal{P}_\Sigma(A(G))\rightleftarrows \Sp^G:\tau_{\ge0}\]can be identified with the adjunction of Mackey stablization \cite[7.1]{bar1}. For this reason, we write the spectral Mackey functor \textit{represented by some finite $G$-set $X$} as $\infsus X$.
	
	Following \cite[3.8]{bar2}, the \infcat\ $\Sp^G$ is also endowed with a symmetric monoidal structure. For each finite $G$-set $X$, the genuine $G$-spectrum $\sig^{\infty}_+ X$ has itself as a dual in $\Sp^G$, and the coevaluation and evaluation are given by the following spans
	\[*\leftarrow X\xrightarrow{\text{diag}}X\times X,\ \ \ \ X\times X\xleftarrow{\text{diag}} X\rightarrow*.\]
	\begin{notation}\label{n2.19}
		Let $G$ be a finite group and $A$ an abelian group. The \textit{constant Mackey functor $\ul{A}$} sends finite $G$-set $X$ to $\map(X/G,A)$, the abelian group of $A$-valued function on the orbit set. Given a span $f$ of finite $G$-sets $X\leftarrow Z\rightarrow Y$, $\ul{A}(f)$ is the composite of the pullback along $Z\to X$ and the summation on the fibres of $Z\rightarrow Y$. This assignment is promoted to a sifted-colimit-preserving functor $\m_{\mathbb{Z},\ge0}\to \Sp^G$ by left Kan extension, which is lax symmetric monoidal \cite[Lemma 2.14]{BCN}.
	\end{notation}
	\begin{df}\label{df2.20}
		Given a connective $\mathbb{E}_1$-ring spectrum $R$ over $\mathbb{Z}$, by $\m^G_{\ul{R}}$ denote \textit{the \infcat\ of left $\ul{R}$-modules} in $\Sp^G$. The \infcat\ $R[\mathcal{O}_G]$ of \textit{finitely generated free $\ul{R}$-modules} is the full subcategory of $\m^G_{\ul{R}}$ consisting of the objects of the form $\ul{R}\otimes \infsus X$, where $X$ is a finite $G$-set.
	\end{df}
	\begin{remark}\label{rk2.21}
		The objects of $R[\mathcal{O}_G]$ projectively generate $\m^G_{\ul{R},\ge 0}$. There are equivalences \[\m_{R[\mathcal{O}_G]}\xrightarrow{\simeq} \m^G_{\ul{R}},\ \ \ \ \m_{\sfil_I R[\mathcal{O}_G]}\simeq \fil_I\m^G_{\ul{R}},\ \ \ \ \m_{\sgr_I R[\mathcal{O}_G]}\simeq \gr_I\m^G_{\ul{R}}.\]Supposing that $R$ is coherent, the additive \infcats\ $R[\mathcal{O}_G]$, $\sfil_I R[\mathcal{O}_G]$ and $\sgr_I R[\mathcal{O}_G]$ are all coherent due to \cite[Lemma 2.11]{BCN} and Proposition \ref{prop2.14}.
	\end{remark}
	\subsection{Pro-coherent modules and right-left extension}\label{sec2.2}
	Now we sketch out the framework of pro-coherent modules, following \cite{BM}\cite[\S2]{BCN}, which provides us with a good duality theory of almost perfect modules. We also emphasize its compatibility with filtrations.
	\begin{df}\label{df2.21}
		For a coherent additive \infcat\  $\mathscr{A}$, the \infcat\ of \textit{(left) pro-coherent $\mathscr{A}$-modules} is defined as $\QC^\vee_{\mathscr{A}}:=\text{Ind}(\coh^{op}_{\mathscr{A}^{op}})$.
	\end{df}
	\begin{remark}\label{rk2.22}
		(1) A functor $F:\aperf_{\mathscr{A}^{op}}\to \Sp$ is said to be \textit{convergent} if $F(X)\simeq \lim_nF(\tau_{\le n}X)$ holds for every $X\in\aperf_{\mathscr{A}^{op}}$. The \infcat\ $\QC^\vee_{\mathscr{A}}$ can be identified with $\text{Fun}_{ex,conv}(\aperf_{\mathscr{A}^{op}},\Sp)$, the \infcat\ of covergent exact functors. Following this, $\QC^\vee_{\mathscr{A}}$ carries a left complete, accessible $t$-structure such that $M\in\QC^\vee_{\mathscr{A}}$ is connective if and only if the associated exact convergent functor $M:\aperf_{\mathscr{A}^{op}}\to \Sp$ is right exact.\\
		(2) The evaluation functor $\aperf_{\mathscr{A}^{op}}\times \mathscr{A}\to \Sp$ induces an adjunction
		\[\iota: \m_{\mathscr{A}}\rightleftarrows \QC^\vee_{\mathscr{A}}:\nu,\]which exhibits $\m_{\mathscr{A}}$ as a right completion of $\QC^\vee_{\mathscr{A}}$\cite[Proposition 2.26]{BCN}.
	\end{remark}
	\begin{exa}\label{exa2.23}
		Set $\QC^\vee_R:=\QC^\vee_{\vect^\omega_R}$ for $R$, a coherent connective $\mathbb{E}_1$-ring spectrum. By Lemma \ref{lemma2.13} and Proposition \ref{prop2.15}, there are natural categorical equivalences
		\[\QC^\vee_{\sfil_I\vect^\omega_R}\simeq \fil_I\QC^\vee_R,\ \ \ \ \QC^\vee_{\sgr_I\vect^\omega_R}\simeq \gr_I\QC^\vee_R.\]When $R$ is eventually coconnective, $\perf_{R^{op}}$ can be identified with a full subcategory of $\coh_{R^{op}}$, then $\iota:\m_R\hookrightarrow\QC^\vee_R$ is fully faithful. The same thing holds in filtered or graded context as well.
	\end{exa}
	\begin{exa}\label{exa2.24}
		Take a coherent connective $\mathbb{E}_1$-algebra $R$ over $\mathbb{Z}$ and a finite group $G$. There is
		\[\QC^\vee_{\sfil_I R[\mathcal{O}_G]}\simeq \fil_I\QC^\vee_{R[\mathcal{O}_G]},\ \ \ \ \QC^\vee_{\sgr_I R[\mathcal{O}_G]}\simeq \gr_I\QC^\vee_{R[\mathcal{O}_G]}.\] 
		Observe that for finite $G$-sets $X$ and $Y$, the mapping space between the free $\ul{R}$-modules
		\begin{align*}
			\map_{\m^G_{\ul{R}}}(\ul{R}\otimes \infsus X,\ul{R}\otimes \infsus Y)\simeq R^{(X\times Y)/G}	\end{align*}
		is a product of finite copies of $R$. As a result, if $R$ is again eventually coconnective, there is an inclusion $\perf_{R[\mathcal{O}_G]^{op}}\hookrightarrow \coh_{R[\mathcal{O}_G]^{op}}$. It implies that the functor $\iota:\m^G_{\ul{R}}\to \QC^\vee_{{R[\mathcal{O}_G]}}$ is fully faithful.
		If $R$ is discrete, $R[\mathcal{O}_G]$ is equivalent with the full additive subcategory of ordinary $R[G]$-modules spanned by $R[G/H]$ for $H<G$.
	\end{exa}
	
	Right-left extension is the fundamental method of constructing operations in pro-coherent context. Here we quote some useful results from \cite{BM} \cite{BCN} for the convenience of readers:
	
	\begin{df}\label{df2.25}
		Let $\mathscr{A}$ be a coherent additive \infcat. The \infcat\ $\aperf^\vee_{\mathscr{A}}$ of \textit{dually almost perfect $\mathscr{A}$-modules} is defined as the essential image of the Yoneda embedding $\aperf_{\mathscr{A}^{op}}^{op}\hookrightarrow \QC^\vee_{\mathscr{A}}$. By $\aperf_{\mathscr{A},\eqslantless 0}^\vee$ denote the full subcategory of $\aperf^\vee_{\mathscr{A}}$ consisting of the functors co-represented by some connective $X\in \aperf^{op}_{\mathscr{A}^{op}}$. 
	\end{df}
	Every $F:\aperf^\vee_{\mathscr{A},\eqslantless0}\to \mathcal{V}$ is equivalent to $F^{op}:\aperf_{\mathscr{A},\ge0}\simeq \aperf^{\vee,op}_{\mathscr{A},\eqslantless0}\xrightarrow{F}\mathcal{V}^{op}$. Here, $F$ is said to be \textit{regular} if $F^{op}$ is convergent. We recall Proposition 2.40 in \cite{BCN}:
	
	\begin{prop}\label{prop2.29}
		Let $\mathscr{A}$ be a coherent additive \infcat\ and $\mathcal{V}$ an \infcat\ with sifted colimits. Then the restriction induces an equivalence \[\mathrm{Fun}_{\Sigma}(\QC^\vee_{\mathscr{A}},\mathcal{V}) \xrightarrow{\simeq} \mathrm{Fun}_{\sigma,reg}(\aperf^\vee_{\mathscr{A},\eqslantless0},\mathcal{V}),\]where the right-hand side consists of \textit{regular functors} that preserves finite geometric realizations, and the inverse is left Kan extension.
		
	\end{prop}
	\begin{remark}\label{rk2.29}
		Let $\text{End}^{\dap,\eqslantless0}_{\Sigma}(\QC^\vee_{\mathscr{A}})$ be the \infcat\ of sifted-colimit-preserving endofunctors that preserve $\aperf^\vee_{\mathscr{A},\eqslantless0}$. The former proposition induces a strictly monoidal equivalence \[\text{End}^{\dap,\eqslantless0}_{\Sigma}(\QC^\vee_{\mathscr{A}})\simeq\text{End}_{\sigma,reg}(\aperf^\vee_{\mathscr{A},\eqslantless0}),\]which is crucial for the construction of the partition Lie algebroid monad.
	\end{remark}
	Many sifted-colimit-preserving functors are induced from \textit{locally polynomial functors} using the above proposition and left Kan extension.
	
	\begin{df}
		A functor $F:\mathscr{A}\to \mathscr{B}$ between two additive $\infty$-categories is said to be
		
		(1) \textit{of degree $\le0$ }if $F$ is constant;
		
		(2) \textit{of degree $\le n$} if the difference  $\text{fib}(F(X\oplus -)\to F(-))$ is of degree $\le (n-1)$ for any $X\in\mathscr{A}$;
		
		(3) \textit{locally polynomial} if $F$ arises as the colimit of a sequence $F_1\to F_2\to \ldots$ of functors in $\text{Fun}(\mathscr{A},\mathscr{B})$, such that each $F_i$ is of finite degree, and that the sequence $F_1(X)\to F_2(X)\to\ldots$ is eventually constant for any $X\in \mathscr{A}$.
	\end{df}
	Given a locally polynomial functor $F:\mathcal{A}\to \mathcal{B}\subset\QC^\vee_\mathcal{B}$, right Kan extension gives rise to $F^R\in\text{Fun}_{\sigma,reg}(\aperf^\vee_{\mathscr{A},\eqslantless0},\QC^\vee_{\mathscr{B}})$ \cite[Proposition 2.46]{BCN}.\ Then, left Kan extension gives rise to a sifted-colimit-preserving functor $F^{RL}:\QC^\vee_{\mathscr{A}}\to \QC^\vee_{\mathscr{B}}$ by Proposition \ref{prop2.29}, which is called the \textit{right-left derived functor} of $F$.
	
	Right-left extension is a functorial procedure. Write $\mathcal{A}dd^{\text{coh,poly}}\subset \mathcal{C}at_\infty$ the \infcat\ of coherent additive $\infty$-categories and locally polynomial functors, and $\mathcal{P}r^{\text{st},\Sigma}\subset \mathcal{C}at_\infty$ the \infcat\ of presentable stable $\infty$-categories and sifted-colimit-preserving functors (strictly larger than $\pr^{\st}$). Then, we recall Theorem 2.52 in \cite{BCN}:
	\begin{theorem}\label{theorem2.29}
		There is a natural transformation of symmetric monoidal functors
		\[\begin{tikzcd}
			{\mathcal{A}dd^{\coh,\mathrm{poly}}} && {\mathcal{P}r^{\st,\Sigma}}
			\arrow[""{name=0, anchor=center, inner sep=0}, "\m", shift left=1, bend left, from=1-1, to=1-3]
			\arrow[""{name=1, anchor=center, inner sep=0}, "{\QC^\vee}"', shift right=1, bend right, from=1-1, to=1-3]
			\arrow["\iota", shorten <=2pt, shorten >=2pt, Rightarrow, from=0, to=1]
		\end{tikzcd}\]sending each locally polynomial functor to its right-left derived functor.
	\end{theorem}
	Proposition 2.55 in \cite{BCN} gives a good duality theory of almost perfect modules:
	\begin{prop}\label{prop2.30}
		Let $\mathscr{A}$ be a coherent additive \infcat\ endowed wih a non-unital symmetric monoidal structure $\otimes$ preserving finite direct sums. Assume that the induced $\otimes$ on $\QC^\vee_{\mathscr{A}}$ admits an eventually connective unit $\mathbbb{1}$, and each object in $\mathscr{A}$ is dualizable with the dual object existing in $\mathscr{A}$. Then taking duals gives rise to an equivalence,
		\[(-)^\vee:=\hom_{\QC^\vee_{\mathscr{A}}}(-,\mathbbb{1}):\aperf_{\mathscr{A}}\xrightarrow{\simeq} \aperf^{\vee,op}_{\mathscr{A}}.\]
		
	\end{prop}
	\begin{exa}\label{exa2.33}
		Consider a coherent additive \infcat\ $\mathscr{A}$, e.g.\ $\vect^\omega_R$ or $R[\mathcal{O}_G]$, the functors $(-)_n$, $[-]_n$, $F_n$ and $\gr$ are extended to the pro-coherent context. The extended functors coincide with those obtained directly from filtration and grading by Example \ref{exa2.23} or \ref{exa2.24}.
	\end{exa}
	\begin{exa}\label{exa2.34}
		Fix a finite group $G$ and a coherent connective $\mathbb{E}_1$ ring $R$ over $\mathbb{Z}$. The\textit{ genuine $G$-fixed points functor} $(-)^G:\m^G_{\ul{R}}\to \m_R$ is evaluating at $*\cong G/G$. It admits a left adjoint $\triv_G:\m_R\to \m^G_{\ul{R}}$, which is defined by $\triv_G(M)\simeq \ul{R}\otimes_R M$. One can check pointwisely that $\triv_G$ preserves small limits. Then it has a left adjoint $(-)_G$ from the adjoint functor theorem, which is addressed as \textit{derived $G$-orbits functor}. The three functors can be extended to pro-coherent context as shown in  \cite[Example 2.51]{BCN}, and the right-left extension is compatible with filtration and grading.
		
	\end{exa}
	\begin{exa}\label{exa2.36}
		Let $\mathscr{O}^\otimes\to N(\mathcal{F}in_*)$ be an \ooop\ in the sense of  \cite[Definition 2.1.1.10]{HA}, and $\mathscr{A}$ be an $\mathscr{O}$-algebra object in $\mathcal{A}dd^{\text{coh,poly}}$. Then one obtains the following $\mathscr{O}$-monoidal functors
		\[\QC^\vee_\mathscr{A}\xrightarrow{(-)_0}\fil\QC^\vee_\mathscr{A}\xrightarrow{(-)^\wedge}\fil^{\cpl}\QC^\vee_\mathscr{A}\xrightarrow{\gr}\gr\QC^\vee_\mathscr{A}.\]
		For example, taking $\mathscr{A}$ as $\vect^\omega_R$ for $R$ a coherent connective $\eoo$-ring spectrum and $\mathscr{O}^\otimes$ as $N(\mathcal{F}in_*)$, we obtain symmetric monoidal structure on different pro-coherent $R$-module categories.
	\end{exa}

	This section is closed with a criterion for a given $V\in \fil_I\QC^\vee_R$ being presented by an ordinary module. We use it to transfer our results in the pro-coherent context back to the usual context.
	\begin{prop}\label{prop2.38}
		Assume that $R$ is an eventually coconnective coherent $\mathbb{E}_1$-ring spectrum, so that $\fil\m_{R}$ can be regarded as a full subcategory of $\fil\QC^\vee_R$ by the embedding $\iota$. Then a filtered pro-coherent $R$-module $V$  lies in $\fil\m_{R}$ if it is (filtration) complete and $\gr(V)$ lies in $\gr\m_{R}$
	\end{prop}
	\begin{proof}
		Since $\iota$ exhibits $\fil\m_{R}$ as a right completion of $\fil\QC^\vee_R$, it suffices to prove that the natural morphism
		$\iota\nu V\simeq \colim_i \tau_{\ge i}V\to V$ in $\fil\QC^\vee_R$ is an equivalence.
		
		First note that, using Milnor exact sequence, $\tau_{\ge i}V$ is complete and eventually stable on each $\pi_n$ varying $i$, which means that $\colim_i \tau_{\ge i}V$ is also complete. Next, we prove that $\gr\iota\nu V\simeq \gr V$. For each $n$, there is a cofibre sequence of exact functors
		$F_{n+1}\to F_{n}\to \gr_n$, which are the right-left derived functors of their restrictions to $\fil\vect^\omega_R$. Here $F_n$ commutes with $\tau_{\ge i}$. Applying $\colim_i$ on
		$\tau_{\ge i}F_{n+1}(V)\to \tau_{\ge i}F_{n}(V)\to \gr_n(\tau_{\ge i} V)$, there is $\colim_i\gr_n(\tau_{\ge i}V)\simeq \iota\mu\gr_nV\simeq \gr_nV$. Meanwhile, $\colim_i\gr_n(\tau_{\ge i}V)\simeq \gr_n(\colim_i\tau_{\ge i}V)\simeq \gr_n(\iota\nu V)$.
		
	\end{proof}

	\subsection{PD $\infty$-operads and filtered algebras}
	
	Operad theory is a general framework for studying various types of algebras with homotopy-coherent higher operations (e.g. $\eoo$-algebras or $L_\infty$-algebras) by describing all the $n$-ary operations and their relations\cite{loday2012algebraic}.\ The $\infty$-operads are a fully homotopy-coherent version of operads that parametrise algebras in general presentable symmetric monoidal \infcats, see \cite[\S4.1.2]{brantner2017lubin} or \cite{haugseng2022operads}. In this subsection, we briefly recall the basics of $\infty$-operads and then pass to their divided power analogues, \textit{spectral} or \textit{derived PD $\infty$-operads} (introduced in \cite[\S3]{BCN}), which permit us to define commutative algebras and the corresponding ``Lie algebras'' in pro-coherent modules. We emphasize the functoriality of the constructions to obtain the theory of filtered algebras via PD $\infty$-operads. 
	\begin{construction}[Recollection of symmetric sequences]\label{c2.35n}
		Let $\text{B}\Sigma:=(\fin^{\simeq},\amalg)$ be the category of finite sets and bijections, together with the symmetric monoidal structure given by disjoint unions. For any $\mathcal{E}\in \calg(\mathcal{P}r^{L})$, the \textit{$\infty$-category of $\mathcal{E}$-valued symmetric sequences} \[\sseq(\mathcal{E}):=\text{Fun}(\text{B}\Sigma,\mathcal{E})\] is endowed with a symmetric monoidal structure $\otimes$ by Day convolution.
		
		We signify by $\mathbbb{1}\in \sseq(\mathcal{E})$ the functor that sends $[1]\in B\sig$ to the unit of $\mathcal{E}$ and $[n]\in B\sig$ to the initial object for all $n\ne 1$. The symmetric monoidal \infcat\ $\sseq(\mathcal{E})$ is freely generated by $\mathbbb{1}$ in $\calg_{\mathcal{E}/}(\mathcal{P}r^{L})$. Thus, there is an equivalence $\mathrm{End}^{L,\otimes}_{\mathcal{E}/}(\sseq(\mathcal{E}))\xrightarrow{\simeq} \sseq(\mathcal{E})$ obtained by evaluating at $\mathbbb{1}$. The \textit{composition product} $\circ$ on $\sseq(\mathcal{E})$ is defined as the opposite of functor composition.
		
		There is another symmetric monoidal structure, denoted as $\levtimes$, on $\sseq(\mathcal{E})$ defined \textit{levelwise}. The composition product can be expressed by the following formula
		\begin{equation}\label{composite formula}
			M\circ N\simeq \mathop{\bigamalg}\limits_{r\in\mathbb{N}} (M(r)\otimes N^{\otimes r})_{h\Sigma_r}.
		\end{equation}
		This formula shows that the inclusion $\mathcal{E}\subset\sseq(\mathcal{E})$ into pure arity $0$ equips $\mathcal{E}$ with a left $(\sseq(\mathcal{E}),\circ)$-tensored structure. Additionally, those symmetric sequences $M$ with $M(0)$ being initial in $\mathcal{C}$ form a monoidal full subcategory $\sseq^{\ge1}(\mathcal{C})\subset\sseq(\mathcal{C})$ with respect to $\circ$.
		
		The assignments $(\sseq(-),\otimes)$ and $(\sseq(-),\levtimes)$ give rise to endofunctors of $\calg(\mathcal{P}r^{L})$, due to the universal property. Similarly, there is a functor $(\sseq(-),\circ):\calg(\mathcal{P}r^{L})\to \alg(\mathcal{P}r^{\sig})$.
	\end{construction}

	\begin{exa}\label{exa2.36n}
		When $\mathcal{E}$ equals to $\m_R$ for certain $\eoo$-ring spectrum $R$, the \infcat\ $\sseq_R:=\sseq(\m_R)$ consists of ordinary symmetric sequences of $R$-module spectra.\ By the functoriality of $\sseq(-)$, the  symmetric monoidal functors $(-)_0$ and $\gr$ induces a sequence of functors
		\[\sseq_R\to \sseq(\fil\m_R)\to \sseq(\gr\m_R),\]which is symmetric monoidal with respect to $\otimes$ and $\levtimes$, and monoidal with respect to $\circ$. In particular, the \infcats\ $\fil_I\m_R$ and $\gr_I\m_R$ are endowed with natural left $(\sseq^{\ge1}_R,\circ)$-tensored structure, where $I=(-\infty,\infty)$, $(-\infty,-a]$ or $[a,\infty)$ for $a\ge 0$. The functors of modules $(-)_0$, $[-]_0$ and $\gr_0$ are enhanced into $(\sseq^{\ge1}_R,\circ)$-linear functors.
		
		Similarly, the symmetric monoidal functors $F_0:\fil_{\ge0}\m_R\to \m_R$ and $\colim:\fil_{\le0}\m_R\to \m_R$ are endowed with natural $(\sseq_R,\circ)$-linear structures, and can be restricted to $(\sseq^{\ge1}_R,\circ)$-linear functors $F_1:\fil_{\ge1}\m_R\to\m_R$ and $\colim:\fil_{\le-1}\m_R\to \m_R$.
	\end{exa}
	\begin{df}\label{df2.37n}
		An \textit{$\infty$-operad over $R$} is an associative algebra object in $(\sseq_R,\circ)$. The $\infty$-category of $\infty$-operads over $R$ is written as $\op_R$. For a left $(\sseq_R,\circ)$-tensored \infcat\ $\mathcal{M}$ and an $\infty$-operad $\mathcal{P}$, \textit{a $\mathcal{P}$-algebra in $\mathcal{M}$} is a left $\mathcal{P}$-module object in $\mathcal{M}$.\ Let $\alg_{\mathcal{P}}(\mathcal{M})$ be the \infcat\ of $\mathcal{P}$-algebras in $\mathcal{M}$.
	\end{df}
	\begin{exa}\label{exa2.38n}
		The unit of $\levtimes$ in $\sseq_R$ is the $\infty$-operad $\eoo$ of (unital) commutative algebra, where $\eoo(i)\simeq R$ for each $i\in\mathbb{N}$. Similarly, the $\infty$-operad $\neoo$ of non-unital commutative algebra is the same with $\eoo$ at arity $\ge 1$, but $\neoo(0)\simeq 0$. Write $\alg_{\neoo}(-)$ as $\calg^{nu}(-)$ for short. There are adjunctions of $\neoo$-algebras following Example \ref{exa2.36n}
		\[\begin{tikzcd}
			{\calg^{nu}(\m_R)} & {\calg^{nu}(\fil_{\ge1}\m_R)} & {\calg^{nu}(\gr_{\ge1}\m_R)}
			\arrow["\adic", shift left=1, from=1-1, to=1-2]
			\arrow["{F_1}", shift left=1, from=1-2, to=1-1]
			\arrow["\gr", shift left=1, from=1-2, to=1-3]
			\arrow[shift left=1, from=1-3, to=1-2]
		\end{tikzcd},\]where $\adic$ is formally defined by the adjoint functor theorem. This diagram satisfies a pleasant property that $\gr\circ\adic(A)$ is freely generated by $\gr_1\circ \adic(A)$ for each $A\in \calg^{nu}_R(\m_R)$.
	\end{exa}
	Before diving into PD $\infty$-operads, it is worth recalling that commutative rings have two natural generalizations in derived geometry, $\eoo$-ring spectra and simplicial commutative rings. We start by reviewing the \textit{spectral PD $\infty$-operads} that are suitable for $\eoo$-ring spectra.
	\begin{construction}[$\mathscr{A}$-pro-coherent symmetric sequences]\label{c2.39n}
		Let $\mathscr{A}$ be an additive \infcat. For any $r\ge 0$, there is a free-forgetful adjunction of the permutation group $\sig_r$-actions
		\[(-)[\Sigma_r]:\mathscr{A}\rightleftarrows\text{Fun}(\Sigma_r,\mathscr{A}):\text{forget}\] such that $\text{forget}(A[\Sigma_r])\simeq A^{\oplus r}$ if $r\ge 1$ and $\text{forget}(A[\Sigma_0])\simeq A$. The full subcategory of $\text{Fun}(\Sigma_r,\mathscr{A})$ spanned by all the free $\Sigma_r$-modules is denoted as $\mathscr{A}[\Sigma_r]$.
		
		Next, assume that $\mathscr{A}$ is coherent and endowed with a symmetric monoidal structure $\otimes_{\mathscr{A}}$ which preserves finite direct sums in each variable. The module category $\m_{\mathscr{A}}$ is promoted into an object in $\calg(\mathcal{P}r^{L})$, which means that $\sseq_{\mathscr{A}}:=\sseq(\m_{\mathscr{A}})$ makes sense. Moreover, the \infcat\ $\mathscr{A}[\Sigma_r]$ can be naturally embedded into $\sseq_{\mathscr{A}}$ by regarding any object in $\mathscr{A}[\Sigma_r]$ as a symmetric sequence concentrating at arity $r$.
		
		Denote by $\mathscr{A}[\Sigma]$ the full subcategory $\oplus_r\mathscr{A}[\Sigma_r]\subset\sseq_{\mathscr{A}}$. It is coherent as an additive \infcat, because the t-structure of $\aperf_\mathscr{A}$ can be transferred to $\aperf_{\mathscr{A}[\sig_r]}$. The pro-coherent modules of $\mathscr{A}[\Sigma]$ are called \textit{$\mathscr{A}$-pro-coherent symmetric sequences}. Write $\QC^\vee_{\mathscr{A}[\Sigma]}$ as $\sseq^\vee_{\mathscr{A}}$.
		
		
	\end{construction}
	\begin{lemma}\label{lemma2.40n}
		The full subcategory $\mathscr{A}[\Sigma]\subset \sseq_{\mathscr{A}}$ is closed under the tensor product of Day convolution $\otimes$, the levelwise tensor product $\levtimes$ and the composition product $\circ$. Furthermore, $\otimes$ and $\levtimes$ are additive in each variable, and $\circ$ is additive in the first variable and locally polynomial.
	\end{lemma}
	\begin{proof}
		Note that $V[\Sigma_m]\otimes U[\Sigma_n]\simeq (V\otimes_{\mathscr{A}}U)[\Sigma_{m+n}]$. Thus $\mathscr{A}[\Sigma]$ is closed under $\otimes$. At the same time, if $m=n$, $V[\Sigma_n]\levtimes U[\Sigma_n]$ is a $n$ copies of $(V\otimes_\mathscr{A} U)[\Sigma_n]$, so $\mathscr{A}[\Sigma]$ is also closed under $\levtimes$. The closedness of $\mathscr{A}[\Sigma]$ under $\circ$ is now implied by the formula (1).
		
		The tensor products $\otimes$ and $\levtimes$ are additive in each variable by their nature. The formula (1) shows that $\circ$ is additive in the first variable. Finally, $\circ$ is locally polynomial by
		\[M\circ N\simeq \colim_n (\mathop{\amalg}_{0\le r\le n} (M(r)\otimes_{\mathscr{A}} N^{\otimes r})_{\Sigma_r}).\]
	\end{proof}
	Let $\mathcal{A}dd^{\text{coh},\oplus}$ be the \infcat\ of coherent additive \infcats\ with additive functors.
	\begin{prop}\label{prop2.41n}
		Construction \ref{c2.39n} gives rise to a functor $\sseq^\vee_{-}:\calg(\mathcal{A}dd^{\text{coh},\oplus})\to \mathcal{P}r^{\st}$, which is functorially endowed with: two closed symmetric monoidal structures $\otimes$ and $\levtimes$, and a monoidal structure $\circ$ preserving sifted colimits in each variable and small colimits in the first variable.
		
		Additionally, there is a lax and oplax monoidal structure on  $\levtimes:\sseq^\vee_{-}\times\sseq^\vee_{-}\to \sseq^\vee_{-}$ with respect to $\circ$.
	\end{prop}
	\begin{proof}

		The symmetric sequence categories $\sseq_{-}\simeq\m_{-[\Sigma]}$ are equiped with functorial operations from Construction \ref{c2.35n}.\ Applying Theorem \ref{theorem2.29} to them, the right-left extension produces the desired $\otimes$ and $\circ$, and a non-unital symmetric monoidal structure $\levtimes$. Besides, the transformation $\iota:\sseq_{-}\to \sseq^\vee_{-}$ sends $\eoo$ to a $\levtimes$-unit in $\sseq_{-}$.\ The lax and oplax structures on $\levtimes$ are given by Proposition 3.9 in \cite{BCN}, whose proof also works here without extra effort.
	\end{proof}
	The functoriality of $\sseq^\vee_{-}$ gives rise to natural actions of $\sseq^\vee_R$ on filtered and graded modules.

	\begin{notation}\label{n2.42n}
		Let $R$ be a coherent connective $\eoo$-algebra. Define the \infcat\ of \textit{pro-coherent symmetric sequences over $R$} as $\sseq^\vee_R:=\sseq^\vee_{\vect^\omega_R}$, which agrees with the definition in \cite{BCN}. Similarly, we call $\sseq_{\fil,R}^\vee:=\sseq^\vee_{\sfil\vect^\omega_R}$ (resp. $\sseq_{\gr,R}^\vee:=\sseq^\vee_{\sgr\vect^\omega_R}$) the \infcat\ of \textit{ filtered (resp. graded) pro-coherent symmetric sequences over $R$}. Besides, let $\sseq^{\vee,\ge1}_R\subset\sseq^\vee_R$ denote the full subcategory spanned by the pro-coherent symmetric sequences without nullary operations.
	\end{notation}
	\begin{construction}\label{c2.43n}
		Following Example \ref{exa2.36n}, the right-left derived functors
		\[\sseq^\vee_R\to \sseq_{\fil,R}^\vee\to\sseq_{\gr,R}^\vee\] are monoidal with respect to $\otimes$, $\levtimes$ and then $\circ$. This exhibits $\QC^\vee_R$, $\fil_I\QC^\vee_R$ and $\gr_I\QC^\vee_R$ as left $(\sseq^{\vee,\ge1}_R,\circ)$-tensored \infcats, where $I=(-\infty,\infty)$, $(-\infty,-a]$ or $[a,\infty)$ for $a\ge0$. Moreover, we have the $\sseq^{\vee,\ge1}_R$-linear functors $(-)_0$, $[-]_0$, $\gr$ and
		\[F_1:\fil_{\ge1}\QC^\vee_R\to \QC^\vee_R,\ \ \ \ \colim:\fil_{\le-1}\QC^\vee_R\to \QC^\vee_R.\]
		
	\end{construction}
	\begin{df}\label{df2.44n}
		A \textit{spectral PD $\infty$-(co)operad over $R$} is an (co)associative (co)algebra object in $(\sseq^\vee_R,\circ)$. The $\infty$-category of PD $\infty$-(co)operads over $R$ is $\op^{\pd}_R$ (resp. $co\op^{\pd}_R$).
		
		Let $\mathcal{M}$ be a left $(\sseq^\vee_R,\circ)$-tensored \infcat. For a PD $\infty$-(co)operad $\mathcal{P}$ ($\mathcal{Q}$), the \infcat\ $\alg_{\mathcal{P}}(\mathcal{M})$ ($co\alg_\mathcal{Q}(\mathcal{M})$) of \textit{$\mathcal{P}$-algebras ($\mathcal{Q}$-coalgebras) in $\mathcal{M}$} is defined as the \infcat\ of left $\mathcal{P}$-modules (resp. left $\mathcal{Q}$-comodules) objects in $\mathcal{M}$.
	\end{df}
	\begin{exa}\label{exa2.45n}
		The functor $\iota$ (\ref{rk2.22}) sends $\neoo$ to a PD $\infty$-operad (still denoted as $\neoo$). It induces a variant of Example \ref{exa2.38n}
		\[\begin{tikzcd}
			{\calg^{nu}(\QC^\vee_R)} & {\calg^{nu}(\QC^\vee_{\ge1}\m_R)} & {\calg^{nu}(\gr_{\ge1}\QC^\vee_R)}
			\arrow["\adic", shift left=1, from=1-1, to=1-2]
			\arrow["{F_1}", shift left=1, from=1-2, to=1-1]
			\arrow["\gr", shift left=1, from=1-2, to=1-3]
			\arrow[shift left=1, from=1-3, to=1-2]
		\end{tikzcd}.\]
	\end{exa}
	Next, we go to \textit{derived PD $\infty$-operads} is suitable for derived algebras.

	\begin{df}\label{df2.46}
		Let $R$ be a simplicial  commutative ring and $R[\mathcal{O}_\Sigma]:=\oplus_{r\ge 0}R[\mathcal{O}_{\Sigma_r}]$ (Definition \ref{df2.20}) be the sum additive \infcat. The \textit{\infcat\ of derived symmetric sequences} is defined as
		\[\sseq^{\gen}_{\ul{R}}:=\m_{R[\mathcal{O}_\Sigma]}\simeq \prod_{r\ge0}\m^{\Sigma_r}_{\ul{R}}.\]
		Similarly, define the \textit{\infcat\ of pro-coherent derived symmetric sequences} as
		$\sseq^{\gen,\vee}_{\ul{R}}:=\QC^\vee_{R[\mathcal{O}_\Sigma]}$, while the \infcats\ of \textit{filtered or graded} pro-coherent derived symmetric sequences refer to
		\[\sseq^{\gen,\vee}_{\fil,\ul{R}}:=\QC^\vee_{\sfil R[\mathcal{O}_\Sigma]}\simeq \fil\sseq^{\gen,\vee}_{\ul{R}},\ \ \ \ \sseq^{\gen,\vee}_{\gr,\ul{R}}:=\QC^\vee_{\sgr R[\mathcal{O}_\Sigma]}\simeq \gr\sseq^{\gen,\vee}_{\ul{R}}.\]
	\end{df}
	The following technical lemma is useful for constructing operations on $\sseq^{\gen}_R$ and its variants.
	\begin{lemma}\label{lemma2.47}\footnote{We give an alternative proof fixing  \cite[Lemma 3.59]{BCN} with the help of Nuiten.}
		The assignment $R\mapsto R[\mathcal{O}_\Sigma]$ determines a sifted-colimit-preserving functor $\mathrm{SCR}\to \mathcal{A}dd^{\oplus}$.
	\end{lemma}
	\begin{proof}
		
		It suffices to show that at each arity $r$, the assignment $R\mapsto R[\mathcal{O}_{\sig_r}]$ can be enhanced into a functor $f_r:\mathrm{SCR}\to \mathcal{A}dd^{\oplus}$ preserving sifted colimits. Fix $r$. Let $\mathfrak{S}_r$ be the set of isomorphism classes of finite $\sig_r$-sets. Observe that $R[\mathcal{O}_{\sig_r}]$ has $\mathfrak{S}_r$ as the set of objects for each $R\in \mathrm{SCR}$, and recall that the category $\mathcal{C}at^{\mathfrak{S}_r}_\Delta$ of small simplicial categories with fixed objects set $\mathfrak{S}_r$ admits a simplicial model structure from \cite[7.2]{dk}.\ Thus the assignment $R\mapsto R[\mathcal{O}_{\sig_r}]$ gives rise to a functor $F_1:\mathrm{SCR}\to N_\Delta(\mathcal{C}at^{\mathfrak{S}_r}_\Delta)$ of \infcats, where $N_\Delta(-)$ is taking the simplicial nerve.
		
		The functor $F_1$ is actually sifted-colimit-preserving. Indeed, recall that for each $R\in \mathrm{SCR}$ and $X,Y$ finite $\sig_r$-sets, the mapping space $\map_{R[\mathcal{O}_{\sig_r}]}(\ul{R}\otimes\infsus X,\ul{R}\otimes\infsus Y)\simeq R^{(X\times Y)/\sig_r}$ preserves sifted colimits with respect to $R$. At the same time, the \infcat\ $N_\Delta(\mathcal{C}at^{\mathfrak{S}_r}_\Delta)$ is naturally equivalent with $\alg_{\mathfrak{S}_r^{cat}}(\mathcal{S})$, where $\mathfrak{S}_r^{cat}$ is the $\mathfrak{S}_r\times \mathfrak{S}_r$-colored \ooop\ of catenary composition of ordered pairings in $\mathfrak{S}_r$ \cite[2.1]{GH}. Therefore, it  suffices to apply \cite[Proposition 3.2.3.1]{HA}.

		Write $\mathcal{C}at_\Delta$ Bergner's model category of small simplicial categories \cite[Theorem 1.1]{Bergner}, and $\mathcal{C}at^{\mathfrak{S}_r/}_\Delta$ the coslice model category of small simplicial categories receiving a functor from $\mathfrak{S}_r$ regarded as a discrete category. There is a left Quillen functor $\mathcal{C}at^{\mathfrak{S}_r}_\Delta\to \mathcal{C}at^{\mathfrak{S}_r/}_\Delta$. As a result, the functor
		\[F_2:\mathrm{SCR}\xrightarrow{F_1} N_\Delta(\mathcal{C}at^{\mathfrak{S}_r}_\Delta)\to N(\mathcal{C}at^{\mathfrak{S}_r/}_\Delta)\simeq (\mathcal{C}at_{\infty})_{N(\mathfrak{S}_r)/}\to\mathcal{C}at_{\infty}\] is sifted-colimit-preserving by Proposition 4.4.2.9 in \cite{HTT}.
		
		It only remains to prove that the inclusion $i:\mathcal{A}dd^{\oplus}\to \mathcal{C}at_{\infty}$ creates sifted colimits, since $F_2$ is equivalent with $i\circ f_r$. Write $\mathcal{C}at^{\pi}_{\infty}$ the (not full) subcategory of small \infcats\ admitting finite products with finite-product-preserving functors. The inclusion $i$ factors through $\mathcal{C}at^{\pi}_{\infty}$,
		\[i:\mathcal{A}dd\xrightarrow{j'}\mathcal{C}at^{\pi}_{\infty}\xrightarrow{j''} \mathcal{C}at_{\infty}\]where $j'$ is a colocalization \cite[Corollary 2.10]{GGN}, $j''$ is the inclusion. The question is reduced to showing that $j''$ creates sifted colimits.
		
		Take an arbitrary sifted diagram $\phi:K\to \mathcal{C}at^{\pi}_{\infty}$. The existence of finite products in each term can be internally reformulated as the following map of $K$-diagrams
		\[\phi\rightleftarrows*_K,\ \ \ \ \Delta:\phi\rightleftarrows\phi\times\phi:-\times-,\](where $*_K$ is the singleton valued constant diagram) together with the unit and counit transformations which satisfy the triangle identities, see \cite[\S1]{riehl2016homotopy}. In other words, termwise having finite products is encoded by a finite diagram in $\mathcal{C}at^{K}_{\infty}$ compatible with sifted colimits. Thus the colimit of $j''\circ\phi$ has a natural lifting in $\mathcal{C}at^{\pi}_{\infty}$, which is nothing but $\colim_K \phi$.

	\end{proof}
	
	\begin{construction}[Monoidal structures on $\sseq^{\gen,\vee}_R$]\label{c2.49n}
		When $R$ is a discrete commutative ring, $R[\mathcal{O}_{\Sigma}]$ is a full subcategory of  $\sseq(\m^{\heartsuit}_R)\subset\sseq_R$. Then one can check that $R[\mathcal{O}_{\Sigma}]$ is closed under the operations $\otimes$, $\levtimes$ and $\circ$. The composition product $\circ$ is expressed as
		\[X\circ Y\simeq \bigoplus_{r\ge 0}(X(r)\otimes_{R} Y^{\otimes r})_{\sig_r},\]where $(-)_{\sig_r}$ is taking ordinary orbits. At each arity $r$, taking $R$-linear dual determines a contravariant autoequivalence $R[\mathcal{O}_{\sig_r}]\simeq R[\mathcal{O}_{\sig_r}]^{op}$. It induces another monoidal structure as follows
		\[X\bcirc Y:=(X^\vee\circ Y^\vee)^\vee\simeq \bigoplus_{r\ge 0}(X(r)\otimes_{R} Y^{\otimes r})^{\sig_r},\]where $(-)^{\sig_r}$ is taking the usual fixed points.\ For each morphism $R\to S$, the transition functor $S\otimes_R(-):R[\mathcal{O}_{\Sigma}]\to S[\mathcal{O}_{\Sigma}]$ is compatible with all the operations $\otimes$, $\levtimes$, $\circ$ and $\bcirc$. Recall that $\mathrm{SCR}\simeq \mathcal{P}_{\Sigma}(\mathrm{Poly})$ where $\mathrm{Poly}$ is the 1-category of finitely generated polynomial rings over $\mathbb{Z}$. Thus, using Lemma \ref{lemma2.47}, the operations can be defined on $R[\mathcal{O}_{\Sigma}]$ for any $R\in \mathrm{SCR}$ by the same manipulation on bases. Moreover, they can be extended to $\sfil R[\mathcal{O}_{\Sigma}]$ and $\sgr R[\mathcal{O}_{\Sigma}]$ as well.
		
		Applying Theorem \ref{theorem2.29} to a coherent simplicial commutative ring $R$, we encounter an abundance of monoidal structures $\otimes$, $\levtimes$, $\circ$ and $\bcirc$ on the following \infcats,
		\begin{equation}\label{esseq}
			\sseq^{\gen}_R\xrightarrow{\iota} \sseq^{\gen,\vee}_R\xrightarrow{(-)_0} \sseq^{\gen,\vee}_{R,\fil}\xrightarrow{\gr}\sseq^{\gen,\vee}_{R,\gr}.
		\end{equation}Besides, these functors are symmetric monoidal with respect to $\otimes$ and $\levtimes$ and monoidal with respect to $\circ$ and $\bcirc$. In particular, the following functors of pro-coherent modules are equiped with both $(\sseq^{\gen,\vee,\ge1}_R,\circ)$ and $(\sseq^{\gen,\vee,\ge1}_R,\bcirc)$ linear structures:
		$(-)_0:\QC^\vee_R\to \fil\QC^\vee_R$, $[-]_0:\QC^\vee_R\to \gr\QC^\vee_R$, $\gr:\fil\QC^\vee_R\to\gr\QC^\vee_R$, $F_1:\fil_{\ge1}\QC^\vee_R\to \QC^\vee_R$ and $\colim: \fil_{\le-1}\QC^\vee_R\to \QC^\vee_R$.
		
	\end{construction}
	
	\begin{remark}
		The levelwise product $\levtimes$ admits a unit $\com$ in $\sseq^{\gen}_R$ which is given by $\com(r)\simeq \ul{R}$ for every $r\in\mathbb{N}$. Then $\com$ is sent to $\levtimes$-units in different contexts by line (\ref{esseq}). In each context, the operation $\levtimes$ is both lax and oplax with respect to $\circ$ and $\bcirc$.
	\end{remark}
	\begin{df}
		For a simplicial commutative ring $R$, the $\infty$-category of derived $\infty$-operads is defined as $\alg\big(\sseq^{\gen}_R,\circ\big)$. Similarly, restricted derived $\infty$-operads form an \infcat\ $\alg\big(\sseq^{\gen}_R,\bcirc\big)$.
	\end{df}
	\begin{exa}\label{exa2.52}
		The $\levtimes$-unit object $\com$ admits a natural derived $\infty$-operad structure, whose algebras in $\m_R$ are nothing but \textit{derived algebras} over $R$ in the sense of \cite[\S4.2]{Raksit}. When $\com$ is regarded as a restricted operad, it parametrises \textit{derived divided power algebras}.
		
		The \infcat\ of $R$-derived algebras (in ordinary, filtered or graded modules) admits a cocartesian symmetric monoidal structure, which is given by $\otimes_R$ on the level of underlying modules, \cite[Proposition 4.2.27]{Raksit}. The same claim also applies to pro-coherent derived algebras.
	\end{exa}
	\begin{df}\label{df3.22}
		Let $R$ be a coherent simplicial commutative ring.\ The \infcat\ $\op^{\gen,\vee}_R$ ($co\op^{\gen,\vee}_R$) of \textit{derived pro-coherent  $\infty$-(co)operads} is defined as the \infcat\ of the (co)associative (co)algebras in  $(\sseq^\vee_R,\circ)$. Parallelly, define the \infcat\ $\op^{\gen,\pd}_R$ ($co\op^{\gen,\pd}_R$) of \textit{derived PD $\infty$-(co)operads} as the \infcat\ of the (co)associative (co)algebras in $(\sseq^\vee_R,\bcirc)$.
	\end{df}
	

		\section{Filtered PD Koszul duality}\label{sec3}
	Recall that \textit{derived partition Lie algebras} are sophisticated algebraic objects designed to classify derived formal moduli problems in general characteristics \cite[\S4.2, \S5.2]{BM}. In this section, we present a filtered PD Koszul duality for derived partition Lie algebras in two steps: First, we use a categorical bar-cobar adjunction to obtain the adjunction between filtered non-unital derived algebras and filtered $\br(\com^{nu})$-coalgebras  (\S\ref{sec3.1}). Then we show that taking $R$-linear dual induces an adjunction between filtered $\br(\com^{nu})$-coalgebras and filtered $\br(\com^{nu})^\vee$-algebras (\S\ref{sec3.2}), namely \textit{filtered derived partition Lie algebras}. It turns out that, the \textit{dually almost perfect} derived partition Lie algebras (with the constant filtrations) are fully faithfully embedded into the \infcat\ of
	filtered non-unital derived algebras (Corollary \ref{cor3.31n}).

	Our method in this section is very formal so that it applies to every (spectral PD, derived pro-coherent or derived PD) $\infty$-operad satisfying condition ($\dagger$) in Construction \ref{c3.19n}. As a by-product, we also classify \textit{dually almost perfect} spectral partition Lie algebras  using filtered $\neoo$-algebras (Corollary \ref{cor3.30n}).
	\subsection{Refined bar-cobar adjunction for PD $\infty$-operads}\label{sec3.1}
	
	Let $\mathcal{C}$ be a pointed monoidal \infcat\ and $\mathcal{M}$ be a left $\mathcal{C}$-tensored \infcat. Suppose that both $\mathcal{C}$ and $\mathcal{M}$ have all geometric realizations and totalizations. A categorical bar-cobar construction established in \cite[\S5.2.2]{HA} and refined in \cite[\S3.4]{BCN} produces a commutative diagram
\[\begin{tikzcd}
	{\mathrm{LMod}(\mathcal{M})} & {\mathrm{LComod}(\mathcal{M})} \\
	{\alg(\mathcal{C})} & {co\alg(\mathcal{C})}
	\arrow["\br", shift left, from=1-1, to=1-2]
	\arrow["\pi"', from=1-1, to=2-1]
	\arrow["\cobr", shift left, from=1-2, to=1-1]
	\arrow["\pi", from=1-2, to=2-2]
	\arrow["\br", shift left, from=2-1, to=2-2]
	\arrow["\cobr", shift left, from=2-2, to=2-1]
\end{tikzcd},\]
	where $\br(A,M)$ is given by $(1\otimes_A 1,1\otimes_A M)$ for a left module object $(A,M)\in \mathrm{LMod}(\mathcal{M})$. By the virtue of its functoriality, we can establish a connection between filtered algebras and coalgebras. In this section, we use another notion of higher operads introduced in \cite[\S2]{HA} and refer them as \ooop s to avoid ambiguity.
	
	\begin{construction}[Following Construction \ref{c2.43n}, \ref{c2.49n}]\label{c3.1n}
		Let $\mathcal{LM}^\otimes$ be the \ooop\ in \cite[Definition 4.2.1.7]{HA}, whose categorical algebra $(\mathcal{C},\mathcal{M})$ consists of a monoidal \infcat\ $\mathcal{C}$ and a left $\mathcal{C}$-tensored \infcat\ $\mathcal{M}$. The underlying category of $\mathcal{LM}^{\otimes}$ is a discrete set $\{\mathfrak{a},\mathfrak{m}\}$.
		
		In this section, suppose that $R$ is a coherent connective $\eoo$-ring spectrum, and $\mathcal{C}$ is a pointed monoidal \infcat\ that is equivalent to either $(\sseq^{\vee,\ge1}_{R,1//1},\circ)$, $(\sseq^{\gen,\vee,\ge1}_{\ul{R},1//1},\circ)$ or $(\sseq^{\gen,\vee,\ge1}_{\ul{R},1//1},\bcirc)$ ($R$ is further assumed to be a simplicial commutative ring in the latter two cases).  The action of $\mathcal{C}$ on $\QC^\vee_R$ can be encoded by a cocartesian fibration of \ooop s	$\mathscr{C}^{\otimes}\to \clm^\otimes$ such that $\mathscr{C}_\mathfrak{a}\simeq \mathcal{C}$ and $\mathscr{C}_\mathfrak{m}\simeq \QC^\vee_R$

		The \infcat\ of left module objects $\mathrm{LMod}(\QC^\vee_R)$ of $\mathscr{C}^{\otimes}\to \clm^\otimes$ is the \infcat\ of \ooop ic sections. There is a natural projection $\pi:\mathrm{LMod}(\QC^\vee_R)\to \alg(\mathcal{C})$ whose fibre over ${\mathcal{P}}$ is $\alg_{\mathcal{P}}(\QC^\vee_R)$, where $\mathcal{P}$ is an augmented PD (derived pro-coherent, derived PD) $\infty$-operad.
		
		The cocartesian fibrations of \ooop s $(\fil_I\mathscr{C})^{\otimes}\to \clm^\otimes, (\gr_I\mathscr{C})^{\otimes}\to \clm^\otimes$ are similarly defined for the actions of $\mathcal{C}$ on $\fil_I\QC^\vee_R$ and $\gr_I\QC^\vee_R$ respectively. In the same way, the actions of $\mathcal{C}^{op}$ are encoded by $(\mathscr{C}^{op})^\otimes \to \clm^\otimes$. The \infcat\ of left comodule objects are defined as $\mathrm{LComod}(\QC^\vee_R):=\mathrm{LMod}(QC^{\vee,op}_R)^{op}$, while $\mathrm{LComod}(\fil_I\QC^\vee_R)$ and $\mathrm{LComod}(\gr_I\QC^\vee_R)$ are defined using the same method. The comodule categories admit projections to $co\alg(\mathcal{C})$.
	\end{construction}
	
	\begin{theorem}\label{thm3.2n}
		There is a commutative diagram\[\begin{tikzcd}[row sep=0.6cm, column sep=0.4cm]
			&& {\mathrm{LMod}(\QC^\vee_R)} && {\mathrm{LComod}(\QC^\vee_R)} \\
			{\mathrm{LMod}(\fil_{\ge1}\QC^\vee_R)} && {\mathrm{LComod}(\fil_{\ge1}\QC^\vee_R)} && {\alg(\mathcal{C})} && {co\alg(\mathcal{C})} \\
			&& {\mathrm{LMod}(\gr_{\ge1}\QC^\vee_R)} && {\mathrm{LComod}(\gr_{\ge1}\QC^\vee_R)}
			\arrow["\br", shift left=1, from=1-3, to=1-5]
			\arrow["\cobr", shift left=1, from=1-5, to=1-3]
			\arrow["\br", shift left=1, from=3-3, to=3-5]
			\arrow["\cobr", shift left=1, from=3-5, to=3-3]
			\arrow[from=1-3, to=2-5]
			\arrow[from=3-3, to=2-5]
			\arrow[from=1-5, to=2-7]
			\arrow[from=3-5, to=2-7]
			\arrow["{F_1}", shift right=1, from=2-1, to=1-3]
			\arrow["\gr"', from=2-1, to=3-3]
			\arrow["\br", shift left=1, from=2-1, to=2-3]
			\arrow["\cobr", shift left=1, from=2-3, to=2-1]
			\arrow["{F_1}"', shift right=1, from=2-3, to=1-5, crossing over, near start]
			\arrow["\gr", from=2-3, to=3-5, crossing over, near start]
			\arrow["\br", shift left=1, from=2-5, to=2-7]
			\arrow["\cobr", shift left=1, from=2-7, to=2-5]
		\end{tikzcd}\]where the rows are adjunctions. Furthermore, the diagram satisfies the following properties:
		\begin{enumerate}[label={(\arabic*)}]
			\item The $\br$s preserve cocartesian edges relative to the projections to the \infcats\ of (co)algebras.
			\item The front square exhibits \textnormal{presentable fibrations} over $co\alg(\mathcal{C})$, while the back square exhibits presentable fibrations over $\alg(\mathcal{C})$ (in the sense of  \cite[Definition 5.5.3.2]{HTT}).
			
			\item The $F_1$s and $\gr$s are \textnormal{morphisms of presentable fibrations}, which means that they preserve both cartesian and cocartesian edges.
		\end{enumerate}
	\end{theorem}
	The proof of this theorem is broken into small pieces as follows to make it more readable. We first prove (2) and (3), then return to the commutativity and (1).
	\begin{remark}[Back square]\label{rk3.3n}
		 In the context of Construction \ref{c3.1n}, it is well known that the algebra category of each (generalised) $\infty$-operad is presentable. Here we briefly recall how to prove it. Since $\circ$ is compatible with sifted colimits, Lemma 3.2.3.4 in \cite{HA} shows that each vertex in the back square is an accessible \infcat; in particular, for each $\mathcal{P}\in \alg(\mathcal{C})$, the algebra category $\alg_{\mathcal{P}}(\QC^\vee_R)$ is accessible \cite[Proposition 5.4.6.6]{HTT}. The pushouts of $B\leftarrow A\rightarrow C$ in $\alg_{\mathcal{P}}(\QC^\vee_R)$ can be constructed explicitly using bar resolutions by free algebras. The same reasoning works for filtered and graded algebras. For each morphism $\mathcal{P}\to\mathcal{Q}$ in $\alg(\mathcal{C})$, the cocartesian edges over it are given by $\mathcal{Q}\otimes_{\mathcal{P}}-$. Thus the back square consists of presentable fibrations over $\alg(\mathcal{C})$. Additionally, the functors $F_1$ and $\gr$ are morphisms of presentable fibrations, since they come from small-colimit-preserving $\mathcal{C}$-linear functors.
	\end{remark}
	
	We need some categorical preparation to treat the front square.
	\begin{lemma}\label{lemma3.4n}
		Let $\mathcal{M}^{\otimes}\to \mathcal{O}^{\otimes}$ be a cocartesian fibration of \ooop s, where $\mathcal{O}^\otimes$ is small as a simplicial set. The \infcat\ of operadic sections $\alg_{/\mathcal{O}}(\mathcal{M})$ is \textit{op-accessible} (i.e.\ its opposite is accessible) if the following conditions are satisfied:
		\begin{itemize}
			\item for each $X\in \mathcal{O}^{\otimes}$, the fibre $\mathcal{M}_X$ is op-accessible;
			\item for each $f:X\to Y$ in $\mathcal{O}^{\otimes}$, the induced functor $f_!:\mathcal{M}_X\to \mathcal{M}_Y$ is filtered-limit-preserving.
		\end{itemize}
	\end{lemma}
	\begin{proof}
		The thrust of proof is \cite[Proposition 5.4.7.11]{HTT}. Since the functor $(-)^{op}$ of taking opposite \infcat\ is an autoequivalence of $\widehat{\mathcal{C}at}_\infty$, we can take $\mathscr{C}$ as the (not full) subcategory $opAcc$ of $\widehat{\mathcal{C}at}_\infty$ of op-accessible $\infty$-categories with filtered-limit-preserving functors, which satisfies the conditions of \cite[Proposition 5.4.7.11]{HTT} by \cite[Remark 5.4.7.13]{HTT}. Consider the categorical fibration $\mathcal{M}^{\otimes}\to \mathcal{O}^{\otimes}$ and let $\mathcal{E}$ be the set of inert morphisms in $\mathcal{O}$. The operadic sections $\alg_{/\mathcal{O}}(\mathcal{M})$ belongs to $opAcc$ by \cite[Proposition 5.4.7.11]{HTT}.
	\end{proof}
	\begin{prop}\label{prop3.5n}
		Following the notation of Construction \ref{c3.1n}, the $\infty$-categories $\mathrm{LComod}(\QC^\vee_R)$, $\mathrm{LComod}(\fil_{\ge1}\QC^\vee_R)$,  $\mathrm{LComod}(\gr_{\ge1}\QC^\vee_R)$ and ${co\alg(\mathcal{C})}$ are all presentable.
	\end{prop}
	\begin{proof}
		Unpacking the definition, $\mathrm{LComod}(\QC^\vee_R)$ is $\alg_{/\mathcal{LM}}(\mathscr{C}^{op})^{op}$, which is accessible according to Lemma \ref{lemma3.4n}. Furthermore,  $\alg_{/\mathcal{LM}}(\mathscr{C}^{op})$  admits small limits by \cite[Corollary 3.2.2.5]{HA} and the fact that $\mathcal{C}$ and $\QC^\vee_R$ are presentable. Therefore, $\mathrm{LComod}(\QC^\vee_R)$ is presentable. The proof is the same for the other three $\infty$-categories.
	\end{proof}
	\begin{prop}\label{prop3.6n}
		For each (spectral PD, derived pro-coherent or derived PD) $\infty$-cooperad $\mathcal{Q}$, the \infcats\ of $\mathcal{Q}$-coalgebras $co\alg_{\mathcal{Q}}(\QC^\vee_R)$, $co\alg_{\mathcal{Q}}(\fil_{\ge1}\QC^\vee_R)$ and $co\alg_{\mathcal{Q}}(\gr_{\ge 1}\QC^\vee_R)$ are presentable.
	\end{prop}
	\begin{proof}
		Consider the homotopy pullback of \infcats
		\[\begin{tikzcd}
			{co\alg_{\mathcal{Q}}(\QC^\vee_R)} & {\mathrm{LComod}(\QC^\vee_R)} \\
			{\{\mathcal{Q}\}} & {co\alg(\mathcal{C})}
			\arrow[from=1-1, to=2-1]
			\arrow[from=2-1, to=2-2]
			\arrow[from=1-2, to=2-2]
			\arrow[from=1-1, to=1-2]
		\end{tikzcd},\]where $co\alg_{\mathcal{Q}}(\QC^\vee_R)$ is accessible by \cite[Proposition 5.4.6.6]{HTT} and Proposition \ref{prop3.5n}. Moreover, $co\alg_{\mathcal{Q}}(\QC^\vee_R)$ admits all small colimits, because the forgetful functor $co\alg_{\mathcal{Q}}(\QC^\vee_R)\to \QC^\vee_R$ creates small colimits. The filtered and graded are proved mutatis mutandis.
	\end{proof}
	\begin{prop}[Front square]\label{prop3.7n}
		The left comodule categories are presentable fibrations over $co\alg(\mathcal{C})$, while the functors $F_1$ and $\gr$ are morphisms of presentable fibrations
		\[\begin{tikzcd}
			{\mathrm{LComod}(\QC^\vee_R)} & {\mathrm{LComod}(\fil_{\ge1}\QC^\vee_R)} & {\mathrm{LComod}(\gr_{\ge1}\QC^\vee_R)} \\
			& {co\alg(\mathcal{C})}
			\arrow["{F_1}"', from=1-2, to=1-1]
			\arrow["\gr", from=1-2, to=1-3]
			\arrow[from=1-2, to=2-2]
			\arrow[from=1-1, to=2-2]
			\arrow[from=1-3, to=2-2]
		\end{tikzcd}.\]
	\end{prop}
	\begin{proof}
		The projection $\pi:\mathrm{LComod}(\QC^\vee_R)\to co\alg(\mathcal{C})$ is a cocartesian fibration such that, for each morphism $p:\mathcal{P}\to \mathcal{Q}$ in $co\alg(\mathcal{C})$, the cocartesian edges over $p$ are given by the forgetful functor of left comodules. Note that forgetful functors here are small-colimit-preserving and each fibre is presentable, so the right adjoints of forgetful functors produce abundant locally cartesian edges by the adjoint functor theorem. Thus $\pi$ is in fact a presentable fibration. The same proof also works for the filtered and graded cases. The functor $F_1$ here commutes with forgetful functors along arrows in $co\alg(\mathcal{C})$, i.e. preserves the cocartesian edges. At the same time, the corresponding $\mathcal{LM}^{\otimes}$-monoidal functor of the underlying module categories admits an oplax left adjoint \cite[Corollary 7.3.2.7]{HA}, which means that $F_1$ has a left adjoint relative to $co\alg(\mathcal{C})$ commuting with the forgetful functors. Taking the right adjoints, $F_1$ commutes with the right adjoints of forgetful functors, i.e. preserves cartesian edges. The proof for $\gr$ is by mutatis mutandis.

	\end{proof}
	\begin{proof}[Proof of Theorem \ref{thm3.2n}]
		Apply Theorem 3.26 in \cite{BCN} to the cocartesian fibrations of \ooop s in Construction \ref{c3.1n}: $\mathscr{C}^{\otimes}\to \clm^\otimes$, $\fil_{\ge 1}\mathscr{C}^{\otimes}\to \clm^\otimes$, $\gr_{\ge 1}\mathscr{C}^{\otimes}\to \clm^\otimes$ and their opposite counterparts. This establishes everything that does not mention $F_1$ and $\gr$. The assertions (2) and (3) are verified in Remark \ref{rk3.3n} and Propostion \ref{prop3.7n}.
		
		It only remains to demonstrate that $F_1$ and $\gr$ commute with $\br$ and $\cobr$. Expanding the categorical construction in \cite[\S3.4]{BCN}, a $\mathcal{C}$-linear functor $f:\mathcal{M}\to \mathcal{N}$ gives rise to a morphism of pairings of \infcats\
		$\tilde{f}:\mathrm{LMod}(\text{Tw}(\mathcal{M}))\to \mathrm{LMod}(\text{Tw}(\mathcal{N}))$. For each $(A,M)\in\mathrm{LMod}(\mathcal{M})$, the left universal object $u$ lying over it can be understood as $(A,M)\to \br(A,M)=:(C,N)$, and its image $\tilde{f}(u)$ can be written as an arrow $(A,fM)\to (C,fN)$ receiving a unique morphism from the left universal object $(A,fM)\to \br(A,fM)=:(C',N')$. According to \cite[Theorem 5.2.2.17]{HA}, both $C$ and $C'$ are calculated by $|\br_\bullet(1,A,1)|$ in $\mathcal{C}$, while the underlying morphism of $fN\to N'$ in $\mathcal{N}$ is
		$f|\br_\bullet(1,A,M)|\to |\br_\bullet(1,A,fM)|$. If $f:\mathcal{M}\to \mathcal{N}$ is assumed to be sifted-colimit-preserving, then $\tilde{f}(u)$ is left universal as well, which means that $f$ commutes with $\br$ \cite[Proposition 5.2.1.17]{HA}. Dually, if $f:\mathcal{M}\to \mathcal{N}$ preserves sifted limits, we have that $f$ commutes with $\cobr$. Finally, recall that $F_1:\fil_{\ge1}\QC^\vee_R\to \QC^\vee_R$ and $\gr:\fil_{\ge1}\QC^\vee_R\to \gr_{\ge1}\QC^\vee_R$ preserve all small limits and colimits, and they are $\mathcal{C}$-linear.
	\end{proof}
	
	Consider $\mathcal{P}\in\alg(\mathcal{C})$ and choose $\mathcal{M}$ among $\QC^\vee_R$, $\fil\QC^\vee_R$ and $\gr\QC^\vee_R$. The functor \[\br:\alg_{\mathcal{P}}(\mathcal{M})\to co\alg_{\br(\mathcal{P})}(\mathcal{M})\] admits a right adjoint, which is given by the following composition
	\[co\alg_{\br(\mathcal{P})}(\mathcal{M})\xrightarrow{\cobr}\alg_{\cobr\circ\br(\mathcal{P})}(\mathcal{M})\xrightarrow{\phi}\alg_{\mathcal{P}}(\mathcal{M}),\]where $\phi$ is the forgetful functor along $\mathcal{P}\to \cobr\circ\br(\mathcal{P})$. Here, $\phi$ can be omitted for \textit{Koszul (generalized) $\infty$-operads}.
	
	\begin{df}\label{df3.8n}
		An augmented (spectral PD, pro-coherent derived or derived PD) $\infty$-operad $\mathcal{P}$ is said to be \textit{Koszul} if the unit morphism
		$\mathcal{P}\to \cobr\circ\br(\mathcal{P})$ is an equivalence.
	\end{df}
	The following result is proven for spectra by Ching  in \cite[Theorem 2.15]{ching2012bar}, cf. also \cite{BG} for the general case:
	\begin{prop}\label{prop3.9n}
		Suppose that $\mathcal{P}\in \alg(\mathcal{C})$ is \textnormal{reduced}, i.e.\ it satisfies the following conditions:
		
		(1) there is no nullary operation, i.e. $\mathcal{P}(0)\simeq 0$;
		
		(2) the unit map induces a natural equivalence $R\simeq \mathcal{P}(1)$ at arity 1;\\
		then $\mathcal{P}$ is Koszul.
	\end{prop}
	\begin{exa}\label{exa3.10n}
		Let $R$ be a connective $\eoo$-ring spectrum. The $\infty$-operad $\neoo$ of augmented commutative $R$-algebra is Koszul as a spectral PD $\infty$-operad. When $R$ is further a simplicial commutative ring, $\com^{nu}$ is Koszul both as a derived pro-coherent  $\infty$-operad and a derived PD $\infty$-operad.
	\end{exa}
	\subsection{Refined duality for PD $\infty$-(co)operads}\label{sec3.2}
	Now we show that taking $R$-linear dual induces an adjunction between coalgebras and algebras that respects the filtrations.
	\begin{notation}\label{n3.11n}
		If $R$ is a coherent connective $\eoo$-ring spectrum, choose $\mathscr{A}$ from $\vect^\omega_R[\sig]$, $\sfil \vect^\omega_R[\sig]$ or $\sgr \vect^\omega_R[\sig]$. The \textit{$R$-linear dual functor of spectral pro-coherent symmetric sequences} is given by the internal hom-object $(-)^\vee:=\hom_{(\QC_\mathscr{A}^\vee,\levtimes)}(-, \eoo)$. 
		
		When $R$ is a coherent simplicial commutative ring, let $\mathscr{B}$ be one of $R[\mathcal{O}_{\sig}]$, $\sfil R[\mathcal{O}_{\sig}]$ or $\sgr R[\mathcal{O}_{\sig}]$. The \textit{$R$-linear dual functor of derived pro-coherent symmetric sequences} is defined as $(-)^\vee:=\hom_{(\QC_{\mathscr{B}}^\vee,\levtimes)}(-, \com)$.
	\end{notation}
	\begin{prop}[Lax monoidal structure of $(-)^\vee$]\label{prop3.12n}
		Following Notation \ref{n3.11n}, we have:
		
		(1) The $R$-linear dual functor $(-)^\vee:\QC_\mathscr{A}^{\vee,op}\to \QC_\mathscr{A}^\vee$ is lax monoidal with respect to $\circ$. Moreover, it restricts to a contravariant monoidal equivalence
		$\aperf_{\mathscr{A}}^{op}\xrightarrow{\simeq}\aperf^\vee_{\mathscr{A}}$.
		
		(2) The derived $R$-linear dual functor gives rise to lax monoidal functors
		\[(-)^\vee:(\QC^{\vee,op}_\mathscr{B},\circ)\to (\QC^\vee_\mathscr{B},\bcirc),\ \ \ \ (-)^\vee:(\QC^{\vee,op}_\mathscr{B},\bcirc)\to (\QC^\vee_\mathscr{B},\circ)\]which form a contravariant adjunction. Furthermore, this adjunction restricts to the monoidal contravariant equivalences
		$(\aperf_{\mathscr{B}}^{op},\circ)\xrightarrow{\simeq} (\aperf^\vee_{\mathscr{B}},\bcirc)$ and $(\aperf_{\mathscr{B}}^{op},\bcirc)\xrightarrow{\simeq} (\aperf^\vee_{\mathscr{B}},\circ)$.
	\end{prop}
	\begin{proof}
		The lax monoidal structure in (1) is established by Lemma \ref{lemma3.13n} below. In the meanwhile, each $R[\Sigma_r]$ is dualizable and has itself as a dual object as a finite free $R$-module. Thus Proposition \ref{prop2.30} implies an equivalence
		$\aperf_{\mathscr{A}}^{\vee,op}\xrightarrow{\simeq}\aperf_{\mathscr{A}}$. This equivalence is monoidal with respect to $\circ$, because it is given by left Kan extension and $\circ$ commutes with sifted colimits.
		
		For (2), we start with the monoidal equivalence $F^{op}:=(-)^\vee:(\mathscr{B}^{op},\circ)\xrightarrow{\simeq} (\mathscr{B},\bcirc)$. Left Kan extension produces a monoidal equivalence $F^{R,op}:(\aperf_{\mathscr{B}}^{\vee,op},\circ)\xrightarrow{\simeq}(\aperf_{\mathscr{B}},\bcirc)$, since $\bcirc$ commutes with sifted colimits. Then right Kan extension produces a lax monoidal right adjoint $F^{RL,op}:(-)^\vee:(\QC^{\vee,op}_\mathscr{B},\circ)\to (\QC^\vee_\mathscr{B},\bcirc)$ \cite[Lemma 2.58]{BCN}. The other direction is done in the same way.
		
	\end{proof}

	\begin{lemma}cf. \cite[Proposition 3.47]{BCN}\label{lemma3.13n}
		Let $\mathcal{E}$ be a presentable \infcat\ endowed with two monoidal structures $\otimes$ and $\boxtimes$, while $\boxtimes$ is closed and symmetric. Suppose that $\mathcal{E}\times\mathcal{E}\xrightarrow{\boxtimes}\mathcal{E}$ is oplax with respect to $\otimes$, and that the unit $1_\boxtimes$ of $\boxtimes$ is an associative algebra in $(\mathcal{E},\otimes)$. Then
		$(-)^\vee:=(\hom_{(\mathcal{E},\boxtimes)}(-,1_\boxtimes):\mathcal{E}^{op}\to \mathcal{E})$ is lax monoidal with respect to $\otimes$.
	\end{lemma}
	\begin{proof}
		Since $\mathcal{E}$ is essentially small, Yoneda embedding induces a monoidal functor $j:\mathcal{E}\hookrightarrow \mathrm{Fun}(\mathcal{E}^{op},\mathcal{S}_{\mathrm{large}})$ with respect to $\otimes$ and Day convolution tensor product. The composite $j\circ (-)^\vee$ corresponding to 
		\[\mathcal{E}^{op}\times \mathcal{E}^{op}\xrightarrow{\boxtimes}\mathcal{E}^{op}\xrightarrow{\map(-,1_\boxtimes)}\mathcal{S}_{\mathrm{large}}\] by the canonical equivalence $\mathrm{Fun}(\mathcal{E}^{op},\mathrm{Fun}(\mathcal{E}^{op},\mathcal{S}_{\mathrm{large}}))\simeq \text{Fun}(\mathcal{E}^{op}\times\mathcal{E}^{op},\mathcal{S}_{\mathrm{large}})$. Here $\map(-,1_\boxtimes)$ is lax monoidal, since $j$ is monoidal and $1_\boxtimes$ is an $\otimes$-algebra. Thus $\map(-\boxtimes-, 1_\boxtimes)$ is a lax monoidal functor. The construction of Day convolution \cite[\S2.2.6]{HA} provides a commutative diagram
		\[\begin{tikzcd}
			{\alg_{(\mathcal{E}^{op}\times_{\text{Assoc}}\mathcal{E}^{op})/\text{Assoc}}(\mathcal{S}_{\mathrm{large}})} & {\text{Fun}(\mathcal{E}^{op}\times\mathcal{E}^{op},\mathcal{S}_{\mathrm{large}})} \\
			{\alg_{\mathcal{E}^{op}/\text{Assoc}}(\text{Fun}(\mathcal{E}^{op},\mathcal{S}_{\mathrm{large}}))} & {\text{Fun}(\mathcal{E}^{op},\text{Fun}(\mathcal{E}^{op},\mathcal{S}_{\mathrm{large}}))}
			\arrow["\simeq", from=1-1, to=2-1]
			\arrow[from=2-1, to=2-2]
			\arrow[from=1-1, to=1-2]
			\arrow["\simeq"', from=1-2, to=2-2]
		\end{tikzcd},\]which implies that $j\circ(-)^\vee$ admits a lax monoidal structure with respect to $\otimes$. Finally, we conclude by noting that $j$ is a fully faithful monoidal functor.
	\end{proof}
	We can regard $R$ as a filtered (graded) module by the embedding $(-)_0$ ($[-]_0$) into weight $0$. The $R$-linear dual of a filtered module $M$ is determined by \begin{equation}\label{dualfil}
		F_{n}(M^\vee)\simeq (\mathop{\colim}_{k\to -\infty}F_kM/F_{-n+1}M)^\vee,
	\end{equation}which is \textit{always} complete. Taking $R$-linear dual of graded modules sends $N_\star$ to $N^\vee_{-\star}$. These formulae show that there is a commuting diagram in $\mathcal{A}dd^{\text{coh,poly}}$
	\begin{equation}\label{e3}
		\begin{tikzcd}
			{(\vect^\omega_{R})^{op}} & {\vect^\omega_{R}} \\
			{(\sfil_{\ge1}\vect^\omega_{R})^{op}} & {\sfil_{\le-1}\vect^\omega_{R}} \\
			{(\sgr_{\ge1}\vect^\omega_{R})^{op}} & {\sgr_{\le-1}\vect^\omega_{R}}
			\arrow["{(-)_1}", from=1-1, to=2-1]
			\arrow["{(-)^\vee}", from=1-1, to=1-2]
			\arrow["\const", from=1-2, to=2-2]
			\arrow["\gr", from=2-2, to=3-2]
			\arrow["{(-)^\vee}", from=2-1, to=2-2]
			\arrow["{(-)^\vee}", from=3-1, to=3-2]
			\arrow["\gr", from=2-1, to=3-1]
		\end{tikzcd}.
	\end{equation}
	\begin{notation}\label{n3.14n}
		Adopt the same notation as Construction \ref{c3.1n}. When $\mathcal{C}=(\sseq^{\vee,\ge1}_{R,1//1},\circ)$, $(\sseq^{\gen,\vee,\ge1}_{R,1//1},\circ)$ or $(\sseq^{\gen,\vee,\ge1}_{R,1//1},\bcirc)$, set $\mathcal{C}_d$ as $(\sseq^{\vee,\ge1}_{R,1//1},\circ)$, $(\sseq^{\gen,\vee,\ge1}_{R,1//1},\bcirc)$ or $(\sseq^{\gen,\vee,\ge1}_{R,1//1},\circ)$ respectively.
		The $\mathcal{LM}^{\otimes}$-monoidal \infcats\ that encode the action of $\mathcal{C}_d$ on $\QC^\vee_R$, $\fil_{\le-1}\QC^\vee_R$ or $\gr_{\le-1}\QC^\vee_R$ are denoted as $\mathscr{C}^\otimes_d$, $(\fil_{\le-1}\mathscr{C}_d)^\otimes$ or $(\gr_{\le-1}\mathscr{C}_d)^\otimes$ respectively.
	\end{notation}
	\begin{prop}\label{prop3.15}
		There exists a commutative diagram of \ooop s over $\clm^{\otimes}$
		\begin{equation}\label{e4}
			\begin{tikzcd}
				{(\mathscr{C}^{op})^\otimes} & {\mathscr{C}_d^\otimes} \\
				{(\fil_{\ge1}\mathscr{C}^{op})^\otimes} & {(\fil_{\le-1}\mathscr{C}_d)^\otimes} \\
				{(\gr_{\ge1}\mathscr{C}^{op})^\otimes} & {(\gr_{\le-1}\mathscr{C}_d)^\otimes}
				\arrow["{(-)_1}", from=1-1, to=2-1]
				\arrow["\gr", from=2-1, to=3-1]
				\arrow["{(-)^\vee}", from=1-1, to=1-2]
				\arrow["\const", from=1-2, to=2-2]
				\arrow["\gr", from=2-2, to=3-2]
				\arrow["{(-)^\vee}", from=2-1, to=2-2]
				\arrow["{(-)^\vee}", from=3-1, to=3-2]
			\end{tikzcd}
		\end{equation}such that:
		
		(1) on the fibres over $\mathfrak{a}$, the vertical arrows are identities of $\mathcal{C}$ or $\mathcal{C}_d$, the horizontal arrows are equivalent to the R-linear dual functor $\mathcal{C}^{op}\xrightarrow{(-)^\vee}\mathcal{C}_d$ defined in Notation \ref{n3.11n};
		
		(2) on the fibres over $\mathfrak{m}$, the diagram is given by right-left derived functors of (\ref{e3}).
		
	\end{prop}	
	\begin{proof}
		The existence of the bottom vertical arrows is clear, since $\gr:\fil_I\QC^\vee_{R}\to \gr_I\QC^\vee_{R}$ is both $\mathcal{C}$-linear and $\mathcal{C}_d$-linear. Now observe that $(-)_1$ admits a right adjoint
		$F_1:\fil_{\ge 1}\QC^\vee_{R}\to \QC^\vee_R$, which is $\mathcal{C}$-linear by Construction \ref{c2.43n} and \ref{c2.49n}. Then $(-)_1$ is promoted to a right ajdoint $(-)_1: {(\mathscr{C}^{op})^\otimes}\to {(\fil_{\ge1}\mathscr{C}^{op})^\otimes}$ relative to $\mathcal{LM}^{\otimes}$ \cite[Corollary 7.3.2.7]{HA}, which is an operadic map satisfying (1) and (2). Applying the same procedure to $\colim\dashv \const$ gives the other lax $\clm^\otimes$-monoidal functor $\const:\mathscr{C}_d^\otimes\to (\fil_{\le-1}\mathscr{C}_d)^\otimes$.
		
		Restricting the lax monoidal functors in Proposition \ref{prop3.12n} leads to the expected horizontal arrows. The commutativity of (\ref{e4}) can be checked by considering the underlying functors.

	\end{proof}

	\begin{corollary}\label{cor3.16n}
		Taking $R$-linear dual induces a commutative diagram of right adjoints
		\begin{equation}\label{e5}
			\begin{tikzcd}
				{\mathrm{LComod}(\QC^\vee_R)^{op}} & {{\mathrm{LMod}(\QC^\vee_R)}} \\
				{\mathrm{LComod}(\fil_{\ge1}\QC^\vee_R)^{op}} & {{\mathrm{LMod}(\fil_{\le-1}\QC^\vee_R)}} \\
				{\mathrm{LComod}(\gr_{\ge1}\QC^\vee_R)^{op}} & {{\mathrm{LMod}(\gr_{\le-1}\QC^\vee_R)}} \\
				{(co\alg(\mathcal{C}))^{op}} & {\alg(\mathcal{C}_d)}
				\arrow["{(-)_1}", from=1-1, to=2-1]
				\arrow["\gr", from=2-1, to=3-1]
				\arrow["\const", from=1-2, to=2-2]
				\arrow["\gr", from=2-2, to=3-2]
				\arrow["{(-)^\vee}", from=1-1, to=1-2]
				\arrow["{(-)^\vee}", from=2-1, to=2-2]
				\arrow["{(-)^\vee}", from=3-1, to=3-2]
				\arrow[from=3-1, to=4-1]
				\arrow[from=3-2, to=4-2]
				\arrow["{(-)^\vee}", from=4-1, to=4-2]
			\end{tikzcd}.
		\end{equation}
	\end{corollary}
	\begin{proof}
		This diagram is constructed by considering the operadic sections of (\ref{e4}), which consists of small-limit-preserving functors using \cite[Proposition3.2.2.1]{HA}. Since the vertices on the left-hand side are all op-presentable \infcats\ (Proposition \ref{prop3.5n}), the arrows in (\ref{e5}) starting from them are all right adjoints. Meanwhile, it is known that the vertices on the right-hand side are presentable $\infty$-categories. Observing that the three right vertical arrows preserve sifted colimits \cite[Proposition 3.2.3.1]{HA}, we conclude the proof using (2) in \cite[Corollary 5.5.2.9]{HTT}.
	\end{proof}
	\begin{notation}\label{n3.17n}
		Denote the adjoints of $(-)^\vee$ in (\ref{e5}) as $(-)^{\ed}$.
	\end{notation}
	The functors $(-)^{\ed}$ are formally deduced from the adjoint functor theorem. However, they can be made explicit for dually almost perfect algebras.
	\begin{corollary}\label{cor3.17n}
		The above adjunctions restrict to the following commutative diagram,
		\[
		\begin{tikzcd}[column sep=2cm]
			{\mathrm{LComod}_{\ap}(\QC^\vee_R)^{op}} & {{\mathrm{LMod}_{\dap}(\QC^\vee_R)}} \\
			{\mathrm{LComod}_{\ap}(\fil_{\ge1}\QC^\vee_R)^{op}} & {{\mathrm{LMod}_{\dap}(\fil_{\le-1}\QC^\vee_R)}} \\
			{\mathrm{LComod}_{\ap}(\gr_{\ge1}\QC^\vee_R)^{op}} & {{\mathrm{LMod}_{\dap}(\gr_{\le-1}\QC^\vee_R)}} \\
			{(co\alg_{\ap}(\mathcal{C}))^{op}} & {\alg_{\dap}(\mathcal{C})}
			\arrow["{(-)_1}", from=1-1, to=2-1]
			\arrow["\const", from=1-2, to=2-2]
			\arrow["\gr", from=2-2, to=3-2]
			\arrow["{(-)^\vee,\simeq}", from=1-1, to=1-2]
			\arrow["{(-)^\vee,\simeq}", from=2-1, to=2-2]
			\arrow["{(-)^\vee,\simeq}", from=3-1, to=3-2]
			\arrow[from=3-1, to=4-1]
			\arrow[from=3-2, to=4-2]
			\arrow["{(-)^\vee,\simeq}", from=4-1, to=4-2]
			\arrow["\gr", from=2-1, to=3-1]
		\end{tikzcd}\]where $\ap$ ($\dap$) means that the underlying module is (dually) almost perfect, and the inverses of $(-)^\vee$ are given by $(-)^\ed$. In particular, the underlying module of $(-)^{\ed}$ is given by $(-)^\vee$ for dually almost perfect objects.
	\end{corollary}
	\begin{proof}
		Restricting the diagram in Proposition \ref{prop3.15} to almost perfect (on the left) or dually almost perfect objects (on the right), this induces $\mathcal{LM}^{\otimes}$-monoidal equivalences. Then conclude by taking operadic sections.
	\end{proof}

	\begin{df}\label{df3.18n}
		For an augmented (spectral PD, derived pro-coherent or derived PD) $\infty$-operad $\mathcal{P}$ over $R$, its \textit{PD Koszul dual $\infty$-operad} is defined as $\kd(\mathcal{P}):=\br(\mathcal{P})^\vee$.
		
		If $A$ is a $\mathcal{P}$-algebra object lies in $\alg_{\mathcal{P}}(\QC^\vee_R)$ (or $\alg_{\mathcal{P}}(\fil_{\ge 1}\QC^\vee_R)$, $\alg_{\mathcal{P}}(\gr_{\ge 1}\QC^\vee_R)$), its \textit{PD Koszul dual} is the $\kd(\mathcal{P})$-algebra object $\kd(A):=\br(A)^\vee$ lies in $\alg_{\kd(\mathcal{P})}(\QC^\vee_R)$ (resp. $\alg_{\kd(\mathcal{P})}(\fil_{\le-1}\QC^\vee_R)$, $\alg_{\kd(\mathcal{P})}(\gr_{\le-1}\QC^\vee_R)$).
	\end{df}
	\subsection{Filtered Koszul duality for algebras}\label{sec3.4}
	
	We assemble the results from the former two subsections together to relate the $\mathcal{P}$-algebras and $\kd(\mathcal{P})$-algebras via an adjunction that respects filtrations. The target is to establish that, under condition ($\dagger$) stated below, the \infcat\ of dually almost perfect $\kd(\mathcal{P})$-algebras can be fully faithfully embedded into that of filtered $\mathcal{P}$-algebras,
	$\aqf:\alg^{\dap}_{\kd(\mathcal{P})}\hookrightarrow\alg_{\mathcal{P}}(\fil_{\ge1}\QC^\vee_R)^{op}$ (Theorem \ref{thm3.24n}). In particular, this theorem establishes a filtered Koszul duality for dually almost perfect partition Lie algebras.
	\begin{construction}\label{c3.19n}
		Following Construction \ref{c3.1n} and fix a $\mathcal{P}\in \alg(\mathcal{C})$ satisfying the condition
		\begin{center}
			(\textdagger) $\mathcal{P}$ is reduced and connective, while $\br(\mathcal{P})$ is almost perfect,
		\end{center}
		so that the unit morphisms $\mathcal{\mathcal{P}}\to \cobr\circ\br(\mathcal{P})$ and $\br(\mathcal{P})\to ((\br(\mathcal{P}))^\vee)^\ed$ (Notation \ref{n3.17n}) are equivalences. Then, Theorem \ref{thm3.2n} and Corollary \ref{cor3.16n} induce a commuting diagram of adjoints
		\begin{equation}\label{ecd}
			\begin{tikzcd}
				{\alg_{\mathcal{P}}(\QC^\vee_R)} & {co\alg_{\br(\mathcal{P})}(\QC^\vee_R)} & {\alg_{\kd(\mathcal{P})}(\QC^\vee_R)^{op}} \\
				{\alg_{\mathcal{P}}(\fil_{\ge1}\QC^\vee_R)} & {co\alg_{\br(\mathcal{P})}(\fil_{\ge1}\QC^\vee_R)} & {\alg_{\kd(\mathcal{P})}(\fil_{\le-1}\QC^\vee_R)^{op}} \\
				{\alg_{\mathcal{P}}(\gr_{\ge1}\QC^\vee_R)} & {co\alg_{\br(\mathcal{P})}(\gr_{\ge1}\QC^\vee_R)} & {\alg_{\kd(\mathcal{P})}(\gr_{\le-1}\QC^\vee_R)^{op}}
				\arrow["\adic"', shift right, from=1-1, to=2-1]
				\arrow["{(-)_1}"', shift right, from=1-2, to=2-2]
				\arrow["\const"', shift right, from=1-3, to=2-3]
				\arrow["\gr"', shift right, from=2-2, to=3-2]
				\arrow["\gr"', shift right, from=2-1, to=3-1]
				\arrow["\br", shift left, from=1-1, to=1-2]
				\arrow["\cobr", shift left, from=1-2, to=1-1]
				\arrow["\br", shift left, from=2-1, to=2-2]
				\arrow["\cobr", shift left, from=2-2, to=2-1]
				\arrow["\br", shift left, from=3-1, to=3-2]
				\arrow["\cobr", shift left, from=3-2, to=3-1]
				\arrow["{(-)^\vee}", shift left, from=1-2, to=1-3]
				\arrow["{(-)^{\ed}}", shift left, from=1-3, to=1-2]
				\arrow["{(-)^\vee}", shift left, from=2-2, to=2-3]
				\arrow["{(-)^{\ed}}", shift left, from=2-3, to=2-2]
				\arrow["{(-)^\vee}", shift left, from=3-2, to=3-3]
				\arrow["{(-)^{\ed}}", shift left, from=3-3, to=3-2]
				\arrow["\gr"', shift right, from=2-3, to=3-3]
				\arrow["{F^1}"', shift right, from=2-2, to=1-2]
				\arrow["{F^1}"', shift right, from=2-1, to=1-1]
				\arrow[shift right, from=3-1, to=2-1]
				\arrow[shift right, from=3-2, to=2-2]
				\arrow["\colim"', shift right, from=2-3, to=1-3]
				\arrow[shift right, from=3-3, to=2-3]
			\end{tikzcd}.	
		\end{equation}
		
		Here, the functors $(-)^\ed$ and $\adic$ are formally determined by the adjoint functor theorem, and the right adjoints of $\gr$ are $\triv$, the functors of associating trivial transfer maps.\ The \textit{Andr\'e-Quillen functor (of $\kd(\mathcal{P})$-algebras)} is defined as
		\[\aq:\alg_{\kd(\mathcal{P})}(\QC^\vee_R)\xrightarrow{\cobr\circ(-)^\ed}\alg_\mathcal{P}(\QC^\vee_{R})^{op},\] and the \textit{Andr\'e-Quillen functor with (completed) Hodge filtration}\footnote{This name is justified later in Proposition \ref{prop4.20old}.} refers to the composite
		\[\aqf:\alg_{\kd(\mathcal{P})}(\QC^\vee_R)\xrightarrow{\const}\alg_{\kd(\mathcal{P})}(\fil_{\le -1}\QC^\vee_R)\xrightarrow{\cobr\circ(-)^\ed}\alg_\mathcal{P}(\fil_{\ge1}\QC^\vee_{R})^{op}.\]
	\end{construction}
	\begin{prop}\label{prop3.20n}
		The functors $\adic$, $(-)_1$ and $\const$ in diagram (\ref{ecd}) are fully faithful.
	\end{prop}
	\begin{proof}
		The proof of Proposition \ref{prop3.15} also shows that the adjunctions $(-)_1\dashv F_1$ and $\colim\dashv\const$ are lying over the corresponding adjunctions of underlying modules (which justifies the symbols), thus there are equivalences $id\simeq F_1\circ(-)_1$ and $id\simeq \colim\circ\const$. To show that $\adic$ is fully faithful, observe that there is a commutative square of left adjoints
		\[\begin{tikzcd}
			{\QC^\vee_R} & {\alg_\mathcal{P}(\QC^\vee_R)} \\
			{\fil_{\ge1}\QC^\vee_R} & {\alg_\mathcal{P}(\fil_{\ge1}\QC^\vee_R)}
			\arrow["{\fr_\mathcal{P}}", from=1-1, to=1-2]
			\arrow["{\fr_\mathcal{P}}", from=2-1, to=2-2]
			\arrow["{(-)_1}"', from=1-1, to=2-1]
			\arrow["\adic", from=1-2, to=2-2]
		\end{tikzcd},\]so $F_1(\adic\circ\fr_\mathcal{P}(V))$ is conanically equivalent with $\fr_\mathcal{P}(V)$ for any $V\in\QC^\vee_R$. For general $A\in \alg_{\mathcal{P}}(\QC^\vee_R)$, conclude by noting that $A$ is equivalent to $|\br_\bullet(id,\fr_\mathcal{P}\circ\text{forget}_{\mathcal{P}},A)|$.
	\end{proof}
	
	The next proposition shows that $\aqf$ is truly a refinement of $\aq$.
	\begin{prop}\label{prop3.21n}
		The functor $\aqf$ is a sifted-colimit-preserving lift of $\aq$ along $F_1$.
	\end{prop}
	\begin{proof}
		Recall that the sifted colimits of algebra objects are calculated by that of the underlying modules, which results in that $\const$ is sifted-colimit-preserving. Since $(-)^\ed$ and $\cobr$ are right adjoints, the functor $\aqf\simeq \cobr\circ(-)^{\ed}\circ \const$ sends sifted colimits to sifted limits. Because (\ref{ecd}) commutes, we have $F_1\aqf\simeq \aq\circ\colim\circ\const\simeq \aq$.
	\end{proof}
	Our main interest of (\ref{ecd}) is that it applies to partition Lie algebras:
	\begin{df}cf. \cite[Definition 3.54, 3.88]{BCN}\label{df3.22n}
		If $R$ is a coherent connective $\eoo$-ring spectrum, the \textit{spectral PD $\infty$-operad $\lie$ of spectral partition Lie algebras over $R$} is defined as $\kd(\neoo)$. When $R$ is a coherent simplicial commutative ring, the \textit{derived PD $\infty$-operad $\dlie$ of derived partition Lie algebras over $R$} is defined as $\kd(\com^{nu})$.
		
		In these cases, the functor $\aq$ ($\aqf$) is preferably written as $\ce$ ($\cet$), which means \textit{Chevalley-Eilenberg functor (with completed Hodge filtration)} for partition Lie algebras.
	\end{df}
	
	Then we deduce a filtered PD Koszul duality through $\aqf$ under some finiteness conditions:
	\begin{notation}
		(1) Let $\alg^{\fil}_{\mathcal{P},\mathrm{qfafp}}\subset\alg_{\mathcal{P}}(\fil_{\ge1}\QC^\vee_R)$ be the full subcategory of complete algebras $A$ such that the natural map $\fr_{\mathcal{P}}([\gr_1A]_1)\simeq\gr(A)$ is an equivalence of graded $\mathcal{P}$-algebras, and $\gr_1(A)$ is an almost perfect $R$-module. Here $\qf$ stands for \textit{quasi-free} and $\afp$ means \textit{almost finitely presented}.
		
		(2) A $\kd(\mathcal{P})$-algebra is said to be \textit{dually almost perfect} if the underlying module is a dually almost perfect $R$-module. Such algebras form a full subcategory $\alg^{\dap}_{\kd(\mathcal{P})}\subset\alg_{\kd(\mathcal{P})}(\QC^\vee_R)$.
	\end{notation}

	\begin{theorem}[Filtered PD Koszul duality]\label{thm3.24n}
		Assume that $P\in \alg(\mathcal{C})$ satisfies the condition (\textdagger). The Andr\'e-Quillen functor with Hodge filtration $\aqf$ induces an equivalence
		\[\aqf:\alg^{\dap}_{\kd(\mathcal{P})}\simeq \alg^{\fil,op}_{\mathcal{P},\mathrm{qfafp}},\]whose homotopy inverse is given by restricting the following composite
		\[\dD:\alg_{\mathcal{P}}(\fil_{\ge1}\QC^\vee_R)\xrightarrow{\kd}\alg_{\kd(\mathcal{P})}(\fil_{\le -1}\QC^\vee_R)^{op}\xrightarrow{\colim}\alg_{\kd(\mathcal{P})}(\QC^\vee_R)^{op}.\]
	\end{theorem}
	\begin{proof}
		The core of this proof is calculating $\aqf(L)$ for $L\in\alg^{\dap}_{\kd(\mathcal{P})}$ (see Lemma \ref{inlemma1}), which is divided into several lemmas.
		
		Recall that $\Delta$ is the 1-category of simplices and let $\Delta_{sur}\subset \Delta$ be the subcategory of surjections. Let $\Delta^{\le n}$ (or $(\Delta_{sur})^{\le n}$) be the full subcategory of $\Delta$ (resp. $\Delta_{sur}$) spanned by the objects of size no more than $n+1$. For an \infcat\ $\mathcal{X}$, the \infcat\ of\textit{ (semi-)cosimpilicial objects in $\mathcal{X}$} is $\text{Fun}(N(\Delta),\mathcal{X})$ (resp. $\text{Fun}(N(\Delta_{sur}),\mathcal{X})$). A (semi-)cosimpilicial object is said to be \textit{$n$-skeletal} if it is a right Kan extension of its restriction to $N((\Delta_{sur})^{\le n})$.
		
		\begin{alert}
			Our notion of semi-cosimplicial objects is \textbf{not} standard. It is only used to facilitate determining skeletal cosimplicial objects. Despite of the following lemma, $\Delta_{sur}\subset\Delta$ is \textbf{not} coinitial.
		\end{alert}

		\begin{lemma}\label{inlemma4}
			Let $\mathcal{X}$ be an \infcat. A cosimplicial object in $\mathcal{X}$ is $n$-skeletal precisely when its restriction to $N(\Delta_{sur})$ is $n$-skeletal as a semi-cosimplicial object.
		\end{lemma}
		\begin{proof}
			This lemma can be regarded as a dual version of \cite[Lemma 6.5.3.8]{HTT}.
			It is sufficient to prove that, for each $[m]\in \Delta$, the simplicial nerve of
			\[(\Delta_{sur})_{[m]/}\times_{\Delta_{sur}}(\Delta_{sur})^{\le n}\subset \Delta_{[m]/}\times_{\Delta}\Delta^{\le n}\]
			is coinitial as a morphism of \infcats. Applying \cite[Theorem 4.1.3.1]{HTT}, it suffices to show that, for each object $\theta:[m]\to [k]$ of $\Delta_{[m]/}\times_{\Delta}\Delta^{\le n}$, considering the category $C$ of all factorizations
			\[[m]\xrightarrow{\theta'} [k']\xrightarrow{\theta''} [k]\] such that $\theta'$ is a surjection and $k'\le n$, its simplicial nerve $N(C)$ is weakly contractible. However, $C$ has a final object (where $\theta''$ is injective), so the lemma is proved.
			
		\end{proof}
		Set $D^{n}:=N((\Delta_{sur})_{[n]/}\times_{\Delta_{sur}}(\Delta_{sur})^{\le n-1})$. The next lemma controls the convergence of the cobar construction for reduced $\infty$-cooperads.
		\begin{lemma}\label{inlemma3}
			Assume that $\mathcal{Q}\in co\alg(\mathcal{C})$ is \textit{reduced}, i.e.\ $\mathcal{Q}(0)\simeq 0$ and $\mathcal{Q}(1)\simeq R$. The cosimplicial diagram $\cobr^\bullet(1,\mathcal{Q},1)$ is locally skeletal in the sense that the following difference\[\Omega_n:= \text{fib}\big(\mathcal{Q}^{\circ n}\to \lim_{D^n} \cobr^{\bullet}(1,\mathcal{Q},1)|_{D^n}  \big)\]is trivial at each arity $p\le n$. In particular, the cobar construction $\cobr(\mathcal{Q})\simeq \tot\big(\cobr^\bullet(1,\mathcal{Q},1)\big)$, at each chosen arity, is a finite totalization.
		\end{lemma}
		\begin{proof}
			This claim holds for $n=1$. We do induction on $n$. Take a cube $N(\Delta^1)^{\times n}$ and decorate each vertex $(a_0,\ldots,a_n)$ with $\mathcal{A}_1\circ \ldots\circ \mathcal{A}_n$, where $\mathcal{A}_i\simeq \mathcal{Q}$ if $a_i=0$ and $\mathcal{A}_i\simeq 1$ if $a_i=1$, and the edges of $N(\Delta)^{\times n}$ are morphisms given by applying the counit map of $\mathcal{Q}$ at the right position. The diagram $\cobr^{\bullet}(1,\mathcal{Q},1)|_{D^n}$ can be reformulated as the diagram of the decorated edges $(0,\ldots,0)\to (a_1,\ldots,a_n)$ such that $\sum|a_i|\ne 0$. This rephrasing leads to a commuting diagram
			\[\begin{tikzcd}
				{\mathcal{Q}^{\circ n}} & {\lim_{D^n}\cobr^{\bullet}(1,\mathcal{Q},1)|_{D^n}} & {\mathcal{Q}^{\circ (n-1)}} \\
				& ({\lim_{D^{n-1}}\cobr^{\bullet}(1,\mathcal{Q},1)|_{D^{n-1}})\circ\mathcal{Q}} & {\lim_{D^{n-1}}\cobr^{\bullet}(1,\mathcal{Q},1)|_{D^{n-1}}}
				\arrow[from=1-3, to=2-3]
				\arrow[from=2-2, to=2-3]
				\arrow["\beta"', from=1-2, to=2-2]
				\arrow[from=1-2, to=1-3]
				\arrow[from=1-1, to=1-2]
				\arrow["\alpha"', from=1-1, to=2-2]
			\end{tikzcd}\]
			where the square is cartesian in $\mathcal{C}$. The $n$th difference $\Omega_{n}$ fits into a fibre sequence,
			\[\Omega_n\to \text{fib}(\alpha)\xrightarrow{\gamma} \text{fib}(\beta),\]where $\gamma$ is equivalent to $\Omega_{n-1}\circ\mathcal{Q}\xrightarrow{id_{\Omega_{n-1}}\circ \epsilon}\Omega_{n-1}$ with $\epsilon$ being the counit of $\mathcal{Q}$. By induction hypothesis, $\Omega_{n}$ is trivial at arity $p\le n$. Then the second assertion follows from Lemma \ref{inlemma4}.
			
		\end{proof}
		\begin{lemma}\label{inlemma2}
			Supposing that $\mathcal{Q}\in co\alg(\mathcal{C})$ is reduced, the commutative diagram of adjunctions 
			\[\begin{tikzcd}
				& {\gr_{\ge1}\QC^\vee_R} \\
				{\alg_{\cobr(\mathcal{Q})}(\gr_{\ge1}\QC^\vee_R)} && {co\alg_{\mathcal{Q}}(\gr_{\ge1}\QC^\vee_R)}
				\arrow["\cobr", shift left=1, from=2-3, to=2-1]
				\arrow["{\mathrm{forget}_{\cobr(\mathcal{Q})}}"'{pos=0.8}, shift right=1, from=2-1, to=1-2]
				\arrow["{e^*}"', shift right=1, from=2-3, to=1-2]
				\arrow["{\fr_{\cobr(\mathcal{Q})}}"', shift right=1, from=1-2, to=2-1]
				\arrow[shift left=1, from=2-1, to=2-3]
				\arrow["{e_!}"', shift right=1, from=1-2, to=2-3]
			\end{tikzcd}\]arising from Theorem \ref{thm3.2n} induces a natural equivalence $\cobr(\mathcal{Q})\circ (-)\simeq e^*e_!$ of monads. 
		\end{lemma}
		\begin{proof}
			By Corollary 5.8 and Corollary 8.9 in \cite{Hau}, there is a morphism of monads $\cobr(Q)\circ (-)\simeq \mathrm{forget}_{\cobr(\mathcal{Q})}\circ\fr_{\cobr(\mathcal{Q})}\to e^*e_!$, whose underlying morphism of endofunctors is given by
			\[\tot\big(\cobr^\bullet(1,\mathcal{Q},1)\big)\circ V\to \tot\big(\cobr^\bullet(1,\mathcal{Q},e_!V)\big),\] for each $V\in\gr_{\ge 1}\QC^\vee_R$. We claim that this arrow is an equivalence. Indeed, at each weight $p\ge 1$,
			\begin{equation}\label{ineq1}
				\tot\big(\cobr^\bullet(1,\mathcal{Q},1)\big)\circ V (p)\to \big(\cobr^\bullet(1,\mathcal{Q},e_!V)\big)(p)
			\end{equation}
			only the components of arity $\le p$ are concerned since $\mathcal{Q}$ has no nullary operation, which means that both sides in (\ref{ineq1}) are finite totalizations as shown in Lemma \ref{inlemma3}. The product $\circ$ (or $\bcirc$) preserves finite totalizations \cite[Proposition 3.37]{BM}. Thus (\ref{ineq1}) is an equivalence for all $p$ and $V$.

		\end{proof}
		\begin{lemma}\label{inlemma1}
			For each $L\in\alg^{\dap}_{\kd(\mathcal{P})}$, the filtered Andr\'e-Quillen algebra $\aqf(L)$ lies in $\alg^{\fil}_{\mathcal{P},\mathrm{qfafp}}$, and $\gr(\aqf(L))$ is freely generated by $[L^\vee]_1$.
		\end{lemma}
		\begin{proof}
			The underlying module of $C:=(\const(L))^\ed$ is $(L^\vee)_1$ by Corollary \ref{cor3.17n}, which belongs to $\sfil_{\ge1}\aperf_R$. The cosimplicial diagram $\cobr^\bullet(1,\br(\mathcal{P}),C)$ consists of almost perfect (then complete by Proposition \ref{prop2.15}) terms. Thus $\aqf(L)\simeq \tot\cobr^\bullet(1,\br(\mathcal{P}),C)$ is complete.
			
			Now we determine the $\br(\mathcal{P})$-coalgebra structure of $\gr(C)$ using a dual version of \cite[Theorem 4.7.1.34]{HA}. Let $\mathcal{C}_{\red}$ be the full subcategory of $\mathcal{C}$ in Construction \ref{c3.1n} spanned by reduced objects, i.e.\ $\mathcal{Q}\in \mathcal{C}$ such that $\mathcal{Q}(0)\simeq 0$ and $\mathcal{Q}(1)\simeq R$ naturally. The \infcat\ $\mathcal{C}_{\red}$ inherits from $\mathcal{C}$ a monoidal structure and a left action on $\gr_{\ge1}\QC^\vee_R$. The \textit{(co)endomorphism category} $\mathcal{C}_{\red}[[L^\vee]_1]$ in the sense of \cite[Definition 4.7.1.1]{HA} consists of the pairs $(\mathcal{Q},\Delta)$ such that $\mathcal{Q}\in \mathcal{C}_{\red}$ and $\Delta:[L^\vee]_1\to \mathcal{Q}\circ[L^\vee]_1$. The natural monoidal projection $p:\mathcal{C}_{\red}[[L^\vee]_1]\to \mathcal{C}_{\red}$ is a left fibration \cite[Lemma 4.7.1.36]{HA} having constant fibres $\mathrm{End}_R(L^\vee)$. By the natural augmentation $\mathcal{Q}\to \mathbbb{1}$ in $\mathcal{C}_{\red}$, we have a second projection $q:\mathcal{C}_{\red}[[L^\vee]_1]\to p^{-1}(\mathbbb{1})\simeq \mathrm{End}_R(L^\vee)$. The projections $p$ and $q$ together induce a product $\mathcal{C}_{\red}[[L^\vee]_1]\simeq \mathcal{C}_{\red}\times \mathrm{End}_{R}(L^\vee)$.

			This product and Theorem 4.7.1.34 in \cite{HA} show that the space of $\br(\mathcal{P})$-coalgebra structures on $[L^\vee]_1$ is equivalent to that of $\mathbbb{1}$-coalgebras, which is contractible \cite[Proposition 4.2.4.9]{HA}. Therefore, there is a cocartesian edge $[L^\vee]_1\xrightarrow{e_!}\gr(C)$ in $\mathrm{LComod}(\gr_{\ge 1}\QC^\vee_R)$ lying over the coaugmentation $\mathbbb{1}\to \br(\mathcal{P})$ in $co\alg(\mathcal{C})$ (given by forgetful functor of coalgebras). Here we have a morphism $\eta:\fr_\mathcal{P}([L^\vee]_1)\to \cobr(e_![L^\vee]_1)\simeq \gr(\aqf(C))$ adjoint to $id_{\gr(C)}$ by Theorem \ref{thm3.2n}. Lemma \ref{inlemma2} shows that the arrow $\eta$ is  an equivalence.

		\end{proof}
		
		The remaining task is to show that $\dD|_{\alg^{\fil}_{\mathcal{P},\mathrm{qfafp}}}$ is a homotopy inverse of $\aqf|_{\alg^{\dap}_{\mathcal{P}}}$. Take any $A\in \alg^{\fil}_{\mathcal{P},\mathrm{qfafp}}$, the underlying module of $\gr(\br(A))$ is $[\gr_1A]_1$ due to Theorem \ref{thm3.2n}, so $\kd(A)$ has $((\gr_1A)^\vee)_{-1}$ as the underlying module, which is dually almost perfect. Thus the functor $\dD$ induces a functor $\alg^{\fil}_{\mathcal{P},\mathrm{qfafp}}\to \alg^{\dap,op}_{\kd(\mathcal{P})}$. Since $\const$ in (\ref{ecd}) is fully faithful, there is a contravariant adjunction
		$\aqf:\alg^{\dap}_{\kd(\mathcal{P})}\rightleftarrows\alg^{\fil,op}_{\mathcal{P},\mathrm{qfafp}}:\dD$.
		For each $L\in \alg^{\dap}_{\kd(\mathcal{P})}$, the unit map $\eta:L\to \dD\circ \aqf(L)$ is given by applying $\const\circ(-)^\vee$ to
		$a:\br\circ \cobr (C)\to C$ with $C:=(\const(L))^\ed$. Meanwhile, the calculation in Lemma \ref{inlemma1} shows that $\gr(a)$ is an equivalence. The completeness of $R$-linear dual (\ref{dualfil}) shows that $\eta$ is an equivalence as well. In the meanwhile, for each $A\in \alg^{\fil}_{\mathcal{P},\mathrm{qfafp}}$, the unit map $\xi:A\to \aqf\circ \dD(A)$ is a morphism of complete algebras and $\gr(\xi)$ is homotopic with $\gr(A)\to \cobr\circ\br( \gr(A))$ by Proposition \ref{cor3.17n}, which is an equivalence by Lemma \ref{inlemma1}, so $\xi$ is also an equivalence.

	\end{proof}
	\begin{remark}\label{rk.3.29n}
		Suppose that $R$ is eventually coconnective so that $\iota:\m_R\hookrightarrow \QC^\vee_R$ is a fully faithful embedding. If $\mathcal{P}$ is a (spectral, derived or restricted derived) $\infty$-operad, then $\alg^{\fil}_{\mathcal{P},\mathrm{qfafp}}$ can be identified with a full subcategory of $\alg_\mathcal{P}(\fil_{\ge 1}\m_R)$ by Proposition \ref{prop2.38}. Under this condition, the filtered Andr\'e-Quillen functor $\aqf$ embeds $\alg^{\dap}_{\kd(\mathcal{P})}$ (contravariantly) into the \infcat\ of ordinary $\mathcal{P}$-algebras with a decreasing filtration.
	\end{remark}
	In particular, Theorem \ref{thm3.24n} applies to partition Lie algebras:
	\begin{corollary}\label{cor3.30n}
		Let $R$ be a coherent connective $\eoo$-ring spectrum. The Chevalley-Eilenberg functor $\cet$ with Hodge filtration induces a fully faithful embedding
		\[\cet:\aplie\hookrightarrow \calg^{nu}(\fil_{\ge 1}\QC^\vee_R)^{op}\]from the \infcat\ of dually almost perfect spectral partition Lie algebras, into the \infcat\ of $R$-augmented filtered pro-coherent $\eoo$-algebras. Furthermore, the essential image of $\cet$ consists of completely filtered pro-coherent ring spectra $A$, such that $\sym_R[\gr_1A]_1\simeq \gr A$ naturally.
	\end{corollary}
	
	\begin{corollary}\label{cor3.31n}
		Let $R$ be a coherent simplicial commutative ring. The Chevalley-Eilenberg functor $\cet$ with Hodge filtration induces a fully faithful embedding
		\[\cet:\apdlie\hookrightarrow \dalg^{nu}(\fil_{\ge 1}\QC^\vee_R)^{op}\]from the \infcat\ of dually almost perfect derived partition Lie algebras into that of $R$-augmented filtered pro-coherent derived algebras. The essential image of $\cet$ consists of such $A$ that is complete and has $\lsym_R[\gr_1A]_1\simeq \gr A$ naturally.
	\end{corollary}
	
	This section is closed by a technique corollary for later usage:
	\begin{corollary}[Totalization of free algebras]\label{totalg}
		Let $R$ be a coherent simplicial commutative ring, and $M\to N^\bullet$ be a totalization diagram existing in $\aperf_R$. Then there is an induced totalization $\lsym_R[M]_1\simeq \tot\lsym_R[N^\bullet]_1$ of graded derived algebras.
	\end{corollary}
	\begin{proof}
		There is a geometric realization $M^\vee\simeq |(N^\bullet)^\vee|$ existing in $\aperf^\vee_{R}$, which can also be regarded as a geometric realization of abelian partition Lie algebras. Thus, there is $\cet(M^\vee)\simeq \tot \cet((N^\bullet)^\vee)$ by Theorem \ref{thm3.24n}. Conclude by considering the associated graded algebras.
	\end{proof}
	\subsection{Application: Lie algebra of Frobenius kernel over $\mathbb{F}_2$}
	Set $k=\mathbb{F}_2$ the field with two elements. The group of second roots of unity $\mu_2=\spec(k[x]/(x^2-1))$ is an infinitesimal flat group of interest. Now, we explicitly determine all the homotopy operations of its  derived partition Lie algebra $\mathrm{Lie}(\mu_2)^{\{2\}}$, where $\{2\}$ suggests the existence of $2$-restrictions.
	
	We begin by explaining the geometric meaning behind each term. As the fibre of the Frobenius $\varphi:\mathbb{G}_m\to \mathbb{G}_m$, the group scheme $\mu_2$ has $\dT_{\mu_2}\simeq k\oplus \sig^{-1}k$ as its tangent complex at the identity. This complex carries a natural $\kdlie$-structure such that $\ce(\dT_{\mu_2})\simeq k[x]/(x^2-1)$, \cite[Theorem 4.20]{BM}. Further, Theorem 4.23 in \cite{BM} shows that $\dT_{\mu_2}$ is a group object in derived partition Lie algebras. Let $\mathrm{Lie}(\mu_2)^{\{2\}}$ denote the $\kdlie$-algebra $B\dT_{\mu_2}$, which is the tangent of $B\mu_2$. Conversely, $\dT_{\mu_2}\simeq \Omega\mathrm{Lie}(\mu_2)^{\{2\}}$. In particular, the underlying module of $\mathrm{Lie}(\mu_2)^{\{2\}}$ is $\sig k\oplus k$.
	
	\begin{prop}\label{EXAprop1}
		(1) The homotopy operations on a $\kdlie$-algebra $L$ concentrating in degree $0$ and $1$ consist of $([-,-],(-)^{\{2\}})$ a restricted Lie structure on $\fg_1:=\pi_1(L)$, a $\fg_{1}$-representation structure on $\fg_0:=\pi_0(L)$ and a new additive operation $R^1:\fg_1\to \fg_0$.
		\[\begin{tikzcd}
			& {\fg_1} &&& {\fg_0} \\
			{} & {\fg_0} &&& {\fg_{-1}}
			\arrow["{[-,-]}", from=1-2, to=1-2, loop, in=150, out=210, distance=5mm]
			\arrow["{(-)^{\{2\}}}", from=1-2, to=1-2, loop, in=330, out=30, distance=5mm]
			\arrow["{R^1}", from=1-2, to=2-2]
			\arrow["{[-,-]}"', shift right, from=1-5, to=2-5]
			\arrow["{(-)^{\{2\}}}", shift left, from=1-5, to=2-5]
			\arrow["{[\fg_1,-]}", from=2-2, to=2-2, loop, in=150, out=210, distance=5mm]
		\end{tikzcd}\]
		
		(2) The homotopy  operations on a $\kdlie$-algebra $L$ concentrating in degree $0$ and $-1$ consist of a symmetric bilinear map $[-,-]$ and a weight $2$ operation $(-)^{\{2\}}$ satisfying
		\[(x+y)^{\{2\}}=(x)^{\{2\}}+[x,y]+(y)^{\{2\}}.\]
	\end{prop}
	\begin{proof}
		We only carry out the proof of (1). The same reasoning applies to (2) immediately.
		
		Recall that the homotopy operations of $\kdlie$-algebras are encoded by the homology of free algebras. Since $L$ concentrates at degree $0,1$, it suffices to consider $\pi_i(\kdlie(\sig k^{\otimes {l_1}}\oplus k^{l_2}))$ for $i=0,1$. Following \cite[Theorem 7.3]{BM}, $\pi_*(\kdlie(\sig k^{\otimes {l_1}}\oplus k^{l_2}))$ has a basis indexed by sequences $(i_1,\ldots,i_k,w)$. Here $w\in B(n_1,\ldots,n_{l_1},m_{1},\ldots,m_{l_2})$ is a \textit{Lyndon word} with $n_j$ ($m_j$) denoting how many times each letter at degree $1$ (or $0$) appears, and $|w|:=1-m_1-\ldots-m_{l_2}$. Besides, the integers $i_k$ satisfy: (1) $0\le i_k\le |w|$; (2) for all $1\le j<k$, $0\le i_j<2i_{j+1}$. Such a sequence has degree $|w|-l_2$. In particular, $\pi_1(\kdlie(\sig k^{\otimes {l_1}}\oplus k^{l_2}))$ is spanned by the sequences in the form of $(1,\ldots,1,(x))$ that all integers (if $k\ge1$) are $1$ and $x$ is a letter of degree $1$. As for $\pi_0(\kdlie(\sig k^{\otimes {l_1}}\oplus k^{l_2}))$, the basis has two types of vectors: (1) $(0,1,\ldots,1,(x))$ with $|x|=1$; (2) $(w)$, where $w$ is a Lyndon word containing exactly one (counting the multiplicity) letter at degree $0$.
		
		Now, we consider some ``atomic'' operations: let $[-,-]$ denote the operations $\fg_1\times\fg_i\to \fg_i$ ($i=0,1$) represented by $((xy))$ with $|x|=1$ and $|y|=0,1$, $(-)^{\{2\}}$ denote the operation $\fg_1\to \fg_1$ represented by the sequence $(1,(x))$ with $|x|=1$, and $R^1$ denote the operation $\fg_1\to\fg_0$ represented by the sequence $(0,(x))$ with $|x|=1$.
		
		Then, we find the relations between $[-,-]$, $(-)^{\{2\}}$ and $R^1$. In fact, Theorem 1.1 in \cite{zhang2022operations} describes thoroughly the relations between the homotopy operations of $\klie$-algebras. To pass to $\kdlie$, recall that the embedding $k[\sig]\hookrightarrow k[\mathcal{O}_\sig]$ induces a monoidal embedding $\sseq^\vee_k\to \sseq^{\gen,\vee}_k$, whose right adjoint maps $\com^{nu}$ to $\neoo$. Therefore, the unit map provides a morphism $f:\neoo\to \com^{nu}$ with $\neoo$ identified with its Borel derived operad. The PD Koszul dual of $f$ gives rise to a morphism of monads $\kd(f):\kdlie\to\klie$. At the same time, a direct calculation shows that, for each finitely dimensional graded $k$-space $V$, $\kdlie(2)\circ V\to \klie(2)\circ V$ is dual to the natural map 
		\[	c_V:(\sym^2_kV^\vee)[1]\to (\lsym^2_k V^\vee)[1].\]
		For $V=k$, $\pi_*(c_V)$ is obviously surjective. For $V=\sig k$, $\pi_*(\lsym^2_k \sig^{-1}k)$ is spanned by a generator $\tau^2$ at degree $-2$ and a generator $\st\tau$ at degree $-1$ given by Steenrod square. Thus, $\pi_*(c_{\sig k})$ is also surjective. Thus, $\kdlie(2)\circ (\sig k^{\oplus l_1}\oplus k^{\oplus l_2})\to \klie(2)\circ (\sig k^{\oplus l_1}\oplus k^{\oplus l_2})$ is injective on the level of homotopy groups. Further, the operations $[-,-]$, $(-)^{2}$ and $R^1$ are sent to the corresponding operations via $\kd(f)$ in \cite[Theorem 1.1]{zhang2022operations}, which satisfy the statement in this proposition.
		
		Moreover, notice that the homotopy operations for $\klie$-algebras represented by $(1,\ldots,1,w)$ with $|w|=1$, $(0,1,\ldots,1,w)$ with $|w|=1$ and $(w)$ with $|w|=0$ are non-null composites of $[-,-]$, $(-)^{\{2\}}$ and $R^1$. Therefore, they actually come from the homotopy operation of $\kdlie$-algebras and spanned by the ``atomic'' operations.
		
		To prove (2), it suffices to note that $c_{\sig^{-1}k}$ is surjective on the level of homotopy groups as $\lsym^2_k (k[1])\simeq 0$ and use the same reasoning.
		
	\end{proof}
	A $\kdlie$-algebra $L$ with perfect underlying module is determined by $\cet(L)$ (Theorem \ref{thm3.24n}). The binary operations on $\pi_*(L)$ are evident from this perspective.
	\begin{lemma}\label{EXAlemma2}
		The composition map $m:(\kdlie(2)\otimes L^{\otimes2})^{\sig_2}\to L$ of a perfect $\kdlie$-algebra $L$ is dual to the natural connection $d_L:L^\vee\to \lsym^2_k(L^\vee)[1]$ via $\cet(L)$.
	\end{lemma}
	\begin{proof}
		Set $H:=\kdlie((L)_{-1})$ the free filtered $\kdlie$-algebra, and consider the composition map $m^{\sfil}:H\to \const(L)$ of filtered $\kdlie$-algebras. Endow both sides with a second constant filtration, and then apply the Chevalley-Eilenberg functor in the bi-filtered context, we obtain a morphism of bi-complete algebras
		\[\Delta:F_\bullet F_* \ce(L)\to F_\bullet F_* \ce(\kdlie(L)),\]where $\gr_{\bullet}\Delta$ is $\lsym(m^{\fil,\vee}):\lsym^{\bullet}_k (L^\vee)_1\to \lsym^{\bullet}_kH^\vee$. The connection map of $\bullet$-filtrations give rise to a commutative square of filtered modules
		\[\begin{tikzcd}[column sep=2.5cm]
			{(L^\vee)_1} & {H^\vee} \\
			{(\lsym^2_k L^\vee)_2[1]} & {\lsym^2_k H^\vee[1]}
			\arrow["{m^{\fil,\vee}}", from=1-1, to=1-2]
			\arrow["{d_L}"', from=1-1, to=2-1]
			\arrow["{d_H}", from=1-2, to=2-2]
			\arrow["{\lsym^2_k(m^{\fil,\vee})}"', from=2-1, to=2-2]
		\end{tikzcd}.\]Applying the $*$-filtration truncation $F_{0\le *\le 2}$, we have
		\[\begin{tikzcd}
			{(L^\vee)_1} & {(L^\vee)_1\oplus ((\kdlie(2)\circ L)^\vee)_2} \\
			{(\lsym^2_k L^\vee)_2[1]} & {(\lsym^2_k L^\vee)_2[1]}
			\arrow["{(id,m^\vee)}", from=1-1, to=1-2]
			\arrow["{d_L}"', from=1-1, to=2-1]
			\arrow["{(\alpha_1,\alpha_2)}", from=1-2, to=2-2]
			\arrow["id"', from=2-1, to=2-2]
		\end{tikzcd}.\]Then, notice that there is a symmetric monoidal left adjoint $\oplus: \gr\QC^\vee_k\to \fil\QC^\vee_k$ sending the free graded algebra $H^{\sgr}:=\kdlie([L]_{-1})\simeq \gr H$ to $H$. Therefore, there is a splitting structure on $F_*\ce(H)\simeq \widehat{\oplus} \ce(H^{\sgr})$. Besides, there is an equivalence $\widehat{\oplus} \ce(H^{\sgr})\simeq \sqz([L^\vee]_{1})$ by the freeness of $H^{\sgr}$ and Proposition \ref{prop3.9n}. It means that $\alpha_1\simeq 0$ and $\alpha_2$ is the natural equivalence, i.e. $d_L\simeq \alpha_2\circ m^\vee$.

	\end{proof}
	\begin{exa}[Lie structure of $\mathrm{Lie}(\mathbb{G}_m)^{\{2\}}$]\label{EXAexa6.3}
	Let $\mathrm{Lie}(\mathbb{G}_m)^{\{2\}}$ denote the $\kdlie$-algebra	corresponding to $B\widehat{\mathbb{G}}_m$ by \cite[Theorem 4.23]{BM}, which has $\sig k$ as the underlying module. Since $B\widehat{\mathbb{G}}_m\simeq |\widehat{\mathbb{G}}^{\times \bullet}_m|$, the filtered Chevalley-Eilenberg algebra $A:=\cet(\mathrm{Lie}(\mathbb{G}_m))$ is the totalization of a cosimplicial resolution $k[[x]]^{\widehat{\otimes} \bullet}$ with all variables of weight $1$, which exhibits the formal group law of $\widehat{\mathbb{G}}_m$. In particular, the coface maps between the terms of degree $\le 3$ are as follows:
	\begin{equation}\label{cosimplicial res}
		k\mathrel{\substack{\textstyle\rightarrow\\[-0.4ex]
				\textstyle\rightarrow }}k[[t]]\mathrel{\substack{\textstyle\rightarrow\\[-0.4ex]
				\textstyle\rightarrow \\[-0.4ex]
				\textstyle\rightarrow}}k[[x_1,x_2]]\mathrel{\substack{\textstyle\rightarrow\\[-0.4ex]
				\textstyle\rightarrow \\[-0.4ex]
				\textstyle\rightarrow\\[-0.4ex]
				\textstyle\rightarrow}}
		k[[y_1,y_2,y_3]]	
	\end{equation}
	$t$ is sent to $x_2$, $x_1+x_2+x_1x_2$, $x_1$ by the three coface maps, and $(x_1,x_2)$ is mapped to $(y_2,y_3), (y_1+y_2+y_1y_2,y_3), (y_1,y_2+y_3+y_2y_3)$ or $(y_1,y_2)$ respectively by the four coface maps.

	The homology of graded pieces $\pi_*\gr(A)$ is known by \cite[Theorem 4.0.1]{priddy1973mod}. In fact, it is the polynomial algebra $\pi_*\gr(A)\cong k[\tau,\sq \tau, \sq^{\circ2} \tau, \ldots]$, where $\sq^{\circ n}\tau$ is of degree $-1$ and weight $2^n$. In particular, there is $\gr_2(A)\simeq k\tau^2\oplus k\sq \tau$. Considering the $\gr_1$ piece of (\ref{cosimplicial res}), there is one non-null class $[t]$ representing $\tau$. The $\gr_2$ piece of (\ref{cosimplicial res}) has two classes $[t^2]$ of degree $-1$ and $[x_1x_2 ]$ of degree $-2$, which represent $\sq\tau$ and $\tau^2$ respectively. Moreover, there is $\partial t=x_1x_2$ in (\ref{cosimplicial res}), which means that the connection map $d:\gr_1A\to \gr_2 A[1]$ sends $\tau=[t]$ to $\tau^2=[x_1x_2]$. The dual of $d$ shows that $\mathrm{Lie}(\mathbb{G}_m)^{\{2\}}$ has null Lie bracket and an identical restriction map, which recovers the ordinary restricted Lie algebra of $\mathbb{G}_m$ (up to a shifting).
	\end{exa}
	\begin{prop}\label{EXAprop4}
		The homotopy operations of $\mathrm{Lie}(\mu_2)^{\{2\}}$, having the underlying module $\sig\dT_{\mu_2}\simeq kD_1\oplus kD_0$ ($|D_i|=i$), can be described as
		
		\[\begin{tikzcd}
			& {kD_1} \\
			{} & {kD_0}
			\arrow["{[-,-]=0}", from=1-2, to=1-2, loop, in=150, out=210, distance=5mm]
			\arrow["{(D_1)^{\{2\}}=D_1}", from=1-2, to=1-2, loop, in=330, out=30, distance=5mm]
			\arrow["{D_1\mapsto D_0}", from=1-2, to=2-2]
			\arrow["{[D_1,-]=0}", from=2-2, to=2-2, loop, in=150, out=210, distance=5mm]
		\end{tikzcd},\]which consists of a trivial Lie bracket $[-,-]$, an identical restriction $(D_1)^{\{2\}}$ and an additive higher restriction $R^1(D_1)=D_0$.
	\end{prop}
	\begin{proof}
		Write $L:=\mathrm{Lie}(\mu_2)^{\{2\}}$ for short.\ Lemma \ref{EXAlemma2} shows that it suffices to understand the connection $d:L^\vee\to \lsym^2_k(L^\vee)[1]$ via $\cet(L)$.\ There is a completed pushout diagram ($A$ as in \ref{EXAexa6.3})
		\[\begin{tikzcd}
			A & A \\
			k & {\cet(L)}
			\arrow["\varphi", from=1-1, to=1-2]
			\arrow[from=1-1, to=2-1]
			\arrow[from=1-2, to=2-2]
			\arrow[from=2-1, to=2-2]
		\end{tikzcd},\]
		whose graded pieces induce a commutative cube
		\[\begin{tikzcd}
			&& {\gr_2A[1]} & {\gr_2A[1]} \\
			{k\tau'} & {k\tau} & 0 & {\gr_2\cet(L)[1]} \\
			0 & {k\upsilon\oplus k\tau}
			\arrow["{\varphi\simeq0}", from=1-3, to=1-4]
			\arrow[from=1-3, to=2-3]
			\arrow[from=1-4, to=2-4]
			\arrow[from=2-1, to=1-3]
			\arrow["{\varphi\simeq0}"', from=2-1, to=2-2]
			\arrow[from=2-1, to=3-1]
			\arrow[from=2-3, to=2-4]
			\arrow[from=3-1, to=2-3]
			\arrow[from=3-1, to=3-2]
			\arrow["d"', from=3-2, to=2-4]
			\arrow[crossing over, from=2-2, to=1-4]
			\arrow[crossing over, from=2-2, to=3-2]
		\end{tikzcd}.\]
		The homotopy group of $\gr\cet(L)$ is $k[\upsilon,\tau,\sq\tau,\ldots]$ by \cite[Theorem 4.0.1]{priddy1973mod}, where $\upsilon=[D_0^\vee]$ is of degree $0$ and $\tau=[D_1^\vee]$ is still of degree $-1$. The right vertical face of the above cube implies that $\pi_{-1}(d):\tau\mapsto \tau^2$ following Example \ref{EXAexa6.3}. The map $\pi_0(d):k\upsilon\to \pi_{-1}(\gr_2\cet(L))$ is encoded by the front vertical face post composed with $d$, regarded as a loop in $\map(k\tau',\gr_2\cet(L)[1])$. The above cube is an equivalence between this $2$-morphism with another one given by the upper horizontal face post composed by the natural inclusion $\gr_2 A[1]\to \gr_2\cet(L)[1]$, since all the other faces are null-homotopic through the triviality of $F_{*\ge1}k$. Then, recall that $\varphi:k\tau'\to k\tau$ is null-homotopic by the explicit model in Example \ref{EXAexa6.3} and the lifting of $\varphi$ along $\gr_2A\to F_1/F_3 A\to \gr_1A\simeq k\tau$. Therefore, the upper horizontal face represents a map $k\tau'\to \gr_2 A$ agreeing with the lifting of $\varphi$, which is $\tau'=[t]\mapsto[t^2]=\st \tau$ following Example \ref{EXAexa6.3}. It means that $\pi_0(d):\upsilon\mapsto \sq\tau$, or dually $R^1:D_1\mapsto D_0$ and $[D_1,D_0]=0$.
	\end{proof}
	A similar but easier process shows that $\cet(\dT_{\mu_2})=(k[\upsilon,\varepsilon]/(\varepsilon^2),d)$ with $|\varepsilon|=1$, where the graded mixed structure $d$ sends $\varepsilon$ to $\upsilon$. It means that $\dT_{\mu_2}\simeq kD_0\oplus kD_{-1}$ has vanishing Lie bracket and $(D_0)^{\{2\}}= D_{-1}$.

	\section{Partition Lie algebroids and Koszul duality}\label{sec4}

	We define \textit{(derived) partition Lie algebroids over $B$} (relative to $R$), in the spirit of \cite{BW}, as the algebras of a sifted-colimit-preserving monad $T$ acting on $\QC^{\vee}_B$ (Definition \ref{df4.4}).\ Here $T$ receives a monad morphism from $\bdlie$. Their interplay with filtrations will be detailed in \S\ref{sec4.2}.
	
	Our purpose is to establish a ``many-object'' generalization of Corollary \ref{cor3.31n}.\ We present in Theorem \ref{thm4.25} an equivalence between \textit{dually almost perfect} partition Lie algebroids with a special types of complete filtered $R$-derived algebras,\ named as \textit{foliation-like algebras} (Notation \ref{n4.14}), which generalize the derived foliations in the sense of \cite[\S1.2]{TV} into general characteristics.

	\subsection{Partition Lie algebroids}\label{sec4.1}
	Before introducing partition Lie algebroids, we review the basics of commutative algebras in pro-coherent modules via derived symmetric sequences (Definition \ref{df2.46}).
	\begin{recoll}\label{recoll4.1}		
		Consider the \infcat\ of coherent simplicial commutative rings $\mathrm{SCR}^{\scoh}\subset\mathrm{SCR}$. From Theorem \ref{theorem2.29}, Lemma \ref{lemma2.47} and the adjoint functor theorem, the (pro-coherent) derived symmetric sequences with different coefficient algebras are encoded in the following diagram
		\[\begin{tikzcd}
			{\mathrm{SCR}^{\scoh}} && {\pr^{L}}
			\arrow["\sseq^{\gen}"=0, from=1-1, to=1-3, bend left]
			\arrow["\sseq^{\gen,\vee}"', name=1, from=1-1, to=1-3, bend right]
			\arrow["\iota",shorten <=6pt, shorten >=6pt, Rightarrow, from=0, to=1]
		\end{tikzcd}.\]	Here, the transfers along arrows $R\to B$ in $\mathrm{SCR}^{\scoh}$ give rise to the base change $B\otimes_R$ for (pro-coherent) derived symmetric sequences, whose right adjoints are restrictions of scalars.
		
		Both the transfers and the natural transformation $\iota$ are symmetric monoidal with respect to $\otimes$ and $\levtimes$, monoidal with respect to $\circ$ and $\bcirc$. In particular, they always preserve $\com$, as the unit of $\levtimes$ and an associative algebra of $\circ$. Then, for each $R\to B$ in $\mathrm{SCR}^{\scoh}$, there is an adjunction
		\[B\otimes_R(-):\dalg(\QC^\vee_R)\rightleftarrows\dalg(\QC^\vee_B):\mathrm{rest}.\]Moreover, this adjunction is monadic by Barr-Beck-Lurie's theorem, which identifies $\dalg(\QC^\vee_B)$ with the slice category $\dalg(\QC^\vee_R)_{B/}$.
	\end{recoll}
	Now, we construct a monad $T$ governing the partition Lie algebroids.
	\begin{construction}\label{c4.2}
		 
		The square zero extension of $B$-modules gives rise to functors into different algebra categories
		\[\sqzr:\QC^\vee_B\xrightarrow{\sqz_B}\dalg^{nu}(\QC^\vee_B)\xrightarrow{\mathrm{rest}}\dalg(\QC^\vee_R)_{/B},\]where $\sqz_{B}$ has a left adjoint $\cot_{B}$ given by $\br$ in Construction \ref{c3.19n}, and $\sqzr$ has a left adjoint $\cotr:=\cot_{B}(B\otimes_{R}-)$ called the \textit{relative cotangent fibre}. Composing with the $B$-linear dual defines a natural but unsatisfying monad on $\QC^\vee_B$
		\[\tbar:=(-)^\vee\circ \cotr\circ \sqzr\circ (-)^\vee.\]
	\end{construction}
	We are going to rectify $\tbar$ into a sifted-colimit-preserving monad using right-left extension.
	\begin{lemma}\label{lemma4.3}
		For each $V\in \aperf^\vee_B$, there is a functorially \textit{splitting} fibre sequence $\bdlie(V)\to \tbar(V)\to \dT_{B/R}(:= \dL^\vee_{B/R})$ in $\QC^\vee_B$.
	\end{lemma}
	\begin{proof}
		Observe that there is a natural transformation $\cotr\circ\sqzr\to \cot_B\circ\sqz_B$ given by the counit map of $(B\otimes_R-)\dashv \mathrm{rest}$. It induces a natural cofibre sequence
		\[\dL_{B/R}\simeq\cotr\circ\sqzr(0)\to\cotr\circ\sqzr(V^\vee)\to \cot_B\circ\sqz_B(V^\vee).\]Indeed, applying $\map_{\QC^\vee_B}(-,M)$ on it gives a homotopy pullback
		\[\begin{tikzcd}
			{\map_{\dalg^{nu}(\QC^\vee_B)}(\sqz_B(V),\sqz_B(M))} & {\{f\}} \\
			{\map_{\dalg(\QC^\vee_R)_{/B}}(\sqzr(V),\sqzr(M))} & {\map_{\dalg(\QC^\vee_R)_{/B}}(B,\sqzr(M))}
			\arrow[from=1-1, to=1-2]
			\arrow[from=1-1, to=2-1]
			\arrow[from=1-2, to=2-2]
			\arrow[from=2-1, to=2-2]
		\end{tikzcd},\]where $f:B\to \sqzr(M)$ is given by $0\to M$. Taking $B$-linear dual gives the wanted fibre sequence as $\bdlie(V)\simeq (\cot_B\circ\sqz_B(V^\vee))^\vee$ canonically \cite[Proposition 3.53]{BCN}. The splitting is given by applying $\tbar$ on the retraction $0\to V\to 0$.
	\end{proof}
	Assuming that the algebraic cotangent complex $\dL_{B/R}$ is almost perfect, the above lemma shows that $\tbar|_{\aperf^\vee_B}\in\mathrm{End}_{\sigma,reg}(\aperf_B^\vee)$.\ Additionally, $\bdlie|_{\aperf^\vee_B}\to\tbar|_{\aperf^\vee_B}$ is a morphism of monads, i.e. a morphism of algebras in $\mathrm{End}_{\sigma,reg}(\aperf_B^\vee)$ with respect to the composition product\cite[Corollary 5.8, 8.9]{Hau}. Then, left Kan extension induces a morphism of sifted-colimit-preserving monads $\bdlie\to T$ acting on $\QC^\vee_B$, Remark \ref{rk2.29}.
	
	\begin{df}\label{df4.4}
		The \infcat\ ${\lagd}$ of \textit{(derived) partition Lie algebroids over $B$ (relative to $R$)} is defined as the \infcat\ $\alg_T(\QC^\vee_R)$ of $T$-algebras.
	\end{df}
	\begin{construction}\label{c4.5}
		The forgetful functor along $\bdlie\to T$ maps $\fr_T(0)$ (an initial object of $\lagd$) to a $\bdlie$-algebra given by $\kd(B\otimes_R B)$  (\ref{df3.18n}). Therefore, there is a functor
		\[\rmfib:\lagd\to (\bdlie)_{\kd(B\otimes_R B)/}\]of taking the ``\textit{fibres}'', where $\alg_{\bdlie}(\QC^\vee_B)$ is written as $\bdlie$ for short.
		
		Further forgetting to pro-coherent $B$-modules, we obtain a commuting diagram of right adjoints
		\begin{equation}\label{e: fibre and forgetful}
			\begin{tikzcd}
				\lagd & {(\QC^\vee_B)_{/\dT_{B/R}[1]}} \\
				{(\QC^\vee_B)_{\dT_{B/R}/}} & {\QC^\vee_B}
				\arrow["\rmfor", from=1-1, to=1-2]
				\arrow["\rmfib"', from=1-1, to=2-1]
				\arrow["\rmfib", from=1-2, to=2-2]
				\arrow["{\mathrm{cofib},\simeq}"', from=2-1, to=1-2]
				\arrow[from=2-1, to=2-2]
			\end{tikzcd},
		\end{equation}
		where $\mathrm{cofib}$ is a categorical equivalence sending each $f:\dT_{B/R}\to M$ to $\rho:\mathrm{cofib}(f)\to \dT_{B/R}[1]$. We call the top horizontal arrow the \textit{forgetful functor} of partition Lie algebroids, whose left adjoint $\fr_{\lagd}$ is called the free partition Lie algebroid functor.
		
		A partition Lie algebroid can be then informally written as $\rho:\fg\to \dT_{B/R}[1]$. The forgetful functor of $T$-algebras is nothing but taking the fibre of $\rho$.
	\end{construction}
	
	Next, we construct a Chevalley-Eilenberg functor $C^*$ for partition Lie algebroids.
	\begin{notation}\label{n4.6}
		A partition Lie algebroid $L$ is said to be \textit{dually almost perfect} if $\rmfor(L)$ belongs to $(\aperf^\vee_B)_{/\dT_{B/R}[1]}$. Let $\aplagd\subset \lagd$ be the full subcategory of dually almost perfect partition Lie algebroids, which is naturally equivalent to $\alg_{T}(\aperf_B^\vee)$ by Lemma \ref{lemma4.3}.
		
		Denote by $\cwafp$ the full subcategory of $\dalg(\QC^\vee_R)_{/B}$ consisting of such $A(\to B)$ that the relative cotangent fibre $\cotr(A)$ is an almost perfect $B$-module, where $\mathrm{wafp}$ means \textit{weakly almost finitely presented}. 
	\end{notation}
	\begin{prop}[Chevalley-Eilenberg functor]\label{prop4.7}
		There is an adjunction of \infcats\[C^*:\lagd\rightleftarrows\big(\dalg(\QC^\vee_R)_{/B}\big)^{op}:\fD\]such that $\rmfor\circ \fD$ maps $A\to B$ to $\dT_{B/A}[1]\to\dT_{B/R}[1]$, where the left adjoint functor $C^*$ is called the  \textnormal{Chevalley-Eilenberg functor} of partition Lie algebroids.
	\end{prop}
	\begin{proof}
		We start with the adjunction $\sqzr\circ(-)^\vee:\aperf^\vee_{R}\rightleftarrows (\cwafp)^{op}:(-)^\vee\circ\cotr$ of full subcategories guaranteed by Lemma \ref{lemma4.3}. It gives rise to the following adjoint by abstract nonsense
		\[\bar{C}^*:\aplagd\simeq \alg_{\tbar}(\aperf^\vee_B)\rightleftarrows(\cwafp)^{op}:\bar{\fD}.\]Define $\fD:\dalg(\QC^\vee_R)_{/B}\to (\lagd)^{op}$ as the left Kan extension of \[\cwafp\xrightarrow{\bar{\fD}}(\aplagd)^{op}\hookrightarrow(\lagd)^{op}.\]
		
		We claim that $\fD$ is small-colimit-preserving. Since $\sqzr$ respects filtered colimits, its left adjoint $\cotr$ preserves compactness, thus $\cwafp$ includes the full subcategory of compact objects $\mathcal{C}^{\omega}$. By the uniqueness of Kan extension, $\fD$ is equivalent to the left Kan extension of $\bar{\fD}|_{\mathcal{C}^{\omega}}$, and further $\bar{\fD}|_{\mathcal{C}^{\omega}}$ preserves finite colimits. Observe that $\dalg(\QC^\vee_R)_{/B}$ is presentable and $(\lagd)^{op}$ is cocomplete (Corollary \ref{cor4.9} as below), thus one may use the reasoning of \cite[Proposition 5.5.1.9]{HTT} to say that $\fD$ is small-colimit-preserving, and it admits a right adjoint $C^*$.
		
		Finally, note that $\rmfor\circ\bar{\fD}(A\to B)\simeq \big(\dT_{B/A}[1]\to\dT_{B/R}[1]\big)$ for $A\to B$ in $\cwafp$. It also holds in general, since the algebraic cotangent complex $\dL_{B/-}$ is small-colimit-preserving.
	\end{proof}
	We then compare the above Chevalley-Eilenberg functor $C^*$ with $\ce$ for $B$-partition Lie algebras, defined in \S\ref{sec3.4}. To do so, we require some categorical preparation.

	\begin{lemma}\label{lemma4.8}
		Let $\mathcal{C}$ be a presentable \infcat\ and $T$ a sifted-colimit-preserving monad on it. Then the \infcat\ $\alg_T(\mathcal{C})$ is presentable.
	\end{lemma}
	\begin{proof}
		Since $\mathcal{C}\simeq \text{Ind}_{\kappa}(\mathcal{C}_0)$ for some regular cardinal $\kappa$ and small \infcat\ $\mathcal{C}_0$, the \infcat\ $\text{End}_{cont,\kappa}(\mathcal{C})\simeq \text{Fun}(\mathcal{C}_0,\mathcal{C})$ of $\kappa$-continuous endofunctors is accessible \cite[Proposition 5.4.4.3]{HTT}. Since $\mathcal{C}$ is left tensored by $\text{End}_{cont,\kappa}(\mathcal{C})$, $\alg_{T}({\mathcal{C}})$ is accessible \cite[Corollary 3.2.3.5]{HA}. Now it remains to check that $\alg_{T}$ admits all small colimits. But $T$ is sifted-colimit-preserving, then $\alg_{T}(\mathcal{C})$ admits sifted colimits \cite[Corollary 4.2.3.5]{HA}, and the forgetful functor to $\mathcal{C}$ creates sifted colimits. The coproduct of two free algebras and the initial object are clear. The coproduct of two general algebras can be found as follows: for each $A_i\in \alg_{T}(\mathcal{C})$, $i=1$ or $2$, we have $A_i\simeq|\br_\bullet(T,T,A_i)|$ by Barr-Beck-Lurie theorem, then $A_1\coprod_{\alg_{T}}A_2$ can be obtained by the geometric realization of the diagonal of $T\big(\br_\bullet(id,T,A_1)\coprod_{\mathcal{C}}\br_\bullet(id,T,A_2)\big)$.
	\end{proof}
	\begin{corollary}\label{cor4.9}
		The \infcat\ $\lagd$ of partition Lie algebroids is presentable.
	\end{corollary}
	\begin{lemma}\label{lemma4.10}
		Consider a cocomplete \infcat\ $\mathcal{D}$. Let $\mathscr{A}$ be a coherent additive \infcat, and $T$ be a sifted-colimit-preserving monad on $\QC^\vee_{\mathscr{A}}$ which preserves the full subcategory $\aperf_{\mathscr{A}}^\vee$. Write $T':=T|_{\aperf_{\mathscr{A}}^\vee}$. If the following $F^0$ satisfies that $f^0:=F^0\circ \fr_{T'}$ belongs to $\text{Fun}_{\sigma,reg}(\aperf_{\mathscr{A}}^\vee,\mathcal{D})$\[
		\begin{tikzcd}
			{\alg_{T'}(\aperf_{\mathscr{A},\eqslantless 0}^\vee)} & {\mathcal{D}} \\
			{\alg_T{\QC^\vee_\mathscr{A}}}
			\arrow[hook, from=1-1, to=2-1]
			\arrow["{F^0}", from=1-1, to=1-2]
			\arrow["F"', dashed, from=2-1, to=1-2]
		\end{tikzcd}\]then there exists uniquely a sifted-colimit-preserving filler $F$ up to a contractible space of choices.
	\end{lemma}
		\begin{proof}
		Remind us that $f^0$ can be uniquely (up to homotopy) extended into $f\in\text{Fun}_\sig(\QC^\vee_{\mathscr{A}},\mathcal{D})$ \cite[Proposition 2.40]{BCN}. The idea is to transfer this extension to $\alg_T{(\QC^\vee_\mathscr{A}})$. Here we generalize the reasoning of \cite[Construction 2.2.17]{Holeman}.
		
		Regard $\text{Fun}_\sig(\QC^\vee_{\mathscr{A}},\mathcal{D})$ $\big($resp. $\text{Fun}_{\sigma,reg}(\aperf_{\mathscr{A}}^\vee,\mathcal{D})$$\big)$ as a right tensored \infcat\ of $\text{End}_{\sig}(\QC^\vee_{\mathscr{A}})$ $\big($resp. $\text{End}_{\sigma,reg}(\aperf_{\mathscr{A}}^\vee)$$\big)$ by composition of functors, and write the corresponding \ooop\ as $\mathcal{F}^{\otimes}\to \mathcal{RM}^{\otimes}$ (resp. $(\mathcal{F^{\dap}})^{\otimes}\to \mathcal{RM}^{\otimes}$). Then, left Kan extension induces a fully faithful embedding $(\mathcal{F^{\dap}})^{\otimes}\hookrightarrow\mathcal{F}^{\otimes}$ of \ooop s by the uniquess of extension. Here follows \[\text{RMod}_{T'}\big(\text{Fun}_{\sigma,reg}(\aperf_{\mathscr{A}}^\vee,\mathcal{D})\big)\hookrightarrow \text{RMod}_{T}\big(\text{Fun}_\sig(\QC^\vee_{\mathscr{A}},\mathcal{D})\big)\]a fully faithful embedding which takes $f^0=F^0\circ\fr_{T'}$ to its left Kan extension $f$ with a canonical right $T$-module structure. At the same time, a direct calculation shows that the assignment $\Phi\mapsto \Phi\circ \fr_T$ induces a categorical equivalence
		\begin{equation}\label{e12}
			-\circ \fr_T:\text{Fun}_{\sig}(\alg_{T}(\QC^\vee_\mathscr{A}),\mathcal{D})\xrightarrow{\simeq}\text{RMod}_{T}\big(\text{Fun}_\sig(\QC^\vee_{\mathscr{A}},\mathcal{D})\big)
		\end{equation}
		whose inverse $B$ can be described as $B(\phi)\simeq |\br_\bullet(\phi,T,\text{forget})|$. The promised filler in the statement can be taken as $B(f)$.

	\end{proof}
	\begin{corollary}[Comparison of $\ce$ and $C^*$]\label{cor4.11}
		There is a commuting square of adjunctions
		\[\begin{tikzcd}
			\bdlie & {(\dalg^{nu}(\QC^\vee_B))^{op}} \\
			\lagd & {(\dalg(\QC^\vee_R)_{/B})^{op}}
			\arrow["\ce", shift left, from=1-1, to=1-2]
			\arrow["{\rho_0}"', shift right, from=1-1, to=2-1]
			\arrow["\kd", shift left, from=1-2, to=1-1]
			\arrow["{\mathrm{restrict}}"', shift right, from=1-2, to=2-2]
			\arrow["\rmfib"', shift right, from=2-1, to=1-1]
			\arrow["{C^*}", shift left, from=2-1, to=2-2]
			\arrow["{B\otimes_R(-)}"', shift right, from=2-2, to=1-2]
			\arrow["\fD", shift left, from=2-2, to=2-1]
		\end{tikzcd},\]where $\ce\dashv\kd$ is given in Construction \ref{c3.19n}, and $\bdlie$ refers to $\alg_{\bdlie}(\QC^\vee_B)$. Additionally, the underlying module of $\rho_0(\fh)$ is  $(\fh\xrightarrow{0}\dT_{B/R}[1])$.
	\end{corollary}
	\begin{proof}
		Recall that forgetting along $\bdlie\to T$ induces the left vertical right adjoint $\rmfib$, which creates sifted colimits by considering the underlying modules. Therefore it admits a left adjoint $\rho_0$ by the adjoint functor theorem. For each free $\fh=\bdlie(V)$, one can see that $\rho_0(\fh)$ is $\fr_T(V)$, which has the underlying module as $\bdlie(V)\xrightarrow{0}\dT_{B/R}[1]$, Lemma \ref{lemma4.3}. For general $\fh$, it suffices to consider the bar resolution $\br_\bullet(id,\bdlie,\fh)$.
		
		Now we verify that $C^*\circ \rho_0\simeq \mathrm{restrict}\circ \ce$. Recall that $(-)^\vee\circ\cotr$ factorizes into $(-)^\vee\circ \cot_B$ and $B\otimes_R(-)$, then there is a commutative diagram of left adjoints
		\[\begin{tikzcd}
			\apbdlie & {(\mathcal{C}^{\text{wafp}}_{\text{aug}})^{op}} \\
			\aplagd & {(\cwafp)^{op}}
			\arrow[from=1-1, to=1-2]
			\arrow[from=1-1, to=2-1]
			\arrow[from=1-2, to=2-2]
			\arrow[from=2-1, to=2-2]
		\end{tikzcd},\]by the generality of adjoints and their monads, where $\mathcal{C}^{\text{wafp}}_{\text{aug}}$ is defined in the same way like $\cwafp$ but for augmented $B$-algebras. Observe that the above diagram can be recovered by restrict the diagram stated in the corollary. Therefore, Lemma \ref{lemma4.10} implies that $C^*\circ \rho_0\simeq \mathrm{restrict}\circ \ce$, since both ends are left adjoints and their restrictions to $\apbdlie$ are homotopic.
	\end{proof}
	
	\begin{remark}\label{rk4.12}
	There is a global version of Definition \ref{df4.4} in \cite{BMN}. Let $X$ be a locally coherent qcqs derived scheme, and $\QC^\vee(X):=\mathrm{Ind}(\coh^{op}(X))$ be the \infcat\ of pro-coherent sheaves. There exists a sifted-colimit-preserving monad $\mathrm{LieAlgd}^{\pi}_{\Delta}$ acting on $\QC^\vee(X)_{/\dT_{X/R}[1]}$ obtained by gluing up the monad of partition Lie algebroids on each affine chart\cite[Theorem 4.2]{BMN}.
		
	We recall their construction as follows: Let $\scr^{\scoh,\laft}_R\subset\scr_R$ be the (not full) subcategory of all coherent $R$-algebras and almost finitely presented morphisms, and $\QC^\vee_{/\dT[1]}$ be the \infcat\ of the pairs $(B,M\to \dT_{B/R}[1])$, where $B\in\scr^{\scoh,\laft,op}_R$ and $M\in \QC^\vee_{B}$. In the light of \cite[Definition 4.16]{BMN}, there is a unique relative monad $\mathrm{LieAlgd}^{\pi}_{\Delta}$ (which they denote as $\mathrm{Lie}^{\pi}_{\Delta}$)
		\[\begin{tikzcd}
			{\QC^\vee_{/\dT[1]}} && {\QC^\vee_{/\dT[1]}} \\
			& {\scr^{\scoh,\laft,op}_R}
			\arrow["{\mathrm{LieAlgd}^{\pi}_{\Delta}}", from=1-1, to=1-3]
			\arrow[from=1-1, to=2-2]
			\arrow[from=1-3, to=2-2]
		\end{tikzcd}\]
		that preserves sifted colimits fibrewise and specializes to the monad of the free-forgetful adjunction in Notation \ref{n4.6} on each fibre. They also prove that the resulting projection $\mathrm{LieAlgd}^{\pi}_{/R,\Delta}\to\scr^{\scoh,\laft,op}_R$ is a cartesian fibration, and $B\mapsto \lagd$ has \'etale descent \cite[Proposition 4.18]{BMN}. Then, $\mathrm{LieAlge}^{\pi}_{X/R,\Delta}$ is defined as \[\lim_{\spec(B)\subset X} \lagd.\]
	\end{remark}

	\subsection{Filtered Chevalley-Eilenberg functor}\label{sec4.2}
	We first discuss some properties of (completed) Hodge-filtered derived infinitesimal cohomology, which is useful for constructing the filtered Chevalley-Eilenberg functor.
	\begin{construction}\label{c4.13}
	There is a natural transformation of small-limit-preserving functors
	\[(F_0\to\gr_0):\sseq^{\gen,\vee}_{\fil,R}\to \sseq^{\gen,\vee}_R,\]where both $F_0$ and $\gr_0$ are lax monoidal with respect to $\circ$, and this arrow respects the lax monoidal structure. Moreover, this natural transformation sends $\com$ (regarded as a filtered derived $\infty$-operad concentrating at degree 0) to the identity of $\com$ as an $\infty$-operad.
	
	The above abstract nonsense induces a right adjoint $(F_0\to \gr_0)$ as follows
	\[\infcohnewfunctor:\dalg(\QC^\vee_R)^{\Delta^1}\rightleftarrows\dalg(\fil_{\ge0}\QC^\vee_R):(F_0\to\gr_0),\]	where the left adjoint is called \textit{Hodge-filtered derived infinitesimal cohomology}. We also write $\infcohnewfunctor(S\to B)$ as $\mathrm{F^H_*}\mathbbb{\Pi}_{B/S}$ or $\infcohnewfunctor(B/S)$ for convenience. It will be shown in Remark \ref{rk4.15} that $\infcohnewfunctor$ is independent from the base ring $R$.
\end{construction}
\begin{remark}
	The notion of \textit{Hodge-filtered derived infinitesimal cohomology} follows Antieau's ongoing work \cite{Antieau}. We will discuss its relation to Grothendieck's infinitesimal cohomology at the end of this section.
\end{remark}
\begin{notation}\label{n4.14}
	(1) Write $\dalg(\fil_{\ge0}\QC^\vee_R)$ as $\dafilr$ for short. For each $B\in\mathrm{SCR}_R$, the \infcat\ $\dafilbr$ is defined by a cartesian diagram of \infcats
	\[\begin{tikzcd}
		\dafilbr & \dafilr \\
		{\{B\}} & {\dalg(\QC^\vee_R)}
		\arrow[from=1-1, to=1-2]
		\arrow[from=1-1, to=2-1]
		\arrow["{\gr_0}"', from=1-2, to=2-2]
		\arrow[from=2-1, to=2-2]
	\end{tikzcd}.\]Intuitively, $\dafilbr$ is the \infcat\ of the filtered $R$-algebras $A$ with a chosen $\gr_0A\simeq B$.
	
	(2) A filtered algebra $A\in \dafilr$ is said to be \textit{foliation-like} if $A$ is complete, $B:=\gr_0A$ is connective and $\lsym_B[\gr_1A]_1\to\gr A$ is a natural equivalence of graded algebras. Here, the $B$-module $\dL_A:=\gr_1A[1]$ is called the cotangent complex of $A$. If $\dL_A$ is an almost perfect $B$-module, we say that $A$ is \textit{almost perfect}.
	
	For each $B\in\mathrm{SCR}_R$, the \infcat\ $\dalg^{\fol}_{B/R}$ of \textit{foliation-like algebras over $B$ relative to $R$} is by definition the full subcategory of foliation-like objects in $\dafilbr$.
	
\end{notation}
\begin{remark}\label{rk4.15}
	Construction \ref{c4.13} is natural in $R$ using a filtered version of Construction \ref{recoll4.1}. There is a commutative diagram of left adjoints
	\[\begin{tikzcd}[column sep=huge]
		{\dalg(\QC^{\vee}_R)^{\Delta^1}} & {\mathcal{D}^{\fil}_{/R}} \\
		{\dalg(\QC^{\vee}_{R'})^{\Delta^1}} & {\mathcal{D}^{\fil}_{/R'}}
		\arrow["{\mathrm{F^H_*}\mathbbb{\Pi}^{(R)}}", from=1-1, to=1-2]
		\arrow["{R'\otimes_R}"', from=1-1, to=2-1]
		\arrow[from=1-1, to=2-1]
		\arrow[from=1-1, to=2-1]
		\arrow["{R'\otimes_R}", from=1-2, to=2-2]
		\arrow["{\mathrm{F^H_*}\mathbbb{\Pi}^{(R')}}"', from=2-1, to=2-2]
	\end{tikzcd},\]where we temporarily emphasize the base ring of $\mathrm{F^H_*}\mathbbb{\Pi}$. The unit map of  $\mathrm{F^H_*}\mathbbb{\Pi}^{(R')}\dashv (F_0\to \gr_0)$ gives rise to a natural transformation \[\mathrm{F^H_*}\mathbbb{\Pi}^{(R)}\circ \mathrm{rest}\to\mathrm{rest}\circ\mathrm{F^H_*}\mathbbb{\Pi}^{(R')},\]which we claim to be an equivalence: for each $(S\to B)\in\dalg(\QC^{\vee}_{R'})^{\Delta^1}$, the filtered algebra $\mathrm{F^H_*}\mathbbb{\Pi}^{(R)}(B/S)$ naturally admits a $S$-algebra structure and then naturally a $R'$-algebra as well. Therefore, we have a natural arrow $\mathrm{F^H_*}\mathbbb{\Pi}^{(R')}(B/S)\to \mathrm{F^H_*}\mathbbb{\Pi}^{(R)}(B/S)$ that is an inverse of the above natural transformation.
\end{remark}
The next lemma shows that $(\infcohnewfunctor(B/-))^\wedge$ are naturally foliation-like over $B$ relative to $R$.
\begin{lemma}\label{lemma4.16}
	The graded algebra $\gr\infcohnewfunctor(B/S)$ is naturally equivalent to $\lsym_B[\dL_{B/S}[-1]]_1$ the freely generated graded $B$-algebra.
\end{lemma}
\begin{proof}
	The unit morphism of $\infcohnewfunctor\dashv (F_0\to \gr_0)$ induces a natural $B\to \gr_0\infcohnewfunctor(B/S)$. The universal property of algebraic cotangent complex\footnote{a pro-coherent version of \cite[Construction 4.4.10]{Raksit}.} determines $\dL_{B/S}\to \gr_1\infcohnewfunctor(B/S)[1]$. So we obtain a comparison map $c(B/S):\lsym_B[\dL_{B/S}[-1]]_1\to \gr\infcohnewfunctor(B/S)$ natural in the arrow $S\to B$. Then, consider the following commuting diagram of right adjoints
	\[\begin{tikzcd}[column sep=2cm]
		{\dalg(\gr_{\ge0}\QC^\vee_R)} & \dafilr & {\dalg(\QC^\vee_R)^{\Delta^1}} \\
		{\gr_{\ge0}\QC^\vee_R} && {(\QC^\vee_R)^{\Delta^1}}
		\arrow["\triv", from=1-1, to=1-2]
		\arrow["\rmfor"', from=1-1, to=2-1]
		\arrow["{(F_0\to\gr_0)}", from=1-2, to=1-3]
		\arrow["\rmfor", from=1-3, to=2-3]
		\arrow["{M_\star\mapsto (M_0\to M_0\oplus M_1[1])}"', from=2-1, to=2-3]
	\end{tikzcd},\]where $\gr\dashv \triv$, and the bottom arrow has a left adjoint sending $f:N_0\to N_1$ to $[N_1]_0\oplus[\rmfib(f)]_1$. The commutativity of the left adjoints then gives rise to an equivalence natural in $f:N_0\to N_1$
	\[\lsym_R[N_1]_0\otimes_R\lsym_R[\rmfib(f)]_1 \xrightarrow{\simeq} \gr\infcohnewfunctor(\lsym_RN_1/\lsym_RN_0).\]Notice that $\dL_{\lsym_RN_1/\lsym_RN_0}$ is naturally equivalent to $\lsym_RN_1\otimes_R \mathrm{cofib}(f)$, so the above line is equivalent to $c(\lsym_R(N_1)/\lsym_R(N_0))$. For the general case, it sufficient to note that both ends of $c(B/S)$ are sifted-colimit-preserving.
\end{proof}

\begin{corollary}\label{cor4.17}
	The adjunction $\infcohnewfunctor:\dalg(\QC^\vee_R)^{\Delta^1}\rightleftarrows\dafilr:(F_0\to\gr_0)$ is relative to $\dalg(\QC^\vee_R)$ in the sense of \cite[\S7.3.2]{HA}, where the left-hand side projects by $\mathrm{ev}_1:(A_0\to A_1)\mapsto A_1$ and the right-hand side projects by $\gr_0$.
\end{corollary}
\begin{corollary}\label{cor4.18}
	For each $B\in \mathrm{SCR}_R$, there is an adjunction
	\[\infcohnewfunctor(B/-):\dalg(\QC^\vee_R)_{/B}\rightleftarrows \dafilbr:F^0.\]In particular, the derived infinitesimal derived cohomology $\infcohR$ is an initial object in $\dafilbr$.
\end{corollary}
\begin{remark}\label{rk4.19}
	One can also see that the adjunction
	\[\infcohnewfunctor(-/R):\dalg(\QC^\vee_R)\rightleftarrows\dafilr:\gr^0\]exhibits a colocalization. Then, $\dafilbr$ can be regarded as a full subcategory of $(\dafilr)_{\infcohR/}$. More generally, consider a colocalization $L$ of an \infcat\ $\mathcal{C}$ and the following diagram
	\[
	\begin{tikzcd}
		{\mathcal{C}} & {\text{Fun}(\Delta^1,\mathcal{C})} \\
		{L\mathcal{C}} & {\mathcal{C}}
		\arrow[hook, from=2-1, to=2-2]
		\arrow["{L\to id_\mathcal{C}}", hook, from=1-1, to=1-2]
		\arrow["L"', from=1-1, to=2-1]
		\arrow["{pr_0}", from=1-2, to=2-2]
	\end{tikzcd}.
	\]The categorical fibre over $Lc\in\mathcal{C}$ induces a fully faithful embedding $\mathcal{C}\times_{L\mathcal{C}}\{Lc\}\hookrightarrow \mathcal{C}_{Lc/}$.
	
	The same trick also demonstrates that $\dafilbr$ can be embedded into $\dalg(\fil_{\ge0}\QC^\vee_R)_{/B}$.
\end{remark}
\newcommand{\sbar}{\overline{S}}
Next, we construct the monad of filtered partition Lie algebroids and compare it with the unfiltered one $T$ (\ref{df4.4}).
\begin{construction}\label{c4.20}
	Let $R\to B$ be a morphism in $\mathrm{SCR}^{\scoh}$ such that $\dL_{B/R}$ is almost perfect over $B$. The above discussion permits us to have a commuting diagram as follows 
	\[\begin{tikzcd}[column sep=huge]
		{\dalg(\QC^\vee_R)_{/B}} & {\dalg^{aug}(\QC^\vee_B)} & {\QC^\vee_B} \\
		\dafilbr & {\dalg^{nu}(\fil_{\ge1}\QC^\vee_B)} & {\fil_{\ge1}\QC^\vee_B}
		\arrow["{B\otimes_R(-)}", shift left, from=1-1, to=1-2]
		\arrow["{\infcohnewfunctor(B/-)}"', shift right, from=1-1, to=2-1]
		\arrow["{\mathrm{rest}}", shift left, from=1-2, to=1-1]
		\arrow["{\cot_{B}}", shift left, from=1-2, to=1-3]
		\arrow["\adic"', shift right, from=1-2, to=2-2]
		\arrow["{\sqz_{B}}", shift left, from=1-3, to=1-2]
		\arrow["{(-)_1}"', shift right, from=1-3, to=2-3]
		\arrow["{F^0}"', shift right, from=2-1, to=1-1]
		\arrow["{B\otimes_{\infcohR}(-)}", shift left, from=2-1, to=2-2]
		\arrow["{F^0}"', shift right, from=2-2, to=1-2]
		\arrow["{\mathrm{rest}}", shift left, from=2-2, to=2-1]
		\arrow["{\cot_B}", shift left, from=2-2, to=2-3]
		\arrow["{F^1}"', shift right, from=2-3, to=1-3]
		\arrow["{\sqz_{B}}", shift left, from=2-3, to=2-2]
	\end{tikzcd},\]where $\dalg^{nu}(\fil_{\ge1}\QC^\vee_B)$ is identified with a full subcategory of $\dalg^{aug}(\fil_{\ge0}\QC^\vee_B)$, $\adic$ is formally given by the adjoint functor theorem. Write the composite of the bottom adjunctions as $\cotrfil\dashv \sqzrfil$. Composing with $B$-linear dual, there is a monad $\sbar$ acting on $\fil_{\le -1}\QC^\vee_B$
	\[\sbar:=(-)^\vee\circ\cotrfil\circ \sqzrfil\circ(-)^\vee.\]At the same time, the counit map $\infcohnewfunctor(B/-)\circ F_0\to id_{\dafilbr}$ induces a morphism of monads \begin{equation}\label{e13}
		\sbar\to\const\circ \tbar\circ\colim=:\tbar^{\fil}.
	\end{equation}
\end{construction}
Next, we rectify $\sbar$ into a sifted-colimit preserving monad through a procedure parallel to the construction of $T$ in \S\ref{sec4.1}
\begin{lemma}\label{lemma4.21}
	Consider the initial object $\infcohR$ and the final object $B$ in $\dafilbr$, their relative cotangent fibres are given by $\cotrfil(\infcohR)\simeq 0$ and $\cotrfil(R)\simeq (\dL_{B/R})_1$.
\end{lemma}
\begin{proof}
	By definition, $\cotrfil(\infcohR)$ is initial and then null-homotopic. As for $B\in \dafilbr$, notice that $\infcohnewfunctor(B/B)\simeq B$ (Remark \ref{rk4.15}), which implies that $\cotrfil(B)\simeq \big(\cot_B(B\otimes_RB)\big)_1\simeq (\dL_{B/R})_1$.
\end{proof}
	
	Let $\sfil_{\le -1}\aperf^\vee_B$ be the full subcategory of dually almost perfect objects in $ \fil_{\le-1}\QC^\vee_B$. Using the above lemma and the same reasoning as Lemma \ref{lemma4.3}, for each $V\in \sfil_{\le -1}\aperf^\vee_B$, there is a natural morphism of \textit{splitting} fibre sequences in $\sfil_{\le -1}\aperf^\vee_B$

	\begin{equation}\label{e:two fibre sequeces}
		\begin{tikzcd}[column sep=2cm]
			{\bdlie(V)} & {\sbar(V)} & {\const(\dL_{B/R})} \\
			{\const\circ\bdlie\circ\colim(V)} & {\tbar(V)} & {\const(\dL_{B/R})}
			\arrow[from=1-1, to=1-2]
			\arrow[from=1-1, to=2-1]
			\arrow[from=1-2, to=1-3]
			\arrow[from=1-2, to=2-2]
			\arrow[Rightarrow, no head, from=1-3, to=2-3]
			\arrow[from=2-1, to=2-2]
			\arrow[from=2-2, to=2-3]
		\end{tikzcd},
	\end{equation}
	where $\bdlie$ means the derived PD $\infty$-operads of $B$-partition Lie algebras in filtered (the upper one) or ordinary (the lower one) pro-coherent modules respectively. Additionally, the middle vertical arrow is a map of monads
	
	The previous discussion has multiple consequences. Firstly, the monad $\sbar|_{\sfil_{\le -1}\aperf^\vee_B}$ is extended uniquely to a sifted-colimit-preserving monad $S$ acting on $\fil_{\le -1}\QC^\vee_B$ (Remark \ref{rk2.29}), and the latter has a presentable \infcat\ of algebra objects (Lemma \ref{lemma4.8})\[\lagdf:=\alg_{S}(\fil_{\le -1}\QC^\vee_B),\]whose objects are called \textit{(increasingly) filtered derived partition Lie algebroids}.\ There is also a commuting square of right adjoints analogous to (\ref{e: fibre and forgetful})
	\[\begin{tikzcd}
		\lagdf & {(\fil_{\le-1}\QC^\vee_B)_{/\const(\dT_{B/R}[1])}} \\
		{(\fil_{\le-1}\QC^\vee_B)_{\const(\dT_{B/R})/}} & {\fil_{\le-1}\QC^\vee_B}
		\arrow["\rmfor", from=1-1, to=1-2]
		\arrow["\rmfib"', from=1-1, to=2-1]
		\arrow[from=1-2, to=2-2]
		\arrow["{\mathrm{cofib},\simeq}"', from=2-1, to=1-2]
		\arrow[from=2-1, to=2-2]
	\end{tikzcd}.\]

	Secondly, $T^{\fil}:=\const\circ T\circ \colim$ is a monad \cite[Lemma 3.10]{BCN}, and the middle vertical arrow in (\ref{e:two fibre sequeces}) extends to a map of sifted-colimit-preserving monads $S\to T^{\fil}$. Furthermore, it gives rise to a sequence of right adjoints
	\[\alg_T(\QC^\vee_B)\xrightarrow{\const} \alg_{T^{\fil}}(\fil_{\le -1}\QC^\vee_B)\xrightarrow{\mathrm{forget}} \alg_S(\fil_{\le -1}\QC^\vee_B).\]
	\begin{lemma}\label{lemma4.2}
		The left adjoint $\phi:\lagdf\to\lagd$ of the above composite is given by $\colim:\fil^\vee_{\le-1}\QC^\vee_B\to\QC^\vee_B$ on the level of underlying modules, which is a localization functor.
	\end{lemma}
	\begin{proof}
		Note that the left adjoint of $\const$ is $\colim$ on the level of underlying modules by Proposition 4.1.9 in \cite{Raksit}. In the free case, $\phi(\fr_S(V))\simeq \colim(\fr_{T^{\fil}}(V))$ for each $V\in\fil_{\le -1}\QC^\vee_B$, while it follows from (\ref{e:two fibre sequeces}) that $\colim\circ S\simeq T\circ\colim$, so the first statement is true for free $S$-algebras. For general $L\in\lagdf$, it suffices to consider its bar resolution $\br_\bullet(id,\fr_S\circ\text{forget},L)$ and observe that every functor here respects sifted colimits. At the end, the full faithfulness of $\lagd\to\lagdf$ is implied by $\colim\circ \const\simeq id$.
	\end{proof}

	Thirdly, there is a Chevalley-Eilenberg functor in the filtered context compatible with that in Proposition \ref{prop4.7}:
	\begin{prop}\label{prop4.23}
		There is a commutative square of adjunctions of \infcats
		\[\begin{tikzcd}
			\lagdf & {(\dafilbr)^{op}} \\
			\lagd & {(\dalg(\QC^{\vee}_R)_{/B})^{op}}
			\arrow["{C^*_{\fil}}", shift left, from=1-1, to=1-2]
			\arrow["\colim"', shift right, from=1-1, to=2-1]
			\arrow["{\mathfrak{D}_{\fil}}", shift left, from=1-2, to=1-1]
			\arrow["{F_0}"', shift right, from=1-2, to=2-2]
			\arrow["\const"', shift right, hook', from=2-1, to=1-1]
			\arrow["{C^*}", shift left, from=2-1, to=2-2]
			\arrow["{\infcohnewfunctor(B/-)}"', shift right, from=2-2, to=1-2]
			\arrow["{\mathfrak{D}}", shift left, from=2-2, to=2-1]
		\end{tikzcd},\]
		where $\rmfib\circ\mathfrak{D}_{\fil}(A)\simeq \big(\const(\dT_{B/R})\to\cotrfil(A)^\vee \big)$.
	\end{prop}
	\begin{proof}
		Set $\aplagdf\subset\lagdf$ the full subcategory of filtered partition Lie algebroids with dually almost perfect underlying modules, and $\cwafpf\subset\dafilbr$ the full subcategory of such $A$ that $\cotrfil(A)$ is almost perfect. The top adjunction is constructed from
		\[\aplagdf\simeq\alg_{\sbar}(\sfil_{\le-1}\aperf^\vee_{B})\rightleftarrows(\cwafpf)^{op}\]in the same way as in Proposition \ref{prop4.7}, but everything is replaced by the filtered counterpart.
		
		Now it remains to check the commutativity. The difficulty is that there is no \textit{a priori} comparing transformation between the two paths. However, there is a commuting square of right adjoints
		\[\begin{tikzcd}
			{\alg_{\sbar}(\sfil_{\le-1}\aperf^\vee_B)} & {(\cwafpf)^{op}} \\
			{\alg_{{\tbar}^{\fil}}(\sfil_{\le-1}\aperf^\vee_B)} & {(\cwafp)^{op}}
			\arrow["{\fD_{\fil}}"', from=1-2, to=1-1]
			\arrow["\rmfor", from=2-1, to=1-1]
			\arrow["{\infcohnewfunctor(B/-)}"', from=2-2, to=1-2]
			\arrow["\psi", from=2-2, to=2-1]
		\end{tikzcd},\]since $\tbar^{\fil}$ and $\sbar$ are defined by some adjunctions in Construction \ref{c4.20}. Noting that $\cotrfil\circ\infcohnewfunctor(B/A)\simeq (\cotr(A))_1$ concentrates at degree $1$, the functor $\psi$ naturally factorizes into
		\[(\cwafp)^{op}\xrightarrow{\mathfrak{D}}\alg_{T}(\aperf^\vee_{B})\xhookrightarrow{\const}\alg_{T^{\fil}}(\sfil_{\le-1}\aperf^\vee_B).\]Then, there is dually $C^*\circ\colim|_{\aplagdf}\simeq F_0\circ C^*_{\fil}|_{\aplagdf}$. By Lemma \ref{lemma4.10}, the square of left adjoints in the statement commutes.
	\end{proof}

		\begin{df}\label{df4.24}
		The composite $\tc:\lagd\hookrightarrow\lagdf\xrightarrow{C^*_{\fil}}(\dafilbr)^{op}$ is called \textit{the (completely) Hodge-filtered Chevalley-Eilenberg functor of partition Lie algebroids}.
	\end{df}
	Then, we go to the main theorem of this section.
	\begin{theorem}\label{thm4.25}
		Let $R\to B$ be a map of coherent simplicial commutative rings such that $\dL_{B/R}$ is an almost perfect $B$-module. The Hodge-filtered Chevalley-Eilenberg functor $\tc$ induces a categorical equivalence 
		\[\tc:\aplagd\xrightarrow{\simeq} (\dalg^{\fol}_{B/R,\ap})^{op},\]
		sending \textnormal{dually almost perfect} partition Lie algebroids to \textnormal{almost perfect} foliation-like algebras (Notation \ref{n4.14} (2)). Moreover, there is $\gr(\tc(L))\simeq \lsym_B[\rmfor(L)^\vee]_1$ for each $L\in\aplagd$. 
	\end{theorem}
	\begin{proof}
		We begin with the fact that $C^*_{\fil}$ takes value in complete algebras.
		\begin{lemma}
			For each $L\in \lagdf$, $C^*_{\fil}(L)$ is complete as a filtered algebra.
		\end{lemma}
		\begin{proof}
			For $L\simeq \fr_{S}(V)$ with $V\in\sfil_{\le-1}\aperf^\vee_{B}$, there is $C^*_{\fil}(L)
			\simeq \sqzrfil(V^\vee)$, which is surely complete. The general case follows from the fact that $\lagdf$ is generated by such $L$ with sifted colimits, and $C^*_{\fil}$ takes colimits to limits.
		\end{proof}
		Now we calculate $\tc$ and check the full faithfulness in two steps.\\
		\noindent\textbf{Pointed case:} $L:=\rho_0(\fh)$\footnote{$\rho_0$ is introduced in Corollary \ref{cor4.11}} for some $\fh\in\alg_{\bdlie}(\aperf^\vee_B)$.
		
		Observe that $\tc(L)\simeq \mathrm{rest}\circ\cet(\fh)$ is foliation-like and $\gr\cet(\fh)\simeq \lsym_B[\fh^\vee]_1$ (Corollary \ref{cor3.31n}).\ Next we verify that $\eta_L:\const(L)\to \fD_{\fil}\circ \tc(L)$ is an equivalence.\ Consider the following commutative diagram:
		
	\[\begin{tikzcd}
		{\const(\fh)} & {\const\circ\rmfib(L)} & {\const\circ\rmfib(\rho_0(0))} \\
		{\kd\circ\cet(\fh)} & {\rmfib\circ\fD_{\fil}\circ\tc(L)} & {\rmfib\circ\fD^{\fil}(B)}
		\arrow[from=1-1, to=1-2]
		\arrow["\simeq"', from=1-1, to=2-1]
		\arrow[from=1-2, to=1-3]
		\arrow["{\rmfib(\eta_L)}", from=1-2, to=2-2]
		\arrow["{\rmfib(\eta_{\rho_0(0)})}", from=1-3, to=2-3]
		\arrow[from=2-1, to=2-2]
		\arrow[from=2-2, to=2-3]
	\end{tikzcd}.\]
		Since $\rho_0(0)\simeq \fr_{T}(0)$, we know that $\rmfib(\rho_0(0))\simeq \dL_{B/R}^\vee$ and the upper line is a fibre sequence. Besides, the right vertical arrow is an equivalence because of $\const(\fr_{T}(0))\simeq\fr_{T^{\fil}}(0)\simeq \fD^{\fil}(B)$. The left vertical equivalence is given by Corollary \ref{cor3.31n}. In the meanwhile, $\rmfib\circ\fD_{\fil}\circ\tc(L)$ is equivalent with $\kd(B\otimes_{\infcohR}B\otimes_B \cet(\fh))$ by Proposition \ref{prop4.23}, which implies that the lower line is also a fibre sequence. Hence $\rmfib(\eta_{L})$ (and $\eta_L$ itself) is an equivalence.
		
		\noindent\textbf{General case:} $L\in\aplagd$.
		
		Recall that Corollary \ref{cor4.11} provides a monadic adjunction between $\bdlie$ and $\lagd$. Here $L$ can be written as the geometric realization of 
		$H_\bullet:=\br_\bullet(\rho_0\circ\mathrm{fib},\rho_0\circ\mathrm{fib},L)$. Since $H_\bullet$ is termwise pointed, and $\rmfor(H_n)\simeq \rmfib(L)\oplus (\dL^\vee_{B/R})^{\oplus n}$, the graded pieces of $\tc(L)$ can be expressed as a totalization of freely generated graded algebras 
		$\gr(\tc(H_\bullet))\simeq \lsym_B \big[\rmfib(L)^\vee\oplus(\dL_{B/R})^{\bullet}\big]_1$ following the pointed case. Observing that the module of generators concentrates at weight $1$, $\gr(\tc(H_{\bullet}))$ is determined by its weight $1$ component. Thus this cosimplicial diagram could be obtained by applying $\lsym_B\circ[-]_1$ to the \v{C}ech conerve of $\rmfor(L)^\vee\to \rmfib(L)^\vee$. Therefore, we have the expected equivalence $\gr(\tc(L))\simeq \lsym_B[\rmfor(L)^\vee]_1$ using Corollary \ref{totalg}.
		
		Next, we show that $\eta:\const(L)\to \fD_{\fil}\circ \tc(L)$ is an equivalence. Here $\eta$ factorizes into
		\[\const(L)\simeq |\const(H_\bullet)|\xrightarrow{\alpha} |\fD_{\fil}\circ \tc(H_\bullet)|\xrightarrow{\beta} \fD_{\fil}\circ \tc(|H|_\bullet)\simeq \fD_{\fil}\circ \tc(L).\] The map $\alpha$ is an equivalence from the pointed case. The map $\beta$ is the $B$-linear dual of $\cotrfil(\tc(L))\to \tot \cotrfil(\tc(H_\bullet))$, whose associated graded map is an equivalence by Corollary \ref{totalg}. Since $B$-linear dual is invariant under completion, $\beta$ itself is also an equivalence.
		
		\noindent\textbf{Essential image of $\tc|_{\aplagd}$:}
		
		Take an arbitrary $A\in \dalg^{\fol}_{B/R,\ap}\subset\dafilbr$. It is clear that $\rmfib(\fD_{\fil}(A))\simeq \big(\text{cofib}(\dL_{R/k}[-1]\to \gr_1A)\big)^\vee$ is dually almost perfect, i.e.\ $L_A:=\colim\circ\fD_{\fil}(A)$ belongs to $\aplagd$. It is remained to show that $\delta_A: A\to \tc(L_A)$ is an equivalence. Since both ends are complete and quasi-free with almost perfect $\gr_1$ components, it suffices to consider $\fD_{\fil}(\delta_A)$. However, we already have the equivalence $\const(L_A)\simeq \fD_{\fil}\circ \tc(L_A)$ as above.

	\end{proof}
		\begin{remark}
		Let $X/R$ be a locally coherent qcqs derived scheme. This theorem can be improved into an equivalence between the \textit{sheaves of almost perfect foliation-like algebras} and dually almost perfect partition Lie algebroids on $X$. In the next section, we globalize this theorem with a more geometric method.
	\end{remark}

	We end this section with a discussion on the derived infinitesimal cohomology. Here, Notation \ref{n4.14} slightly differs from that in \cite{Antieau}, since Antieau works with ordinary modules rather than pro-coherent modules.\ But, if \textit{$R$ is eventually coconnective}, the two notions of Hodge-completed derived infinitesimal cohomology would coincide following Proposition \ref{prop2.38}, and better, all almost perfect foliation-like algebras live in $\dalg(\fil_{\ge0}\m_R)$. Thus, let us keep this assumption here.
	
	Now, we go to a proposition that justifies the name of $\tc$:
	\begin{prop}\label{prop4.20old}
		Consider a coherent $S\in \mathrm{SCR}_{R//B}$ such that $\dL_{B/S}$ is an almost perfect $B$-module.\
		The foliation-like algebra $\tc\circ\fD(S)\in\dalg^{\fol}_{B/R,\ap}$ admits a natural $S$-algebra structure, by which it can be identified with $(\infcohnewfunctor(B/S))^\wedge\in \dalg^{\fol}_{B/S,\ap}$, the Hodge-completed derived infinitesimal cohomology of $B$ relative to $S$.
		
	\end{prop}
	\begin{proof}
		By the adjunction $(-)_0:\dalg(\QC^\vee_R)_{/B}\rightleftarrows\dalg(\fil_{\ge0}\QC^\vee_R)_{/B}:F^0$, the unit map $S\to C^*\circ\fD(S)$ gives rise to a natural map $\alpha_S:S\to \tc\circ \fD(S)$. According to Remark \ref{rk4.19}, the map $\alpha_S$ exhibits $\tc\circ \fD(S)$ as an object in $\mathcal{D}^{\fil}_{B/S}$. Since $\infcohnewfunctor(B/S)$ is initial in $\mathcal{D}^{\fil}_{B/S}$, there is a canonical comparison map $\beta_S:(\infcohnewfunctor(B/S))^\wedge\to \tc\circ\fD(S)$.
		
		Now, let's verify that $\beta_S$ is an equivalence. For each $A\in \dalg^{\fol}_{B/R,\ap}$, set $L_A:=\colim\circ \fD(A)\in\aplagd$. Theorem \ref{thm4.25} gives a natural equivalence $A\simeq\tc(L_A)$, which implies that $F^0(A)\simeq C^*(L_A)$. Furthermore, we have the functorial equivalences of mapping spaces for $A$
		\[\map_{\dalg_{R//B}}(S,F^0A)\simeq \map_{\lagd}(L_A,\fD(S))\simeq \map_{\dalg^{\fol}_{B/R,\ap}}(\tc\circ\fD(S),A).\]
		The morphisms $S\to F^0 (\infcohnewfunctor(B/S))^\wedge\xrightarrow{F^0(\beta_S)} C^*\circ\fD(S)$ correspond to 
		\[\tc\circ\fD(S)\xrightarrow{\gamma_S}(\infcohnewfunctor(B/S))^\wedge\xrightarrow{\beta_S} \tc\circ\fD(S),\] where $\beta_S\circ \gamma_S$ is an identity. At the same time, there is $\gamma_S\circ \beta_S\simeq id_{(\infcohnewfunctor(B/S))^\wedge}$ by the universal property of $\infcohnewfunctor(B/S)$.
	\end{proof}
	When $S$ is discrete and $B$ is a smooth $S$-algebra, $(\infcohnewfunctor(B/S))^\wedge$ is equivalent with Grothendieck's infinitesimal cohomology as proved in \cite[Theorem 3.2]{Toen2023} (using Theorem \ref{thmn5.11} as a dictionary). More generally, in the view of \cite[\S3]{Toen}, $\tc(L)$ induces a Hodge filtration on the (complete) foliated infinitesimal cohomology by $L$. It might also be worth mentioning that, when $R$ is over $\mathbb{Q}$, derived infinitesimal cohomology agrees with derived de Rham cohomology.
	\section{Towards infinitesimal derived foliations}\label{sec5}
	We have established an equivalence between partition Lie algebroids and foliation-like algebras under some finiteness conditions (Theorem \ref{thm4.25}).\ In \S\ref{sec5.1}, we spell out their geometrization, namely \textit{infinitesimal derived foliations}. Finally, \S\ref{sec5.2} is devoted to the proof of the main theorem (\ref{thmn5.13}).\ Familiarity with the geometry of filtrations is assumed, referring to \cite{Mou} for details.
	\subsection{Infinitesimal derived foliations}\label{sec5.1}
	To motivate the construction of infinitesimal derived foliations, we transform the foliation-like algebras into terms of homotopy-coherent cochain complexes.
	
	Recall that, for each $\mathcal{E}\in\calg(\pr^{\st})$, there is a symmetric monoidal equivalence $\fil^{\cpl}\mathcal{E}\simeq \dg_-\mathcal{E}$ (Theorem \ref{thm dg_-}), which identifies the completely filtered $\eoo$-algebras and homotopy-coherent $\eoo$-dg-algebras. The next proposition elaborates on a derived analogue.

	\begin{prop}\label{prop5.1n}
		Let $R$ be a (coherent) simplicial commutative ring, $\dD^\vee_-$ be the derived Hopf algebra from Remark \ref{rk2.8: k[t]} and $\mathcal{E}$ be $\m_R$ (or $\QC^\vee_R$). There is a commuting square of symmetric monoidal functors
		\[\begin{tikzcd}
			{\dalg(\fil^{\cpl}\mathcal{E})} & {\mathrm{LComod}_{\dD^\vee_-}(\dalg(\gr\mathcal{E}))} \\
			{\fil^{\cpl}\mathcal{E}} & {\dg_{-}\mathcal{E}}
			\arrow["{f,\simeq}", from=1-1, to=1-2]
			\arrow["\simeq", from=2-1, to=2-2]
			\arrow["{\mathrm{forget}}"', from=1-1, to=2-1]
			\arrow["{\mathrm{forget}}", from=1-2, to=2-2]
		\end{tikzcd}.\]

	\end{prop}
	 \begin{proof}
	 	Recall that $\fil\mathcal{E}\simeq \mathrm{LMod}_{1_\mathcal{E}[t]}(\gr\mathcal{E})$ (Remark \ref{rk2.8: k[t]}), where $1_\mathcal{E}[t]\simeq \fr_{\com}([R]_{-1})$. We construct $f$ as the following horizontal composite
	 	\[\begin{tikzcd}
	 		{f:\dalg(\fil^{\cpl}\mathcal{E})} & {\mathrm{LMod}_{1_{\mathcal{E}}[t]}(\dalg(\gr\mathcal{E}))} & {\mathrm{LComod}_{\dD^\vee_-}(\dalg(\gr\mathcal{E}))} \\
	 		&& {\dalg(\gr\mathcal{E})}
	 		\arrow["{\mathrm{forget}}", from=1-1, to=1-2]
	 		\arrow["{1\otimes_{1_{\mathcal{E}}[t]}-}", from=1-2, to=1-3]
	 		\arrow[from=1-3, to=2-3]
	 		\arrow["\gr"', from=1-1, to=2-3]
	 	\end{tikzcd},\]where one can see that $f$ preserves small colimits, since $\gr$ and the forgetful functor of comodules are conservative left adjoints. Since every \infcat\ here is presentable, $f$ is a left adjoint and the above diagram gives rise to a morphism of comonads on $\dalg(\gr\mathcal{E})$, $\alpha:\gr\circ \triv\to \dD^\vee_-\otimes-$. Forgetting the derived algebra structure commutes with both $\gr$ and $\triv$ by Theorem \ref{thm3.2n}. Besides, the forgetful functor, as a symmetric monoidal functor, also commutes with $\dD^\vee_-\otimes-$ . Therefore, by considering the underlying modules, $\alpha$ is an equivalence \cite[Theorem 3.2.14]{Raksit}. Finally, using Barr-Beck-Lurie theorem, $\gr:\dalg(\fil^{\cpl}\mathcal{E})\to \dalg(\gr\mathcal{E})$ is comonadic, which implies that $f$ is an equivalence.
	 \end{proof}
	 \begin{remark}\label{rk5.4}
	 	(1) The \infcat\ $\dg_-\m_R$ has a twin $\dg_+\m_R$\cite[Definition 5.1.4]{Raksit}.\ Although these twin \infcats\ are equivalent in $\pr^{\st}$ by the shearing functor $[+2*]$, their derived algebra objects differ subtly away from characteristic $0$.\ For an in-depth study of $\dg_+\m_R$ and its algebras, see \cite{MRT} (in the term of graded mixed complexes) and \cite{Raksit}.  
	 	
	 	(2) In characteristic $0$, the derived algebras are equivalent to $\eoo$-algebras (both modeled as cdgas). Let $R$ be a connective cdga over $\mathbb{Q}$. The shearing functor $[+2*]$ induces an auto-equivalence of $\dalg_R(\gr\m_R)$, where $\dD^\vee_-\simeq [R]_0\oplus [R[1]]_{-1}$ is sent to a graded Hopf cdga $\dD^\vee_+\simeq [R]_0\oplus [R[-1]]_{-1}$. Then, there is an equivalence
	 	\begin{equation}\label{shearing}
	 		\dalg(\fil^{\cpl}\m_R)\simeq\mathrm{LComod}_{\dD^\vee_-}(\dalg(\gr\m_R))\xrightarrow{[+2*],\simeq}\mathrm{LComod}_{\dD^\vee_+}(\dalg(\gr\m_R)),
	 	\end{equation}
	 	where the right-hand side is denoted as $\epsilon-\mathrm{cdga}^{\sgr}_R$ and modeled by graded mixed cdgas in \cite{TV}.
	 	
	 	Let $B$ be a connective cdga finitely presented over $\mathbb{C}$. Following Proposition \ref{prop2.38} and \ref{prop5.1n}, (\ref{shearing}) identifies $\dalg^{\fol}_{B/\mathbb{C}}$ with the \infcat\ of graded mixed algebras $(\mathcal{A}^\bullet,d)$ over $\mathbb{C}$ with a chosen $\mathcal{A}^0\simeq B$ such that the natural map $\sym_B \mathcal{A}^1\xrightarrow{\simeq} \mathcal{A}^\bullet$ of graded cdgas is an equivalence. Therefore, the derived foliations over $B$ in the sense of \cite[Definition 1.2.1]{TV} form a full subcategory of $\dalg^{\fol}_{B/\mathbb{C}}$ via (\ref{shearing}), which is spanned by such $(\mathcal{A}^\bullet,d)$ that $\mathcal{A}^1[-1]$ is a connective perfect $B$-module.

	 \end{remark}

	\begin{remark}
	 	In the $R$-linear context, $\dD^\vee$ is free as a derived algebra. Indeed, the map of modules $[R[1]]_{-1}\to [R]_0\oplus[R[1]]_{-1}\simeq \dD^\vee$ induces a morphism of derived algebras $\fr_{\com}([R[1]]_{-1})\to \dD^\vee_-$, which turns out to be an equivalence. Further, $\spec(\dD^\vee)$ is precisely the loop group $G_0:=0\times_{\mathbb{G}_a}0$ of additive group with the natural $\mathbb{G}_m$-action.
	 \end{remark}
	 
	 Now we geometrize the comultiplication of $\dD^\vee_-$. Applying \cite[\S4]{Mou}, we have a symmetric monoidal equivalence \[\dg_-\m_R\simeq \mathrm{LComod}_{\mathcal{O}(G_0)}(\QC(B\mathbb{G}_m)).\]Set $\cH:=\mathbb{G}_m\ltimes G_0$ the semidirect product derived group scheme. Its classifying stack $B\cH$ can be regarded as a $B\mathbb{G}_m$-pointed derived stack.
	 \begin{lemma}\label{lemma5.1}
	 	The $B\mathbb{G}_m$-pointed stack $B\cH$ serves as the relative classifying stack of $[G_0/\mathbb{G}_m]\to B\mathbb{G}_m$, where $[G_0/\mathbb{G}_m]$ is regarded as an abelian graded group stack.
	 \end{lemma}

	 \begin{proof}
	 	The classifying stack of $[G_0/\mathbb{G}_m]\to B\mathbb{G}_m$ is calculated by the geometric realization of the diagonal $\delta^\bullet$ of the bisimplicial stack $G_0^{\times \bullet}\times \mathbb{G}_m^{\times \star}$. The transition $\cH^{\times n}\to G_0^{\times n}\times \mathbb{G}^{\times n}_m$ sending $(x_1,g_1,\ldots,x_n,g_n)$ to $(x_1,g_1^{-1}x_2,\ldots,(g_1\ldots g_{n-1})^{-1}x_n,g_1,\ldots,g_n)$ identifies $\br_\bullet(pt,\cH,pt)$ with $\delta^\bullet$. Taking geometric realization finishes this proof.
	 \end{proof}
	 
	 \begin{prop}\label{prop5.2}\footnote{There is a similar result for $\dg_{+}\m_R$ in \cite[\S4.2]{MRT}, which is harder since $B\mathrm{Ker}$ (unlike $G_0$) is not affine.}
	 	There is an equivalence in $\calg(\pr^{st})$ identifying the quasi-coherent sheaves on $B\cH$ with the graded representations of $[G_0/\mathbb{G}_m]$, i.e.
	 	$\QC(B\cH)\simeq\mathrm{LComod}_{\mathcal{O}(G_0)}(\QC(B\mathbb{G}_m))$.
	 \end{prop}
	 \begin{proof}
	 	The morphism $p:B\mathbb{G}_m\to B\cH$ is the geometric realization of $p^\bullet:G_0\times \cH^{\times \bullet}\to \cH^{\times \bullet}$, where the pullback of quasi-coherent modules induces a diagram of left adjoints $p^{\bullet,*}:\QC(\cH^{\times \bullet})\to \QC(G_0\times\cH^{\times \bullet})$. Due to $\mathcal{O}(G_0)\simeq R\oplus R[1]$, the diagram $p^{\bullet,*}$ is also right adjointable\cite[Definition 7.3.1.1]{HTT} and the right adjoint diagram $p^{\bullet}_*$ consisits of small-colimit-preserving functors. Taking totalization, there is a comonadic adjunction of presentable \infcats
	 	\[p^*:\QC(B\cH)\rightleftarrows\QC(B\mathbb{G}_m):p_*\]\cite[Proposition 3.3.3.1, 5.5.3.13]{HTT}. Note that $p^*$ factorizes into the following left adjoints
	 	\[\QC(B\cH)\xrightarrow{\simeq}\tot\QC([G_0^{\times\bullet}/\mathbb{G}_m])\to \mathrm{LComod}_{\mathcal{O}(G_0)}(\QC(B\mathbb{G}_m))\to \QC(B\mathbb{G}_m),\]where the left equivalence is given by Lemma \ref{lemma5.1} and pullbacks, the middle arrow is taking the underlying cosimplicial diagrams, and the right arrow is the forgetful functor of $\mathcal{O}(G_0)$-comodules. Thus we have a natural morphism $\alpha:p^*p_*\to \mathcal{O}(G_0)\otimes^{\sgr}(-)$ of comonads. The cartesian diagram
	 	\[\begin{tikzcd}
	 		{G_0} & {[G_0/\mathbb{G}_m]} & {B\mathbb{G}_m} \\
	 		pt & {B\mathbb{G}_m} & B\cH
	 		\arrow["{p^0}"', from=1-1, to=2-1]
	 		\arrow["p", from=1-3, to=2-3]
	 		\arrow["q", from=1-2, to=2-2]
	 		\arrow["p", from=2-2, to=2-3]
	 		\arrow["q", from=1-2, to=1-3]
	 		\arrow["\pi", from=2-1, to=2-2]
	 		\arrow["{\pi'}", from=1-1, to=1-2]
	 	\end{tikzcd}\]induced by Lemma \ref{lemma5.1} helps to rephrase the transformation $\alpha$ as $p^*p_*\to q_*q^*$. Remind us that $p_*$ is the totalization of the diagram $p^{\bullet}_*$, thus $\pi^*p^*p_*\simeq p^{0}_*\pi'^{*}q^*$. However, $\pi$ and $\pi'$ are smooth covering, which means that $\alpha$ is an equivalence using \cite[A.1.5, A.1.10]{HLP}.
	 \end{proof}
	 
	 We are then tempted to geometrize foliation-like algebras as derived $\cH$-equivariant stacks. 
	 \begin{df}
	 	The assignment $S\mapsto \map_{\dalg(\m_R)}(S,-)$ induces a functor $\dalg(\m_R)^{op}\to\dst\subset \mathrm{Fun}(\mathrm{SCR}_R,\mathcal{S})$ taking values in the \infcat\ of derived $R$-stacks \cite[Lemma 2.2.2.13]{HAG2}. We call it the \textit{functor of non-connective spectra} $\spnc$.
	 \end{df}

	 For $B\in \mathrm{SCR}_R$, there is a functor of \textit{linear stacks} $\mathbb{V}:\m^{op}_B\to \mathbf{dSt}_{B}$ defined by
	 \[E\mapsto \spnc(\lsym_B E)(-)\simeq \map_{B}(E,-).\]The $B$-derived stack $\mathbb{V}(E)$ can be equipped with a natural graded structure by posing $E$ at weight $1$, which enhances $\mathbb{V}$ into
	 \[\mathbb{V}:\m^{op}_B\to\mathbb{G}_m-\mathbf{dSt}_B.\]Moreover, the restriction $\mathbb{V}|_{\m^{op}_{B,+}}$ of the enhanced $\mathbb{V}$ is fully faithful\cite[Theorem 2.5]{Monier}, where $\m_{B,+}$ is the full subcategory of eventually connective $B$-modules.
	 
	 On a derived scheme $X$, there is a global version of $\mathbb{V}$ sending a quasi-coherent sheaf $E$ to the relative non-connective affine stack corepresented by the sheaf of algebras $\lsym_{\mathcal{O}_X}E$. Moreover, since essential connectivity is preserved by Zariski descent, there is a fully faithful embedding
	 \[\QC_{+}(X)\hookrightarrow \mathbb{G}_m-\mathbf{dSt}_X,\]sending \textit{essentially connective} coherent sheaves to their ``stacky vector bundles''. The essential image is denoted as $\mathbf{dSt}^{lin}_X$, which means \textit{linear stacks} over $X$.

	 Now we consider the extra differential. Recall that the forgetful functor of equivariant stacks along $\mathbb{G}_m\to \cH$ admits a right adjoint $\mathcal{L}^{\sgr}_{\pi}(-/R)$ such that the underlying derived stack of $\mathcal{L}^{\sgr}_{\pi}(X/R)$ is \[\map_{\mathbb{G}_m-\dst}(\cH,X)\simeq \map_{\mathbb{G}_m-\dst}(G_0\times\mathbb{G}_m,X)\simeq \map_{\dst}(G_0,X).\]

	 \begin{df}[Definition 2.1 \cite{Toen2023}]\label{df5.5}
	 	On some derived scheme $X$ over $R$, an \textit{infinitesimal derived foliation (relative to $R$)} is an $\cH$-equivariant stack $\mathscr{F}$ together with a $\mathbb{G}_m$-equivariant map $p:\mathscr{F}\to X$, where $\mathbb{G}_m$ acts trivially on $X$, and $p$ is equivalent to $\mathbb{V}(E)\to X$ the projection of some linear stack over $X$. More categorically, the \textit{\infcat\ of infinitesimal derived foliation on $X$ (relative to $R$)} is given by a pullback of \infcats
	 	\[\begin{tikzcd}
	 		{\fol^{\pi}(X/R)} & {\mathbf{dSt}^{lin}_X} \\
	 		{(\cH-\dst)_{/\mathcal{L}^{\sgr}_{\pi}(X/R)}} & {(\mathbb{G}_m-\dst)_{/X}}
	 		\arrow[from=1-1, to=1-2]
	 		\arrow[hook, from=1-1, to=2-1]
	 		\arrow[hook, from=1-2, to=2-2]
	 		\arrow["{\text{forget}}"', from=2-1, to=2-2]
	 	\end{tikzcd}.\]
	 	Here $E[1]$ is called \textit{cotangent complex of the infinitesimal derived foliation} $\mathscr{F}$ and denoted as $\dL_{\mathscr{F}}$.
	 \end{df}
	 
	 \begin{remark}[Restriction on $\dL_{\mathscr{F}}$]
	 	We do not require perfectness like in \cite[\S2]{Toen2023}, because almost perfectness is preferable for PD Koszul duality. Other than this, we want to include $\mathcal{L}^{\sgr}_{\pi}(X/R)$ so that $\fol^{\pi}(X/R)$ always has a final object (see Corollary \ref{corn5.10}). However, the eventually connectiveness is \textbf{neccessary} for defining cotangent complexes. For a counterexample, set $B:=\lsym_R(R[2])$ and $E$ as the $B$-module
	 	\[\colim\big( B\xrightarrow{\times t}\Sigma^{-2}B\xrightarrow{\times t}\Sigma^{-4}B\xrightarrow{\times t}\ldots\big).\] 
	 	Note that each $B$-derived algebra $(A,a)$ consists of a $R$-derived algebra $A$ and a point of $\pi_2(A)$. The $B$-linear stack $\mathbb{V}(E)$ is defined by the assignment $\mathbb{V}(E)\big((A,a)\big)\simeq \map_{\m_R}\big(E,(A,a)\big)$. This mapping space is the limit of an inverse system
	 	\[\Omega^{\infty+2}A\xleftarrow{\times a}\Omega^{\infty}A\xleftarrow{\times a}\Omega^{\infty-2}A\xleftarrow{\times a}\ldots,\]which is obviously contractible. However, the $B$-module $E$ itself is non-trivial.
	 \end{remark}
	 \begin{df}\label{df5.7n}
	 	Let $X$ be a derived scheme over $R$.\ An infinitesimal derived foliation $\mathscr{F}\in \fol^\pi(X/R)$ is said to be \textit{almost perfect} if its cotangent complex $\dL_{\mathscr{F}}$ is an almost perfect quasi-coherent sheaf \cite[Definition 2.8.4.4]{SAG}.\ The full subcategory of almost perfect infinitesimal derived foliations is written as $\fol^{\pi}_{\ap}(X/k)$.

	 \end{df}
	\subsection{Realization of foliations}\label{sec5.2}
	So far, we have translated foliation-like algebras into the language of homotopy-coherent chain complexes (Proposition \ref{prop5.1n}). Additionally, we have observed that \textit{infinitesimal derived foliations} (Definition \ref{df5.5}) could potentially encode the foliation-like algebras in a geometric way.
	
	In this subsection, we will confirm this expectation, at least when the cotangent complex is almost perfect. It then leads to the main theorem, Theorem \ref{thmn5.13}, that compares partition Lie algebroids and infinitesimal derived foliations on \textit{locally coherent qcqs} derived schemes. We write $\dalg(\m_R)$ as $\dalg_R$ and $\mathrm{LComod}_{\dD^\vee_-}(\dalg(\gr\m_R))$ as $\dg_{-}\dalg_R$ for short.
	
	We begin with constructing a functor that realizes foliation-like algebras into infinitesimal derived foliations. Note that the functor $\spnc:\dalg_R^{op}\to \dst$ is a right adjoint.\ In particular, it is symmetric monoidal with respect to the cartesian monoidal structures. In the light of \cite[\S3]{Mou}, there is a symmetric monoidal equivalence $\dalg(\gr\m_R)\simeq \mathrm{LComod}_{R[t,t^{-1}]}(\dalg_R)$ and then a symmetric monoidal functor $\dalg(\gr\m_R)\to (\mathbb{G}_m-\dst)^{op}$ induced by $\spnc$.
	
	Similarly, noting that $\dD^\vee$ is a cocommutative graded Hopf derived algebra and using Proposition \ref{prop5.1n}, \ref{prop5.2}, $\spnc$ induces a symmetric monoidal functor
	\[\dg_-\dalg_R\to  \big(\Omega_0\mathbb{G}_a-(\mathbb{G}_m-\dst)\big)^{op}\simeq (\cH-\dst)^{op},\]where the latter equivalence comes from $\cH=\mathbb{G}_m\ltimes \Omega_0\mathbb{G}_a$.
	
	As to $\mathcal{O}:\dst\to (\dalg_R)^{op}$, the left adjoint of $\spnc$, it is oplax symmetric monoidal and sends $X$ to a limit $\mathcal{O}(X):=\lim\limits_{\text{Spec}(A)\to X}A$ in $\dalg_R$, where $A$ ranges among connective algebras. The oplax symmetric monoidal structure refines $\mathcal{O}$ into an $\cH$-equivariant version:
	 	\begin{prop}\label{propn5.8}
	 	From the above discussion, there is a commuting diagram of adjunctions
	 	
	 	\[\begin{tikzcd}[column sep=1.5cm]
	 		{\dg_-\dalg_R} & {\dalg(\gr\m_R)} & {\dalg_R} \\
	 		{(\cH-\dst)^{op}} & {(\mathbb{G}_m-\dst)^{op}} & {(\dst)^{op}}
	 		\arrow["{\mathrm{forget}}", shift left, from=1-1, to=1-2]
	 		\arrow["\spnc"', shift right, from=1-1, to=2-1]
	 		\arrow["{\dD^\vee_-\otimes^{\sgr}_{R}(-)}", shift left, from=1-2, to=1-1]
	 		\arrow["{\mathrm{forget}}", shift left, from=1-2, to=1-3]
	 		\arrow["\spnc"', shift right, from=1-2, to=2-2]
	 		\arrow["{R[t,t^{-1}]\otimes_R(-)}", shift left, from=1-3, to=1-2]
	 		\arrow["\spnc"', shift right, from=1-3, to=2-3]
	 		\arrow["{\mathcal{O}^{\epsilon-\sgr}}"', shift right, from=2-1, to=1-1]
	 		\arrow["{\mathrm{forget}}", shift left, from=2-1, to=2-2]
	 		\arrow["{\mathcal{O}^{\sgr}}"', shift right, from=2-2, to=1-2]
	 		\arrow["{[\cH\times -/\mathbb{G}_m]}", shift left, from=2-2, to=2-1]
	 		\arrow["{\mathrm{forget}}", shift left, from=2-2, to=2-3]
	 		\arrow["{\mathcal{O}}"', shift right, from=2-3, to=1-3]
	 		\arrow["{\mathbb{G}_m\times-}", shift left, from=2-3, to=2-2]
	 	\end{tikzcd},\]where one has a natural equivalence $\mathcal{O}^{\sgr}\circ\text{forget}\simeq \text{forget}\circ \mathcal{O}^{\epsilon-\sgr}$. Addtionally, given an affine $\mathbb{G}_m$-derived scheme $\text{Spec}(A)$, there is an equivalence $\mathcal{O}^{\sgr}(\text{Spec}(A))\simeq A$ of connective graded algebras.
	 \end{prop}
	 \begin{proof}
	 	It is clear by construction that the left adjoints commute. Then, by abstract non-sense, the adjoint map $[\cH\times-/\mathbb{G}_m]\to id_{\cH-\dst}$ induces a natural transformation
	 	\[\tau:\text{forget}\circ \mathcal{O}^{\epsilon-\sgr}\to\mathcal{O}^{\sgr}\circ\text{forget}.\]To show that it is an equivalence, we need the fact that $\mathcal{O}^{\sgr}(\spnc(A))\simeq A$ for a graded simplicial commutative ring $A$. The assertion holds obviously for free connective affine $\mathbb{G}_m$-schemes $\mathbb{G}_m\times \text{Spec}(B)\simeq \text{Spec}(B[t,t^{-1}])$. The general case is done by doing flat descent on $\br_\bullet(\mathbb{G}_m,\mathbb{G}_m,\text{Spec}(A))\to \text{Spec}(A)$.
	 	
	 	For each $\cH$-stack of the form $[\cH\times \text{Spec}(A)/\mathbb{G}_m]$, where $A$ is a graded simplicial commutative ring, the unit map of $\rmfor\dashv [\cH\times-/\mathbb{G}_m]$ is given by the multiplication of $G_0$ as graded group stack. Thus, after applying $\mathcal{O}^{\epsilon-\sgr}$, it becomes the comultiplication $\dD^\vee_-\otimes^{\sgr}_RA\to \dD^\vee_-\otimes^{\sgr}_R\dD^\vee_-\otimes^{\sgr}_RA$ which is adjoint to the $id_{A[\epsilon]}$.
	 	
	 	For general $\cH$-stack $X$, note that it can be written as a small colimit $\colim_i X_i$ such that each $X_i$ has the form of $[\cH\times \text{Spec}(A_i)/\mathbb{G}_m]$. Then observe that, the upper left forgetful functor is straightforwardly small-limit-preserving, and the left lower forgetful functor has a left adjoint, which is the opposite of $\mathcal{L}^{\sgr}_{\pi}(-/R)$. Therefore, $\tau$ is always an equivalence since both sides takes small colimits to small limits.
	 \end{proof}
	 
	 \begin{remark}\label{rk5.9}
	 	The functors $\mathcal{O}^{\sgr}$ and $\mathcal{O}$ do not commute with the forgetful functors, because $R[t,t^{-1}]$ is not of finite rank over $R$. For a counterexample, consider the free $\mathbb{G}_m$-stack $X:=\mathbb{G}_m\times \widehat{\mathbb{A}}^{1}_{R,0}$ generated by the formal completion of the affine line at its origin. Its graded global section (forgetting the grading) is $\mathcal{O}^{\sgr}(X)\simeq (R[[s]])[t,t^{-1}] $, but the ordinary global section is $\mathcal{O}(X)\simeq (R[t,t^{-1}])[[s]]$.
	 \end{remark}
	 
	 The discussion on $\infcohnewfunctor$ (\ref{c4.13}-\ref{rk4.19}) can be reproduced in the context of ordinary derived algebras, instead of pro-coherent derived algebras. There is a notion of \textit{ordinary} Hodge-filtered infinitesimal cohomology $\infcoho$ as the left adjoint of $\gr_0:\dalg(\fil_{\ge0}\m_R)\to \dalg_R$.
	 \begin{corollary}\label{corn5.10}
	 	Proposition \ref{propn5.8} induces a commuting diagram of adjunctions
	 	\[\begin{tikzcd}[column sep=1.5cm]
	 		{\dalg(\gr\m_R)} & {\dg_-\dalg_R} \\
	 		{(\mathbb{G}_m-\dst)^{op}} & {(\cH-\dst)^{op}}
	 		\arrow["\infcohgr", shift left, from=1-1, to=1-2]
	 		\arrow["\spnc"', shift right, from=1-1, to=2-1]
	 		\arrow["{\text{forget}}", shift left, from=1-2, to=1-1]
	 		\arrow["\spnc"', shift right, from=1-2, to=2-2]
	 		\arrow["{\mathcal{O}^{\sgr}}"', shift right, from=2-1, to=1-1]
	 		\arrow["{\mathcal{L}^{\sgr}_{\pi}(-/R)}", shift left, from=2-1, to=2-2]
	 		\arrow["{\mathcal{O}^{\epsilon-\sgr}}"', shift right, from=2-2, to=1-2]
	 		\arrow["{\text{forget}}", shift left, from=2-2, to=2-1]
	 	\end{tikzcd},\]where $\infcohgr$ satisfies that $\infcohgr\circ [-]_0\simeq (\infcoho)^\wedge$. In particular, $\mathcal{L}^{\sgr}_{\pi}(\text{Spec}(B)/R)$ admits a natural infinitesimal derived foliation structure $\spnc((\infcohRo)^\wedge)$, corepresented by the \textit{Hodge-completed} derived infinitesimal cohomology of $\spec(B)$.
	 \end{corollary}
	 For each $B\in\mathrm{SCR}_R$, we can define \textit{ordinary foliation-like algebras} like in Notation \ref{n4.14}, where the only difference is that they live in $\dalg(\fil_{\ge0}\m_R)_{/B}$ instead of pro-coherent derived algebras. Let $\dalg^{\fol,or}_{B/R,b}$ (resp. $\dalg^{\fol,or}_{B/R,\ap}$) be the full subcategory spanned by \textit{ordinary} foliation-like algebras $A$ with \textit{eventually connective (resp. almost perfect) cotangent complex} $\gr_1(A)[1]$. If $R$ is coherent and eventually coconnective, then the functor $\iota$ in Remark \ref{rk2.22} induces an equivalence $\dalg^{\fol,or}_{B/R,b}\simeq \dalg^{\fol}_{B/R,b}$ by Proposition \ref{prop2.38}.
	 \begin{theorem}\label{thmn5.11}
	 	Let $B$ be a simplicial commutative ring over $R$. Taking non-connective spectra induces a functor sending foliation-like algebras to infinitesimal derived foliations,  which respects the cotangent complexes
	 \[\begin{tikzcd}
	 	{(\dalg^{\fol,or}_{B/R,b})^{op}} & {\fol^{\pi}(\spec(B)/R)} \\
	 	& {\m^{op}_{B,+}}
	 	\arrow["\spnc", from=1-1, to=1-2]
	 	\arrow["{\dL_{-}}"', from=1-1, to=2-2]
	 	\arrow["{\dL_{-}}", from=1-2, to=2-2]
	 \end{tikzcd}.\]
	 	If $B$ is further assumed to be coherent,\ then for each $A\in \dalg^{\fol,or}_{B/R,\ap}$, the unit map $A\to \mathcal{O}^{\epsilon-\sgr}\circ \spnc(A)$ is an equivalence. Furthermore, there is a categorical equivalence\[ \spnc: {(\dalg^{\fol,or}_{B/R,\ap})^{op}} \xrightarrow{\simeq}{\fol^{\pi}_{\ap}(\spec(B)/R)}.\]
	 \end{theorem}
	 
	 \begin{proof}
	 	
	 	Firstly, Corollary \ref{corn5.10} provides the following adjunctions of undercategories
	 	\[\begin{tikzcd}[column sep=1.5cm]
	 		{\dalg(\gr\m_R)_{B/}} & {(\dg_-\dalg_R)_{(\infcohRo)^\wedge/}} \\
	 		{(\mathbb{G}_m-\dst)^{op}_{\spec(B)/}} & {(\cH-\dst)^{op}_{\mathcal{L}^{\sgr}_{\pi}(\spec(B)/R)/}}
	 		\arrow["\infcohgr", shift left, from=1-1, to=1-2]
	 		\arrow["\spnc"', shift right, from=1-1, to=2-1]
	 		\arrow["{\text{forget}}", shift left, from=1-2, to=1-1]
	 		\arrow["\spnc"', shift right, from=1-2, to=2-2]
	 		\arrow["{\mathcal{O}^{\sgr}}"', shift right, from=2-1, to=1-1]
	 		\arrow["{\mathcal{L}^{\sgr}_{\pi}(-/R)}", shift left, from=2-1, to=2-2]
	 		\arrow["{\mathcal{O}^{\epsilon-\sgr}}"', shift right, from=2-2, to=1-2]
	 		\arrow["{\text{forget}}", shift left, from=2-2, to=2-1]
	 	\end{tikzcd},\]where the left vertical adjunction is equivalent to $\dalg(\gr\m_B)\rightleftarrows (\mathbb{G}_m-\mathbf{dSt}_{B})^{op}$.
	 	
	 	On the other hand, the two ends of the promised realization functor can be fully faithfully embedded into the right vertices: This holds  for $\fol^{\pi}(\spec(B)/R)$ by definition. As for $\dalg^{\fol,or}_{B/R,b}$, note that $(-)^\wedge\circ\infcoho\dashv \gr_0$ is a colocalization and use Remark \ref{rk4.19}. Besides, since the forgetful functors commutes with the $\spnc$'s, $\spnc|_{\dalg^{\fol,or}_{B/R,b}}$ takes value in $\fol^{\pi}(\spec(B)/R)$ and respects the cotangent complexes.
	 	
	 	Next, assuming that $B$ is coherent, we demonstrate that $A\simeq \mathcal{O}^{\epsilon-\sgr}\circ\spnc(A)$ for $A\in \dalg^{\fol,or}_{B/R,\ap}$. According to Corollary \ref{corn5.10}, it is sufficient to show that $\text{forget}(A)\simeq \mathcal{O}^{\sgr}\circ\spnc(\text{forget}(A))$, which is guaranteed by Lemma \ref{ff for gr stack} as below. Conversely, for each $\mathscr{F}\in \fol^{\pi}_{\ap}(\spec(B)/R)$, $\mathcal{O}^{\sgr}(\mathscr{F})$ is equivalent to $\lsym_B [\dL_{\mathscr{F}}[-1]]_{1}$, thus $\mathcal{O}^{\epsilon-\sgr}(\mathscr{F})$ is an almost perfect foliation-like algebra. The adjunction map $\mathscr{F}\to \spnc\circ \mathcal{O}^{\epsilon-\sgr}(\mathscr{F}) $ is then an equivalence, since the underlying map of graded stacks is, which implies that $\mathscr{F}$ lies in the essential image of $\dalg^{\fol,or}_{B/R,\ap}$ under $\spnc$.
	 	
	 \end{proof}
	 \begin{lemma}\label{ff for gr stack}
	 	Let $B$ be a coherent simplicial commutative ring, and $E$ be an almost perfect $B$-module. Set $C:=\lsym_B [E]_1$ the freely generated graded algebra, then the unit map $\eta:C\xrightarrow{\simeq}\mathcal{O}^{\sgr}\circ \spnc(C)$ is an equivalence. 
	 \end{lemma}
	 \begin{proof}
	 	This lemma is inspired by \cite{Mou}. Observe that there exists some $n\in\mathbb{N}$ such that $E\in \aperf_{B,\ge -n}$. We do induction on $n$. The case $n=0$ is done in the proof of Proposition \ref{propn5.8}. When $n>0$, choose a map $s:B[-n]^{\oplus m}\to E$ such that $\pi_{-n}(s)$ is surjective. Write $B[-n]^{\oplus m}$ as $P[-1]$ and $\text{cofib}(s)$ as $F$. Now we consider the \v{C}ech conerve $F^\bullet$ of $E\to F$, where $F^k\simeq F\oplus P^{\oplus k}$ belongs to $\aperf_{B,\ge -(n-1)}$.
	 	
	 	Since $\spnc$ takes small colimits to limits, the simplicial object $\mathbb{V}(F^{\bullet})$ is identified with the \v{C}ech nerve of $\mathbb{V}(F)\to \mathbb{V}(E)$. Then $q:|\mathbb{V}(F^{\bullet})|\to \mathbb{V}(E)$ is $(-1)$-truncated\cite[Proposition 6.2.3.4]{HTT}. Lemma 2.7 in \cite{Monier} shows that $q$ is also an epimorphism and then an equivalence. Let's look at the following commutative square
	 	\[\begin{tikzcd}
	 		C & {\tot(\lsym_B F^\bullet)} \\
	 		{\mathcal{O}^{\sgr}\circ \spnc(C)} & {\tot(\mathcal{O}^{\sgr}\circ \mathbb{V}(F^\bullet))}
	 		\arrow["\simeq", from=1-1, to=1-2]
	 		\arrow["\eta"', from=1-1, to=2-1]
	 		\arrow["\simeq", from=1-2, to=2-2]
	 		\arrow["{\mathcal{O}^{\sgr}(q)}"', from=2-1, to=2-2]
	 	\end{tikzcd},\]where the top arrow is an equivalence by Corollary \ref{totalg}, the right one is an equivalence by induction hypothesis. In conclusion, $\eta$ is an equivalence as well.
	 \end{proof}
	 Combining Theorems \ref{thm4.25} and \ref{thmn5.11}, we obtain the main comparison theorem between infinitesimal derived foliations and partition Lie algebroids on derived schemes. For a global definition of partition Lie algebroids, see \cite[\S4]{BMN}, or refer to Remark \ref{rk4.12} for a brief recollection.
	 \begin{theorem}[\textbf{Main theorem}]\label{thmn5.13}
	 	Let $X$ be a locally coherent qcqs derived scheme over an eventually coconnective coherent simplicial commutative ring $R$, and assume that $\dL_{X/R}$ is an almost perfect quasi-coherent sheaf. There is a categorical equivalence 
	 	\[\aplagdX\xrightarrow{\simeq}\fol^{\pi}_{\ap}(X/R)\]bewteen \textnormal{dually almost perfect} partition Lie algebroids and \textnormal{almost perfect} infinitesimal derived foliations.
	 \end{theorem}
	 \begin{proof}
	 	When $X=\spec(B)$, the equivalence is given by $\spnc\circ \tc$ (Theorem \ref{thm4.25} and \ref{thmn5.11}).
	 	
	 	In the global case, we need to show that $\spnc \circ\tc$ is compatible with base change on both sides along finitely presented open immersions $\spec (B')\subset \spec (B)$. Here, we follow the construction in \cite{BMN}. 
	 	
	 	\textbf{Relative version of $\tc$:}
	 	Set $\dalg^{\Delta^1}_R:=\mathrm{Fun}(\Delta^1,\dalg_R)$, regarded as an \infcat\ over $\dalg_R$ by the projection $\mathrm{ev}_1:(A_0\to A_1)\mapsto A_1$. Let $\QC$ be the \infcat\ of all the pairs $(A,M)$, where $A$ is a derived $R$-algebra and $M$ is an $A$-module (a non-connective version of \cite[Notation 25.2.1.1]{SAG}). The apparent projection $\QC\to \dalg_R$ has a section $\dL[-1]$ given by $A\mapsto (A,\dL_{A/R}[-1])$, inducing an \infcat\ $\QC_{\dL[1]/}$ of pairs $(A,\dL_{A/R}[-1]\to M)$. Then, there is an adjunction
	 	\begin{equation}\label{rel adjunction}
	 		\cot:\dalg^{\Delta^1}_R\rightleftarrows\QC_{\dL[1]/}:\sqz
	 	\end{equation}relative to $\dalg_R$ in the sense of \cite[\S7.3.2]{HA}, where $\cot:(A_0\to A_1)\mapsto\big(A_1,\dL_{A_1/R}[-1]\to \dL_{A_1/A_0}[-1]\big)$, and $\sqz(A,\alpha:\dL_{A/R}[-1]\to M)$ has the underlying arrow in $R$-modules as $A\oplus_{\alpha}M\to A$\footnote{Recall that $d_{DR}:A\to \dL_{A/R}$ is the universal differential. The extension module $A\oplus_{\alpha}M$ is $\mathrm{cofib}(\alpha\circ d_{DR}:A[-1]\to M)$.}.\ Theorem 3.29 in \cite{BMN} states that $\dalg_{R//B}\rightleftarrows(\m_{B})_{\dL_{B/R}[-1]/}$, the fibre of (\ref{rel adjunction}) over $B\in \scr_R$, restricts to a comonadic adjunction of full subcategories
	 	\[\cot:\dalg^{\wedge\afp}_{R//B,\ge0}\rightleftarrows(\aperf_{B,\ge0})_{\dL_{B/R}[-1]/}:\sqz\]where $\dalg^{\wedge\afp}_{R//B,\ge0}$ is defined in \cite[Notation 3.28]{BMN}(but we only need to know that it lies in $\scr^{\Delta^1}_R\subset \dalg^{\Delta^1}_R$).
	 	
	 	Now, let $\mathcal{M}$ be the \infcat\ of all the pairs $(B,M)$ with $B\in \mathrm{SCR}^{\scoh,\laft}_R$ (Remark \ref{rk4.12}) and $M\in (\aperf_{B,\ge0})_{\dL_{B/R}[-1]}$, and $\mathcal{A}\subset\dalg^{\Delta^1}_R$ denote the full subcategory of all $A\to B$ such that $B\in \mathrm{SCR}^{\scoh,\laft}_R$ and $A\in \dalg^{\wedge\afp}_{R//B,\ge0}$. Then (\ref{rel adjunction}) restricts to a fibrewise comonadic adjoint pair
	 	\begin{equation}\label{com adjunction}
	 		\cot:\mathcal{A}\rightleftarrows\mathcal{M}:\sqz
	 	\end{equation}relative to $\scr^{\scoh,\laft}_R$.
	 	
	 	The restriction of the relative monad $\mathrm{LieAlgd}^{\pi}_{\Delta}$ (Remark \ref{rk4.12}) on $\mathcal{M}^{op}(\subset\QC^\vee_{/\dT[1]})$ is given by $(\cot\circ\sqz(-)^\vee)^\vee$. Furthermore, the fibrewise comonadicity of (\ref{com adjunction}) induces an equivalence $C^*:\mathrm{LieAlgd}^{\pi,\dap}_{/R,\Delta,\eqslantless0}\simeq\mathcal{A}^{op}$ via the Chevalley-Eilenberg complex on each fibre, where $\mathrm{LieAlgd}^{\pi,\dap}_{/R,\Delta,\eqslantless0}\subset \mathrm{LieAlgd}^{\pi}_{/R,\Delta}$ is spanned by all $(B,L)$ such that the underlying module of $L$ lies in $\aperf^\vee_{B,\eqslantless0}$.

	 	Now, consider $\dalg^{\fil}_R:=\dalg(\fil_{\ge0}\m_R)$, which projects to $\dalg_R$ by $\gr_0$.
	 	There is a functor $\infcohnewfunctor^{or}:\dalg^{\Delta^1}_R\to\dalg^{\fil}_R$ relative to $\scr_R$ parallel to Construction \ref{c4.13} but in the ordinary algebraic context. Then, the reasoning of Lemma \ref{lemma4.16} induces a functor relative to $\scr^{\scoh,\laft}$ \[(\infcohnewfunctor)^{\wedge}:\mathcal{A}\to\dalg^{\fol,or}_{/R}\times_{\scr_R}\scr^{\scoh,\laft}\hookrightarrow \dalg^{\fol}_{/R}\times_{\scr^{\scoh}_R}\scr^{\scoh,\laft}_R,\]and the right-hand side will be abbreviated as $\dalg^{\fol}_{\laft/R}$. It induces a functor relative to $\scr^{\scoh,\laft,op}_R$
	 	\begin{equation*}
	 		F^0:\mathrm{LieAlgd}^{\pi,\dap}_{/R,\Delta,\eqslantless0}\xrightarrow{C^*,\simeq}\mathcal{A}^{op}\xrightarrow{(\infcohnewfunctor)^{\wedge}}\dalg^{\fol,op}_{\laft/R}
	 	\end{equation*}whose restriction to the fibre over $B$ agrees with $\tc|_{\mathrm{LieAlgd}^{\pi,\dap}_{B/R,\Delta,\eqslantless0}}$ (Definition \ref{df4.24}) by Proposition \ref{prop4.20old}. As Lemma \ref{lemma4.10} is functorial in $\mathscr{A}$, $F^0$ extends to a fibrewise sifted-colimit-preserving functor
	 	\[\tc:\mathrm{LieAlgd}^{\pi}_{/R,\Delta}\to \dalg^{\fol,op}_{\laft/R},\]which agrees with $\tc$ on each fibre. Moreover, Theorem \ref{thm4.25} shows that it restricts to an equivalence $\mathrm{LieAlgd}^{\pi,\dap}_{/R,\Delta}\simeq \dalg^{\fol,op}_{\laft/R,\ap}$.
	 	
	 	\textbf{Relative version of $\spnc$:}
	 	Following Proposition \ref{propn5.8}, there is
	 	\[\cH\text{-}\mathbf{dSt}_R\xrightarrow{\rmfor} \mathbb{G}_m\text{-}\mathbf{dSt}_R\xrightarrow{\mathcal{O}^{\sgr}}\dalg(\gr\m_R)^{op}\xrightarrow{\gr_0}\dalg^{op}_R.\]Then, we take the fibre product of \infcats
	 	\[\mathcal{S}t^{\sgr}:=\mathbb{G}_m\text{-}\mathbf{dSt}_R\times_{\dalg^{op}_R}\scr^{\scoh,\laft,op}_R,\ \ \ \ \mathcal{S}t^{\epsilon\textit{-}\sgr}:=\cH\text{-}\mathbf{dSt}_R\times_{\dalg^{op}_R}\scr^{\scoh,\laft,op}_R.\]Corollary \ref{corn5.10} gives rise to categorical sections $\spec(-):\scr^{\scoh,\laft,op}_R\to\mathcal{S}t^{\sgr}$ and $\mathcal{L}^{\pi}(-/R):\scr^{\scoh,\laft,op}_R\to\mathcal{S}t^{\epsilon\textit{-}\sgr}$ as well as an adjunction
	 	$(\mathcal{S}t^{\epsilon\textit{-}\sgr})_{/\mathcal{L}^{\pi}}\rightleftarrows (\mathcal{S}t^{\sgr})_{/\spec}$ relative to $\scr^{\scoh,\laft,op}_R$. Then, we can consider those $(B,\mathcal{F}\to \mathcal{L}^{\pi}(\spec(B)/R))\in(\mathcal{S}t^{\epsilon\textit{-}\sgr})_{/\mathcal{L}^{\pi}}$ such that $\mathcal{F}$ is forgotten to a linear stack over $\spec(B)$, which forms a full subcategory $\fol^{\pi}\subset\mathcal{S}t^{\epsilon\textit{-}\sgr}$. The fibre of $\fol^{\pi}$ over some $B\in \scr^{\scoh,\laft,op}$ is exactly $\fol^{\pi}(\spec(B)/R)$.
	 	
	 	Theorem \ref{thmn5.11} shows that
	 	\[\dalg^{\fol,op}_{\laft/R,\ap}\xrightarrow{\spnc}\cH\text{-}\mathbf{dSt}_R\xrightarrow{\rmfor} \mathbb{G}_m\text{-}\mathbf{dSt}_R\xrightarrow{\mathcal{O}^{\sgr}}\dalg(\gr\m_R)^{op}\] is equivalent to $\gr$, so $\rmfor\circ\spnc$ lifts to a relative functor $\dalg^{\fol,op}_{\laft/R,\ap}\to(\mathcal{S}t^{\sgr})_{/\spec}$ whose image contains only linear stacks. Meanwhile, every fibre of $\dalg^{\fol,op}_{\laft/R,\ap}$ has a final object in the form of $(B,(\infcohR)^\wedge)$. Therefore, $\spnc$ lifts to a functor $\spnc:\dalg^{\fol,op}_{\laft/R,\ap}\to \fol^{\pi}$ relative to $\scr^{\scoh,\laft,op}_R$ fibrewise agreeing with Theorem \ref{thmn5.11}, and its essential image $\fol^{\pi}_{\ap}$ consists of all $(B,\mathcal{F})$ such that $\dL_{\mathcal{F}}$ is an almost perfect $B$-module.
	 	
	 	\textbf{Base change structure of $\spnc\circ\tc$:} In summary, we have constructed an equivalence
	 	\[\spnc\circ\tc:\mathrm{LieAlgd}^{\pi,\dap}_{/R,\Delta}\simeq \fol^{\pi}_{\ap}\]of \infcats\ over $\scr^{\scoh,\laft,op}_R$. The left-hand side is a cartesian fibration, and $B\mapsto \aplagd$ has \'etale descent \cite[Proposition 4.18]{BMN}. In particular, the cartesian arrows are in the form of $(B',\dT_{B'/R}[1]\times_{\dT_{B/R}[1]\otimes_B B'}L\otimes_B B')\to (B,L)$. The right-hand side also has sufficient cartesian edges. For each $\spec(B')\to \spec(B)$, the pullback of $\fol^{\pi}_{\ap}$ sending $\mathcal{F}\in \fol^{\pi}_{\ap}(\spec(B)/R)$ to the fibre product of $\cH$-stacks $\mathcal{F}\times_{\mathcal{L}^\pi(\spec(B)/R)}\mathcal{L}^\pi(\spec(B')/R)$. Now, we see that $\spnc\circ \tc$ respects the cartesian edges, i.e. $\spnc\circ\tc$ is compatible with base change. Additionally, $B\mapsto \fol^{\pi}_{\ap}(\spec(B)/R)$ has \'etale descent by considering the cotangent complex. Thus, for each locally coherent qcqs derived scheme $X$, there is an equivalence $\aplagdX\simeq \fol^{\pi}(X/R)$ by gluing up the equivalences on the affine charts.

	 \end{proof}

	\bibliographystyle{alpha}
	\bibliography{theme}
	\vspace{2em}
	\begin{flushleft}
	\Large\textsc{Université Paul Sabatier\\
	Institut de Mathématiques de Toulouse\\
	118, route de Narbonne\\
	F-31062 Toulouse Cedex 9\vspace{1em}\\}
	\Large\textit{E-mail}: jiaqi.fu@univ-tlse3.fr\\
	\Large\textit{Homepage}: \url{https://jiaqifumath.github.io/DAG/}
	\end{flushleft}

\end{document}